\documentclass[numbers,webpdf,imaiai]{ima-authoring-template}

\usepackage{tikz}

\usepackage{mathrsfs}
\usepackage{amsmath}
\usepackage{amssymb}
\usepackage{amsthm}
\usepackage{mathrsfs}
\usepackage{xspace}
\usepackage{bm}

\usepackage{calrsfs}
\DeclareMathAlphabet{\pazocal}{OMS}{zplm}{m}{n}
\SetMathAlphabet\pazocal{bold}{OMS}{zplm}{bx}{n}
\usepackage[Symbol]{upgreek}
\usepackage{commath}
\usepackage{mleftright}
\definecolor{tblblack}{rgb}{0.93,0.93,1.0}
\definecolor{tblred}{rgb}{1,0.93,0.93}
\definecolor{darkblack}{rgb}{0,0,0.7} 
\definecolor{darkgreen}{RGB}{20,120,43} 
\definecolor{darkred}{rgb}{0.8,0,0} 
\definecolor{lightblack}{RGB}{101,124,191}
\definecolor{skyblack}{RGB}{135,206,235}
\definecolor{gold}{RGB}{204,168,66}
\definecolor{strongblack}{RGB}{60,146,228}
\definecolor{lightgray}{gray}{0.5}
\definecolor{verylightgray}{RGB}{101,124,191}
\definecolor{mistyrose}{RGB}{238,213,210}
\definecolor{firebrick3}{RGB}{205,38,38}
\usetikzlibrary{lindenmayersystems}
\pgfdeclarelindenmayersystem{cantor set}{
  \rule{F -> FfF}
  \rule{f -> fff}
}
\usetikzlibrary{arrows.meta,arrows}

\newcommand{\safemath}[2]{\newcommand{#1}{\ensuremath{#2}\xspace}}

\newcommand{\subc}{c}
\newcommand{\supc}{b}
\safemath{\setA}{\pazocal{A}}
\safemath{\setB}{\pazocal{B}}
\safemath{\setC}{\pazocal{C}}
\safemath{\setD}{\pazocal{D}}
\safemath{\setE}{\pazocal{E}}
\safemath{\setF}{\pazocal{F}}
\safemath{\setG}{\pazocal{G}}
\safemath{\setH}{\pazocal{H}}
\safemath{\setI}{\pazocal{I}}
\safemath{\setJ}{\pazocal{J}}
\safemath{\setK}{\pazocal{K}}
\safemath{\setL}{\pazocal{L}}
\safemath{\setM}{\pazocal{M}}
\safemath{\setN}{\pazocal{N}}
\safemath{\setO}{\pazocal{O}}
\safemath{\setP}{\pazocal{P}}
\safemath{\setQ}{\pazocal{Q}}
\safemath{\setR}{\pazocal{R}}
\safemath{\setS}{\pazocal{S}}
\safemath{\setT}{\pazocal{T}}
\safemath{\setU}{\pazocal{U}}
\safemath{\setV}{\pazocal{V}}
\safemath{\setW}{\pazocal{W}}
\safemath{\setX}{\pazocal{X}}
\safemath{\setY}{\pazocal{Y}}
\safemath{\setZ}{\pazocal{Z}}
\safemath{\emptySet}{\varnothing}

\safemath{\colA}{\mathscr{A}}
\safemath{\colB}{\mathscr{B}}
\safemath{\colC}{\mathscr{C}}
\safemath{\colD}{\mathscr{D}}
\safemath{\colE}{\mathscr{E}}
\safemath{\colF}{\mathscr{F}}
\safemath{\colG}{\mathscr{G}}
\safemath{\colH}{\mathscr{H}}
\safemath{\colI}{\mathscr{I}}
\safemath{\colJ}{\mathscr{J}}
\safemath{\colK}{\mathscr{K}}
\safemath{\colL}{\mathscr{L}}
\safemath{\colM}{\mathscr{M}}
\safemath{\colN}{\mathscr{N}}
\safemath{\colO}{\mathscr{O}}
\safemath{\colP}{\mathscr{P}}
\safemath{\colQ}{\mathscr{Q}}
\safemath{\colR}{\mathscr{R}}
\safemath{\colS}{\mathscr{S}}
\safemath{\colT}{\mathscr{T}}
\safemath{\colU}{\mathscr{U}}
\safemath{\colV}{\mathscr{V}}
\safemath{\colW}{\mathscr{W}}
\safemath{\colX}{\mathscr{X}}
\safemath{\colY}{\mathscr{Y}}
\safemath{\colZ}{\mathscr{Z}}

\safemath{\reals}{\mathbb R}
\safemath{\positivereals}{\reals^{+}}
\safemath{\integers}{\mathbb Z}
\safemath{\posint}{\integers^{+}}
\safemath{\naturals}{\mathbb N}
\safemath{\posnaturals}{\naturals^{+}}
\safemath{\complexset}{\mathbb C}
\safemath{\rationals}{\mathbb Q}

\safemath{\opA}{\mathbb{A}}
\safemath{\opB}{\mathbb{B}}
\safemath{\opC}{\mathbb{C}}
\safemath{\opD}{\mathbb{D}}
\safemath{\opE}{\mathbb{E}}
\safemath{\opF}{\mathbb{F}}
\safemath{\opG}{\mathbb{G}}
\safemath{\opH}{\mathbb{H}}
\safemath{\opI}{\mathbb{I}}
\safemath{\opJ}{\mathbb{J}}
\safemath{\opK}{\mathbb{K}}
\safemath{\opL}{\mathbb{L}}
\safemath{\opM}{\mathbb{M}}
\safemath{\opN}{\mathbb{N}}
\safemath{\opO}{\mathbb{O}}
\safemath{\opP}{\mathbb{P}}
\safemath{\opQ}{\mathbb{Q}}
\safemath{\opR}{\mathbb{R}}
\safemath{\opS}{\mathbb{S}}
\safemath{\opT}{\mathbb{T}}
\safemath{\opU}{\mathbb{U}}
\safemath{\opV}{\mathbb{V}}
\safemath{\opW}{\mathbb{W}}
\safemath{\opX}{\mathbb{X}}
\safemath{\opY}{\mathbb{Y}}
\safemath{\opZ}{\mathbb{Z}}
\safemath{\opZero}{\mathbb{O}}
\safemath{\identityop}{\opI}

\DeclareMathOperator{\sinc}{sinc}
\newcommand{\ind}[1]{\chi_{#1}}	

\newcommand{\tp}[1]{\ensuremath{#1^{\mathsf{T}}}} 

\theoremstyle{thmstyletwo}
\newtheorem{thm}{Theorem}[section]
\newtheorem{lem}{Lemma}[section]
\newtheorem{prp}{Proposition}[section]
\newtheorem{cor}{Corollary}[section]
\theoremstyle{definition}

\theoremstyle{remark}

\newtheorem{exa}{Example}[section]
\newtheorem{dfn}{Definition}[section]

\numberwithin{equation}{section}
\allowdisplaybreaks[4]

\begin{document}
\renewcommand{\labelenumi}{\theenumi}

\DOI{DOI HERE}
\copyrightyear{2022}
\vol{00}
\pubyear{2022}
\access{Advance Access Publication Date: Day Month Year}
\appnotes{Paper}
\copyrightstatement{Published by Oxford University Press on behalf of the Institute of Mathematics and its Applications. All rights reserved.}
\firstpage{1}

\title[Lossy Compression of  General Random Variables]{Lossy Compression of  General Random Variables}

\author{Erwin Riegler*
\address{\orgdiv{Chair for Mathematical Information Science}, \orgname{ETH Zurich}, \orgaddress{\street{Sternwartstrasse 7}, \postcode{8092}, \state{Zurich}, \country{Switzerland}}}}
\author{G\"unther Koliander
\address{\orgdiv{Acoustics Research Institute}, \orgname{Austrian Academy of Sciences}, \orgaddress{\street{Wohllebengasse 12-14}, \postcode{1040}, \state{Vienna}, \country{Austria}}}}
\author{Helmut B\"olcskei
\address{\orgdiv{Chair for Mathematical Information Science}, \orgname{ETH Zurich}, \orgaddress{\street{Sternwartstrasse 7}, \postcode{8092}, \state{Zurich}, \country{Switzerland}}}}
\authormark{E. Riegler et al.}
\corresp[*]{Corresponding author: \href{email:eriegler@mins.ee.ethz.ch}{eriegler@mins.ee.ethz.ch}}

\received{14}{03}{2022}
\accepted{06}{09}{2022}



\abstract{
This paper is concerned with the lossy compression of general random variables, specifically with rate-distortion theory and quantization of random variables taking values in general measurable spaces such as, e.g., manifolds and fractal sets. Manifold structures are prevalent in data science, e.g., in compressed sensing, machine learning, image processing, and handwritten digit recognition. Fractal sets find application in image compression and in the modeling of Ethernet traffic. Our main contributions are bounds on the rate-distortion function and the quantization error. These bounds are very general and essentially only require the existence of reference measures satisfying certain regularity conditions in terms of small ball probabilities. To illustrate the wide applicability of our results, we particularize them to random variables taking values in i) manifolds, namely, hyperspheres and Grassmannians, and ii) self-similar sets characterized by  iterated function systems satisfying the weak separation property.}




\maketitle

\section{Introduction} 
This paper is concerned with the lossy compression of general random variables, specifically with   
rate-distortion (R-D) theory  and  quantization of  random variables  taking  values in  general measurable spaces such as, e.g.,  manifolds and fractal sets. 
Manifold structures are prevalent in data science, e.g., in compressed sensing \cite{bawa09,care09,capl11,albdekori18,ristbo15,cose06}, machine learning \cite{femina16,lz12}, image processing \cite{lufahe98,soze98}, directional statistics  \cite{maju00}, and handwritten digit recognition \cite{hidare97}. 
Fractal sets find application in image compression and in modeling of Ethernet traffic \cite{letawi94}.

In the following, we formally introduce the mathematical setup underlying  R-D theory and  quantization, review the corresponding relevant results in the literature, and summarize our main contributions.   
Quantities not defined directly are introduced in the notation paragraph at the end of this section. 

\subsection{Rate-Distortion Theory}\label{sec:RDIntro}
In R-D theory \cite{sh59,be71,gr90,grne98,deko00}, one is interested in the characterization of the ultimate limits on the compression of sequences of random variables under a distortion constraint, here expressed in terms of expected average distortion.  
Specifically, let  $(\setA,\colA)$ and $(\setB,\colB)$ be  measurable spaces 
equipped with a measurable function  $\sigma\colon \setA\times\setB\to[0,\infty]$, henceforth referred to as   distortion function,\footnote{For $\setA=\setB=\reals^d$, a distortion function $\sigma\colon \reals^d\times\reals^d\to [0,\infty]$ is called  a difference distortion function if  there exists a measurable function $f\colon \reals^d\to [0,\infty]$ such that $\sigma(x,y)=f(x-y)$ for all $x,y\in\reals^d$.} 
 and let $(X_i)_{i\in\naturals}$ be a 
sequence of random variables taking values in $(\setA,\colA)$.     
For every $\ell\in\naturals$,  measurable mappings $g_{(\ell)}\colon \setA^\ell \to \setB^\ell$ with  $\lvert g_{(\ell)}(\setA^\ell) \rvert <\infty$ are referred to as source codes of length $\ell$.  
For a given rate $R\in[0,\infty)$ and distortion $D\in[0,\infty)$, 
the  pair $(R,D)$ is said to be $\ell$-achievable if
there exists a source code $g_{(\ell)}$ of length $\ell$ with $\abs{g_{(\ell)}(\setA^\ell)} \leq \lfloor e^{\ell R}\rfloor$ and  expected average distortion  
\begin{align}\label{eq:leqdfin}
\opE\mleft[\sigma_{(\ell)}\big((X_1,\dots,X_\ell),g_{(\ell)}(X_1,\dots,X_\ell)\big)\mright] \leq D, 
\end{align} 
where 
\begin{align}\label{eq:ghol}
\sigma_{(\ell)}((x_1,\dots,x_\ell),(y_1,\dots,y_\ell))=\frac{1}{\ell}\sum_{i=1}^\ell\sigma(x_i,y_i) 
\end{align}
is the average distortion function  of length $\ell$. 
Moreover, the  pair $(R,D)$  is said to be achievable 
if  there exists an $\ell\in\naturals$ such that  $(R,D)$  is $\ell$-achievable.  
For every $\ell\in\naturals$ and $D\in[0,\infty)$, 
we set\footnote{We use the convention that the infimum of the empty set is $\infty$.} 
\begin{equation}
R_{(\ell)}(D) = \inf\{R\in[0,\infty):  \text{$(R,D)$ is $\ell$-achievable}\}
\label{eq:rdfin}
\end{equation} 
and 
\begin{equation}
R(D) = \inf\{R\in[0,\infty):  \text{$(R,D)$ is achievable}\}.  
\label{eq:rdfin2}
\end{equation}

The set of achievable $(R,D)$-pairs can be characterized as follows. 
For general measurable spaces  $(\setX,\colX)$ and $(\setY,\colY)$  equipped with a general distortion function  $\rho\colon \setX\times\setY\to[0,\infty]$ and the  random variable  $X$ taking  values in $(\setX,\colX)$,   
the R-D function  is defined 
as 
\begin{align}\label{eq:RD}
R_{X}(D)=\inf_{\mu_{Y\mid X}:\,\opE[\rho(X,Y)]\,\leq\, D} I(X;Y),   
\end{align}
where $I(\cdot\,\,;\,\cdot)$ denotes the mutual information (defined in \eqref{eq:I})  between $X$ and $Y$ with $Y$ taking values in $(\setY,\colY)$  
and the infimum is over all  $\mu_{Y\vert  X}\colon \colY\times\setX\to[0,1]$ with  (i) $\mu_{Y\vert X}(\cdot,x)$  a probability measure on $(\setY,\colY)$ for all $x\in\setX$, (ii)  $\mu_{Y\mid X}(\setC,\cdot)$  a measurable function for all $\setC\in\colY$,  and (iii)
\begin{equation}
\opE[\rho(X,Y)]= \opE\mleft[ \int \rho(X,y) \,\mathrm d \mu _{Y\vert X}(y,X) \mright].   
\end{equation}  

Now, let $(\setA,\colA)$, $(\setB,\colB)$, and  $\sigma\colon \setA\times\setB\to[0,\infty]$ be as defined above and consider a stationary ergodic process  $(X_i)_{i\in \naturals}$, where  $X_i$ takes values in $(\setA,\colA)$ for all $i\in\naturals$.   
For every $\ell\in\naturals$, define 
$R_{(X_1, \dots, X_\ell)}(D)$ as in \eqref{eq:RD} for $X=(X_1, \dots, X_\ell)$, $Y=(Y_1, \dots, Y_\ell)$, $(\setX,\colX)=(\setA^\ell,\colA^\ell)$, $(\setY,\colY)=(\setB^\ell,\colB^\ell)$, and 
 $\rho=\sigma_{(\ell)}$. Finally,  set 
\begin{align}\label{eq:limitR}
R_{(X_i)_{i\in \naturals}}(D)=\lim_{\ell\to\infty}\frac{1}{\ell} 
R_{(X_1, \dots, X_\ell)}(D). 
\end{align} 
The  limit in \eqref{eq:limitR} is well-defined by \cite[Problem 4.17]{be71}. 
Then, we have  the following achievability result with  matching converse: 
\renewcommand{\theenumi}{(\roman{enumi})}
\begin{enumerate}[(ii)]
\item 
Suppose that  there exists a  $b^*\in\setB$ with  
$\opE\mleft[\sigma(X_1,b^*)\mright]<\infty$. Then, for every $D\in[0,\infty)$ with $R_{(X_i)_{i\in \naturals}}(D)<\infty$, $(R,D)$ is achievable for all $R>R_{(X_i)_{i\in \naturals}}(D)$ 
 \cite[Theorem 7.2.4]{be71}. 
\item \label{converse} 
For every $D\in[0,\infty)$ and  $R<R_{(X_i)_{i\in \naturals}}(D)$, 
$(R,D)$ is not achievable   \cite[Theorem 7.2.5]{be71}.
\end{enumerate}

Characterizing $R_{(X_i)_{i\in\naturals}}(D)$ analytically is, in general, very difficult, even in the case of  i.i.d. sequences  $(X_i)_{i\in\naturals}$, 
where \eqref{eq:limitR} reduces to the single-letter expression  $R_{(X_i)_{i\in \naturals}}(D)=R_{X_1}(D)$.   
For general   stationary ergodic processes $(X_i)_{i\in\naturals}$, computation of $R_{(X_i)_{i\in \naturals}}(D)$  even requires knowledge of $R_{(X_1, \dots, X_\ell)}(D)$  for all $\ell\in\naturals$ with  possibly the exception of finitely many $\ell$. 
One therefore resorts to  bounds on $R_X(D)$ 
for general $X$   (e.g., $X=X_1$ for i.i.d. sequences or $X=(X_1, \dots, X_\ell)$ for general stationary ergodic processes).  
While upper bounds can be obtained by evaluating  $I(X;Y)$ for a specific $Y$  satisfying  $\opE[\rho(X,Y)]\leq D$, 
lower bounds  are notoriously hard to obtain. 
The best-known lower bounds are the Shannon lower bounds for i) discrete random variables  
with $\rho$ such that, for every $s\in(0,\infty)$,  
$\sum_{x\in\setX}e^{-s\rho(x,y)}$ does not depend on $y$ \cite[Lemma  4.3.1]{gr90}   
and ii)   continuous random variables 
with a difference distortion function \cite[Equation (4.6.1)]{gr90}.  
For continuous $X$ of finite differential entropy  under  the difference distortion function $\rho(x,y)=\lVert x-y\rVert^k$, where $\lVert\,\cdot\,\rVert$ is a semi-norm and $k\in(0,\infty)$, the Shannon lower bound is known explicitly \cite[Section VI]{yatagr80} and, 
provided that $X$ additionally  satisfies a certain moment constraint, is tight as $D\to 0$ \cite{liza94,ko16}. 
For the class of  $m$-rectifiable random variables \cite[Definition 11]{kopirihl16}, a Shannon lower  bound   was reported recently in \cite[Theorem 55]{kopirihl16}.    
This bound  is, however, not in explicit form and depends on a parametrized integral.  
 
 Asymptotic results on the R-D function 
 for  $(\setX,\omega)$  a metric space, $\setY=\setX$, and $\rho(x,y)=\omega^k(x,y)$ with $k\in(0,\infty)$ were reported in  \cite{kade94}  in terms of the lower and upper  R-D dimensions of order $k$ defined as 
\begin{align}
\underline{\dim}_\text{R}(X)=\liminf_{D\to 0}\frac{ R_{X}(D^k)}{\log (1/D)}
\end{align}
and 
\begin{align}
\overline{\dim}_\text{R}(X)=\limsup_{D\to 0} \frac{R_{X}(D^k)}{\log (1/D)}, 
\end{align}
respectively. 
 If $\underline{\dim}_\text{R}(X)=\overline{\dim}_\text{R}(X)$, then this common value, denoted by ${\dim}_\text{R}(X)$, is referred to as the R-D dimension of order $k$.  
Specifically, it is shown in  \cite{kade94} (see also \cite[Remark 14.19]{grlu00}), that for a random variable $X$ of  self-similar distribution  $\mu_X$ satisfying  the assumptions of 
 \cite[Theorem 14.17]{grlu00}, for every $k\in(0,\infty)$,   
 the R-D dimension of order $k$  exists and equals the Hausdorff dimension of $\mu_X$ given by 
  \begin{align}\label{eq:dimHX}
\dim_\mathrm{H}(\mu_X)=\inf\{\dim_\text{H}(\setC):\setC\subseteq\reals^d\ \text{is Borel with}\ \mu_X(\setC)=1\}.   
\end{align}

\subsection{Quantization}
In  quantization  \cite{gegr92,grlu00}, one is concerned with characterizing the ultimate limits on the discretization of  random variables. 
Specifically, let  $(\setX,\colX)$ and $(\setY,\colY)$ be measurable spaces equipped with the  distortion function  $\rho\colon \setX\times\setY\to[0,\infty]$,  and let  $X$ be a random variable taking values in $(\setX,\colX)$.     
For every $n\in\naturals$, consider the set $\setF_n(\setX,\setY)$ of all 
measurable mappings $f\colon \setX \to \setY$ with  $\lvert f(\setX)\rvert\leq n$. The elements of $\setF_n$ are  referred to as  $n$-quantizers and  
the $n$-th quantization error is defined as
\begin{equation}\label{eq:quanterr}
V_{n}(X)=\inf_{f\in\setF_n(\setX,\setY)}\opE[\rho(X,f(X))]. 
\end{equation}
The typical setup  is  $\setX\subseteq\setY=\reals^d$  with  difference distortion function $\rho(x,y)=\|x-y\|^k$, where  $\lVert\,\cdot\,\rVert$ is a norm on $\reals^d$ and $k\in[1,\infty)$ \cite[Section III]{grne98}. 
The case $d>1$   is commonly referred to as vector quantization (see \cite{wuyu19} and references therein).  
While we restrict our attention to the expected average error as defined in \eqref{eq:quanterr}, results in terms of the 
worst case error 
\begin{equation}\label{eq:quanterrworst}
\widetilde V_{n}(X)=\inf_{f\in\setF_n(\setX,\setY)}\operatorname{ess} \sup \lVert X -f(X)\rVert  
\end{equation}
 can be found in  \cite{grlu00}.

The $n$-th quantization error $V_n(X)$ is difficult to characterize  in general.  
Asymptotic results are, however,  available, specifically 
in terms of 
the   
 lower and upper {$k$-th quantization dimensions},  defined as
 \begin{align}\label{eq:Dinf}
\underline{D}_k(X)=\liminf_{n\to\infty}  \frac{ k \log (n)}{\log (1/V_{n}(X))}  
\end{align}
and 
\begin{align}\label{eq:Dsup}
\overline{D}_k(X)=\limsup_{n\to\infty} \frac{ k \log (n)}{\log (1/V_{n}(X))},  
\end{align}
respectively. If $\underline{D}_k(X)=\overline{D}_k(X)$, then this common value, denoted by $D_k(X)$, is referred to as the $k$-th quantization dimension.   
For the special case  $\setX\subseteq\setY=\reals^d$,    $\rho(x,y)=\|x-y\|^k$ with $\lVert\,\cdot\,\rVert$  a norm on $\reals^d$,  and  $k\in [1,\infty)$, 
it is known that 
$\underline{D}_k(X)\geq \max\{\dim_\text{H}(\mu_X),\underline{\dim}_\text{R}(X) \}$ \cite[Theorems 11.6 and 11.10]{grlu00} with 
 $\dim_\text{H}(\mu_X)$ as  defined  in \eqref{eq:dimHX}. 
If, in addition, $\mu_X$ is $\lVert\, \cdot\,\rVert^k$-regular of dimension $m$  (see Definition \ref{def:subreg}),  
then $D_k(X)=\dim_H(\mu_X)=m$ \cite[Theorem 12.18]{grlu00}, i.e., the $k$-th  quantization dimension   exists and equals the Hausdorff dimension of $\mu_X$.  

If the  $k$-th quantization dimension $D_k(X)$ exists in  $(0,\infty)$, a more accurate characterization  of the 
asymptotic behavior of   $V_n(X)$    can be obtained in terms of 
 the   lower and upper {$k$-th quantization coefficient}   defined as  
\begin{align}\label{eq:Cinf}
\underline{C}_k(X)=\liminf_{n\to\infty} n^\frac{k}{D_k(X)}V_n(X) 
\end{align}
and 
\begin{align}\label{eq:Csup}
\overline{C}_k(X)=\limsup_{n\to\infty} n^\frac{k}{D_k(X)}V_n(X),  
\end{align}
respectively. If $\underline{C}_k(X)=\overline{C}_k(X)$, then this common value, denoted by $C_k(X)$,  is referred to as the $k$-th quantization coefficient.  
Results on the  $k$-th quantization coefficient are  scarce and available in very specific cases only. More concretely, if  $\setX=\setY=\reals^d$  with   difference distortion function $\rho(x,y)=\|x-y\|^k$, where  $\lVert\,\cdot\,\rVert$ is a norm on $\reals^d$ and $k\in[1,\infty)$, then the following results are available:   
\begin{itemize}
\item   If the continuous part of the distribution of  $X$ is  nonzero then, 
for every $k\in[1,\infty)$, provided that there is a $\delta\in(0,\infty)$ such that   $\opE[\lVert X \rVert^{\delta+k}]<\infty$,    $C_k(X)$ exists  in $(0,\infty)$ and an analytic expression for $C_k(X)$  is available   \cite[Theorem 6.2, Remark 6.3]{grlu00}.  
\item
If the distribution of $X$ is $\rho^{1/k}$-regular of dimension $m$ (see Definition \ref{def:subreg}),    
then 
$0<\underline{C}_k(X)\leq \overline{C}_k(X)<\infty$    \cite[Theorem 12.18]{grlu00}. 
A condition for the existence of  $C_k(X)$  does not seem to be available 
\cite[Remark 12.19]{grlu00}. 
\item For  $X$ of self-similar distribution, $C_k(X)$ exists in $(0,\infty)$ provided that  the defining vector \cite[Equation (1.1)]{po04} of the self-similar distribution  is nonarithmetic\footnote{A vector $(x_1,\dots, x_d)\in\reals^d$ with $x_i > 0$ for $i=1,\dots,d$  is said to be arithmetic if there exists  a $\theta \in(0,\infty)$ such that $x_i/\theta \in \naturals$ for $i=1,\dots,d$.}
  \cite[Theorem 1]{po04}. Analytic characterizations  for  the lower and upper $k$-th quantization coefficient do not seem to be available in general,  with the exception  of  $X$ distributed  uniformly  on the middle-third Cantor set \cite[Theorem 5.2]{grlu97}.  
\end{itemize}

\subsection{Contributions} 
Most of our results, both for  R-D theory and quantization,   hold for general measurable spaces $(\setX,\colX)$ and $(\setY,\colY)$ equipped only with a distortion function  $\rho\colon \setX\times\setY\to[0,\infty]$ and do not need  a common ambient space or metric on $\setX$ or $\setY$. 
Our main contribution to R-D theory 
is   a lower bound $R_X^{\text{L}}(D)$ on  the {R-D} function $ R_X(D)$ in \eqref{eq:RD}   for random variables  $X$ taking values in the  measurable space $(\setX,\colX)$.  The only requirement for this bound to apply is  that the  
  distribution $\mu_X$ of $X$ is absolutely continuous with respect to a  
 ($\sigma$-)finite  measure $\mu$  on  $(\setX,\colX)$ of  finite  generalized entropy $h_\mu(X)$ (as defined in \eqref{eq:entropy}) and  satisfying a certain subregularity condition.    
Specifically, this subregularity condition guarantees that the measure $\mu$ is not too  concentrated   on $\rho$-balls of small radii  in the following sense:  
There exist constants $m\in[0,\infty)$, $c\in(0,\infty)$, and $\delta_0\in(0,\infty]$  such that  
\begin{align}\label{eq:subregularityI}
\mu\mleft(\setB_{\rho}\mleft(y,\delta\mright)\mright)\leq c\delta^m\quad\text{for all $y\in\setY$ and  $\delta\in (0,\delta_0)$} 
\end{align}
with the  $\rho$-ball of radius $\delta$ centered at $y$ defined as $\setB_{\rho}(y,\delta)=\{x\in\setX:\rho(x,y)<\delta\}$.  
We refer to $m$ as subregularity dimension and to  $c,\delta_0$ as  subregularity constants. 
The subregularity condition \eqref{eq:subregularityI} is satisfied, e.g.,   by the $m$-dimensional Hausdorff measure restricted to   
an arbitrary regular set $\setK$ of dimension $m$ in $\reals^d$ and with $\setY=\setK$ (see \cite[Section 12]{grlu00}). 
Specific examples of regular sets of dimension $m$ are compact convex sets $\setK\subseteq\reals^m$ with $\operatorname{span}(\setK)=\reals^m$ \cite[Example 12.7]{grlu00}, surfaces of compact convex sets $\setK\subseteq\reals^{m+1}$ with $\operatorname{span}(\setK)=\reals^{m+1}$ \cite[Example 12.8]{grlu00}, 
$m$-dimensional compact $C^1$-submanifolds of $\reals^d$ \cite[Example 12.9]{grlu00},  
self-similar sets of similarity dimension $m$ satisfying the weak separation property \cite[Theorem 2.1]{frheolro15}, and finite unions 
of regular sets of dimension $m$ \cite[Lemma 12.4]{grlu00}. 
The lower bound $R_X^{\text{L}}(D)$ we obtain allows us to conclude that 
the lower R-D dimension $\underline{\dim}_\text{R}(X)$ of order $k\in(0,\infty)$ is  lower-bounded by the subregularity dimension $m$, i.e.,  $\underline{\dim}_\text{R}(X)\geq m$.
For continuous $X$ of finite differential entropy and distortion function $\rho(x-y)=\lVert x-y\rVert^k$, 
 where  $\lVert\,\cdot\,\rVert$ is a semi-norm on $\reals^d$ and $k\in(0,\infty)$, 
 our lower bound reduces to the classical Shannon lower bound  in   \cite{yatagr80}.  

Our first main contribution to quantization  is  a lower bound  on the $n$-th quantization error  for  random variables  $X$ taking values 
 in the  measurable space $(\setX,\colX)$. 
  The only requirement for this bound to apply is  that the  
  distribution $\mu_X$ 
of $X$ is absolutely continuous with respect to a 
 $\sigma$-finite measure $\mu$    
 on $(\setX,\colX)$ satisfying the subregularity condition \eqref{eq:subregularityI} and $\lVert \mathrm d \mu_X/\mathrm d \mu\rVert_{p/(p-1)}^{(\mu)}<\infty$ with $p\in[1,\infty)$.   
Further, the lower   bound we obtain  allows us to conclude that 
\begin{align}\label{eq:upperDk}
\underline{D}_k(X)\geq m/p\quad\text{for all $k\in(0,\infty)$.}
\end{align}
 Moreover, we show that, if $D_k(X)$ exists and satisfies $D_k(X)=m/p$, then 
 \begin{align}
\underline{C}_k(X)\geq   \frac{m}{m+pk}  c^{-\frac{k}{m}} \mleft(\lVert \mathrm d \mu_X/\mathrm d \mu\rVert_{p/(p-1)}^{(\mu)}\mright)^{-\frac{pk}{m}}.
 \end{align}

Our second main contribution to quantization  is an upper bound  on the $n$-th quantization error for  random variables  $X$ taking values in the  measurable space $(\setX,\colX)$. 
This result  
requires  the existence of a finite measure\footnote{We assume, without loss of generality, that $\nu(\setY)=1$.\label{ftn:hier2}} $\nu$ on $(\setY,\colY)$ satisfying  
a certain superregularity condition. 
Specifically, this superregularity condition guarantees that the measure $\nu$ is not too  small   on $\rho$-balls of small radii  in the following sense: 
There exist constants $m\in[0,\infty)$, $b\in(0,\infty)$, and  $\delta_0\in(0,\infty]$  such that  
\begin{align}\label{eq:subregularityIa}
\nu\mleft(\widetilde\setB_{\rho}\mleft(x,\delta\mright)\mright)\geq b\delta^m\quad\text{for all $x\in\setX$ and  $\delta\in (0,\delta_0)$} 
\end{align}
with  the  $\rho$-ball of radius $\delta$ centered at $x$ defined as 
 $\widetilde\setB_{\rho}\mleft(x,\delta\mright)=\mleft\{y\in\setY:\rho(x,y)<\delta\mright\}$. We refer to $m$ as the superregularity dimension and to $b,\delta_0$ as  superregularity constants. 
As in the case of subregularity,   the $m$-dimensional Hausdorff measure  
restricted to an arbitrary   regular set $\setK$ of dimension $m$ in $\reals^d$ 
satisfies  the superregularity condition \eqref{eq:subregularityIa}  with $\setX=\setK$ (see \cite[Section 12]{grlu00}). 
 The upper bound  we obtain allows us to conclude that 
the $k$-th upper  quantization dimension  $\overline{D}_k(X)$  is  upper-bounded by the superregularity dimension $m$, i.e., 
\begin{align}\label{eq:lowerDk}
\overline{D}_k(X)\leq m\quad\text{for all $k\in(0,\infty)$.}
\end{align} 
Moreover, we show that, if $D_k(X)$ exists and satisfies $D_k(X)=m$, then   
 \begin{align}\label{eq:RQ1}
 \overline{C}_k(X)\leq \Gamma\mleft(1+\frac{m}{k}\mright)\supc^{-\frac{k}{ m}}.
 \end{align}
 

If, in addition to the assumptions pertaining to the  upper bound,  $D_k(X)=m$,  $\setX=\setY\subseteq\reals^d$ is Borel, $\rho(x,y)=\omega^k(x,y)$ for $k\in(0,m)$ with $\omega$ a metric on $\reals^d$ such that 
\begin{align}
\sup_{x,y\,\in\,\setX}\omega(x,y)<\infty,
\end{align}
 and  $\nu\ll\mu_X$ with $\lVert \mathrm d\nu/\mathrm d \mu_X\rVert^{(\mu_X)}_{\infty}<\infty$, then we get the  improved upper bound  
\begin{align}\label{eq:RQ2}
 \overline{C}_k(X)\leq \opE\mleft[\mleft( \frac{\mathrm d \nu}{\mathrm d \mu_X}\mright)^{\frac{k}{m}}(X)\mright]\Gamma\mleft(1+\frac{m}{k}\mright) \supc^{-\frac{k}{ m}}.
 \end{align} 



To illustrate the wide applicability of our results, we particularize them   to i) compact manifolds, specifically hyperspheres and Grassmannians, and ii)  self-similar sets generated  by iterated function systems satisfying the weak separation property, with   the middle third Cantor set as a specific (and prominent) example.

\emph{Notation.} 
Sets are designated by calligraphic letters, e.g., $\setA$, with $\abs{\setA}$  denoting  cardinality,    
$\overline{\setA}$  closure, and $\setA^\ell$ the $\ell$-fold cartesian product.    
$\sigma$-algebras are indicated by script letters, e.g., $\colX$,  
and are  assumed to contain all singleton sets.  
Unless stated otherwise, (subsets of) topological spaces, e.g., (subsets of) $\reals^d$, hyperspheres, or Grassmannians   are understood to be equipped with the Borel $\sigma$-algebra   corresponding to the (induced) topology. Measures defined on Borel sets are  assumed to be 
Borel measures.   
 We let $\colX\otimes\colY$ be the product $\sigma$-algebra formed by $\colX$ and $\colY$ and write $\colX^\ell$ for the $\ell$-fold  product $\sigma$-algebra corresponding to  $\colX$. 
For  $k\in[1,\infty)$,   $\lVert\,\cdot\,\rVert_k$ stands for the  $k$-norm on $\reals^d$.  
Measures are  assumed to be positive. 
For a real-valued and $\mu$-measurable function $f$  and $p\in (0,\infty]$, we let  
\begin{align}
\|f\|^{(\mu)}_{p}&=\mleft(\int\abs{f(x)}^{p}\mathrm d\mu(x)\mright)^\frac{1}{p}\quad \text{if $p<\infty$,}
\end{align}
and set $\|f\|^{(\mu)}_{\infty}=\operatorname{ess\ sup} \abs{f}$.

For a measure space $(\setX,\colX,\mu)$, we write\footnote{The restricted measure $\mu\vert_{\setA}$  is defined on $\setX$ equipped with the trace $\sigma$-algebra $\colX_\setA:=\{\setZ\cap\setA:\setZ\in\colX\}$. It induces a measure on $(\setA,\colX_\setA)$ as well, for which we shall also use the symbol $\mu\vert_{\setA}$.} $\mu\vert_{\setA}$ for the restriction of $\mu$ to $\setA \in \colX$ and 
$f_\ast(\mu)$ for the pushforward measure corresponding to the  measurable mapping $f$ from $(\setX,\colX,\mu)$ into any measurable space  $(\setY,\colY)$. 
Throughout, random variables are indicated by  capital letters, e.g., $X$,   and we write $\mu_X$ for the corresponding distribution.  The support of the Borel measure $\mu$ is  denoted by   $\operatorname{supp}(\mu)$ \cite[Definition 1.64]{amfupa00}. 
For  measures $\mu$ and $\nu$,   $\mu\otimes\nu$ designates the corresponding product measure. 
For $\mu$ and $\sigma$-finite $\nu$ defined on the same measurable space with 
$\mu$ absolutely  continuous with respect to $\nu$, expressed by  $\mu\ll\nu$,  we write  $\mathrm d\mu/\mathrm d \nu$ for  the Radon--Nikod\'ym derivative of $\mu$ with respect to $\nu$.  
We denote  the $m$-dimensional Hausdorff measure  by $\colH^m$ \cite[Definition 2.46]{amfupa00}, write  $\dim_\text{H}$ for the Hausdorff dimension \cite[Definition 2.51]{amfupa00}, and let  $\colL^m$ be the $m$-dimensional Lebesgue measure. 
The $\colH^m$-measurable set $\setA\subseteq\reals^n$ is   said to be  $\colH^m$-rectifiable  \cite[Definition 2.57]{amfupa00} if $\colH^m(\setA)<\infty$  and there exist  Lipschitz mappings $f_i\colon\reals^m\to\reals^n$, $i\in\naturals$,  such that 
\begin{align}
\colH^m\mleft(\setA\setminus\bigcup_{i\in\naturals}f_i(\reals^m)\mright)=0. 
\end{align}  
$\opE[\,\cdot\,]$ denotes  the expectation operator. 
For  $X$ taking  values in the  $\sigma$-finite measure space $(\setX,\colX,\mu)$ with $\mu_X\ll\mu$,  we define the generalized entropy  as 
\begin{align}\label{eq:entropy}
h_\mu(X)
&=-\opE\mleft[\log\mleft( \frac{\mathrm d\mu_X}{\mathrm d\mu}(X)\mright)\mright].  
\end{align}
For $X$ taking values in  $(\setX,\colX)$  and $Y$ taking values in $(\setY,\colY)$, the mutual information between $X$ and $Y$ is 
\begin{align}\label{eq:I}
I(X;Y)=\opE\mleft[\log \mleft(\frac{\mathrm d\mu_{(X,Y)}}{\mathrm d(\mu_X\otimes\mu_Y)}(X,Y)\mright)\mright]
\end{align}
if $\mu_{(X,Y)}\ll\mu_X\otimes\mu_Y$, and $I(X;Y)=\infty$ else.  
For given distortion function $\rho\colon \setX\times\setY\to[0,\infty]$, we define the open $\rho$-balls of radius $\delta\in(0,\infty)$ centered at  $y\in\setY$ and  $x\in\setX$  according to
\begin{align} \label{eq:Ball1}
\setB_{\rho}\mleft(y,\delta\mright)=\mleft\{x\in\setX:\rho(x,y)<\delta\mright\}
\end{align}
and
\begin{align} \label{eq:Ball2}
\widetilde\setB_{\rho}\mleft(x,\delta\mright)=\mleft\{y\in\setY:\rho(x,y)<\delta\mright\},
\end{align}
respectively. 
For $a\in(0,\infty)$, the gamma function is given by $\Gamma(a)=\int_0^\infty t^{a-1}e^{-t}\,\mathrm d t$.  
For $a\in(0,\infty)$ and $s\in[0,\infty)$, the lower and upper incomplete gamma functions  are $\gamma(a,s)=\int_0^s t^{a-1}e^{-t}\,\mathrm d t$ and  
$\Gamma(a,s)=\int_s^\infty t^{a-1}e^{-t}\,\mathrm d t$, respectively.  
For $a,b\in(0,\infty)$, the beta function is given by $B(a,b)=\int_0^1 u^{a-1}(1-u)^{b-1} \,\mathrm d u$ and satisfies the identity  \cite[Theorem 1.1.4]{anasro99} 
\begin{align}\label{eq:gammabeta}
B(a,b)=\frac{\Gamma(a)\Gamma(b)}{\Gamma(a+b)}.
\end{align}
For $a,b\in(0,\infty)$ and  $s\in(0,1]$, the incomplete beta function is defined as 
$B_{a,b}(s)=\int_0^su^{a-1}(1-u)^{b-1}\, \mathrm d u$ and  
\begin{align}\label{eq:normalizedI}
I_{a,b}(s)=\frac{B_{a,b}(s)}{B(a,b)}
\end{align}
denotes the normalized  incomplete beta function.

For  $a\in \reals$, we let  $\lfloor a \rfloor$ be  the largest integer less than or equal to $a$.   
 $\log $ denotes the  logarithm   to  base $e$  and we set  $\sinc(x)=\sin(\pi x)/(\pi x)$ for all $x\in(-\infty,\infty)\!\setminus\!\{0\}$ with $\sinc(0):=1$ by  continuous extension. 
 We write $\ind{A}(\cdot)$ for the indicator function on the set $\setA$. 
For  $r\in(0,\infty)$ and  $d\in\naturals$ with $d>1$,  we set $\setS^{d-1}(r)=\{x\in\reals^d:\lVert x\rVert_2=r\}$ with corresponding surface area
\begin{align}\label{eq:area}
a^{(d-1)}(r):= \colH^{d-1}(\setS^{d-1}(r))=\frac{2\pi^{\frac{d}{2}}}{\Gamma\mleft(\frac{d}{2}\mright)}r^{d-1}. 
\end{align}
For  $d\in\naturals$ and $\lVert\,\cdot\,\rVert$ a norm on $\reals^d$, we write $\kappa_d({\lVert\,\cdot\,\rVert})$  for the volume of the corresponding unit ball in $\reals^d$ and we set
\begin{align}\label{eq:lebball}
 v^{(d)}(r)=\kappa_d({\lVert\,\cdot\,\rVert_2}) r^d= \frac{\pi^\frac{d}{2} r^d}{\Gamma\mleft(1+\frac{d}{2}\mright)}. 
\end{align}
We use the convention  $0\cdot\infty=0$.  

\section{The Regularity Conditions} \label{sec:sub}
We start by stating  the formal definitions of subregularity, superregularity, and regularity of measures. 

\begin{dfn}\label{def:subreg}
Let $(\setX,\colX)$ and   $(\setY,\colY)$ be measurable  spaces  equipped with the distortion function  $\rho\colon \setX\times\setY\to[0,\infty]$.   
\begin{enumerate}[(iii)]
\item\label{item:sub1} A  measure $\mu$ on $(\setX,\colX)$ is said to be $\rho$-subregular of dimension $m\in[0,\infty)$ if there  exist subregularity 
constants $c\in(0,\infty)$  and $\delta_0\in(0,\infty]$  such that  
\begin{align}\label{eq:subregularity}
\mu\mleft(\setB_{\rho}\mleft(y,\delta\mright)\mright)\leq c\delta^m\quad\text{for all $y\in\setY$ and  $\delta\in (0,\delta_0)$}.
\end{align}
\item \label{item:sub2} A  measure $\nu$ on $(\setY,\colY)$  is said to be $\rho$-superregular of dimension $m\in[0,\infty)$ if there exist 
superregularity constants $\supc\in(0,\infty)$ and $\delta_0\in(0,\infty]$  such that  
\begin{align}\label{eq:superregularity}
\nu\mleft(\widetilde \setB_{\rho}\mleft(x,\delta\mright)\mright)\geq  \supc \delta^m\quad\text{for all $x\in\setX$ and  $\delta\in (0,\delta_0)$}.
\end{align}
\item \label{item:sub3}  A measure $\nu$ on $(\setY,\colY)$ is said to be {$\rho$-regular of dimension $m\in[0,\infty)$} if there exist regularity
constants $\supc,c\in(0,\infty)$ and $\delta_0\in(0,\infty]$  such that  
\begin{align}\label{eq:regularity}
\supc \delta^m
\leq \nu\mleft(\widetilde \setB_{\rho}\mleft(x,\delta\mright)\mright)
\leq \subc\delta^m\quad\!\text{for all $x\in\setX$ and $\delta\in (0,\delta_0)$}. 
\end{align} 
\end{enumerate}
\end{dfn} 

If $\nu$ is $\rho$-regular, then it is also $\rho$-superregular with the same  constants $b$ and $\delta_0$. For $\setX=\setY$ and $\rho$  symmetric, i.e., $\rho(x,y)=\rho(y,x)$ for all $x,y\in\setX$, $\nu$ is $\rho$-regular if and only if it is  both $\rho$-subregular and $\rho$-superregular. 
For $\nu$  a Borel measure on $\reals^d$ with $\setX=\operatorname{supp}(\nu)$ compact 
and  $\rho\colon \setX\times \reals^d\to [0,\infty)$,  $\rho(x,y)=\lVert x-y\rVert$, where $\lVert\,\cdot\, \rVert$ is a norm on $\reals^d$, $\rho$-regularity of dimension $m$ agrees with regularity of dimension $m$ as defined in \cite[Definition 12.1]{grlu00}. In this case, $\rho$-regularity also implies $\rho$-subregularity provided that $\nu$ is a finite measure \cite[Lemma 12.3]{grlu00}. 



 We next give  some examples of $\rho$-regular   measures for  $\setX\subseteq \reals^d$ and $\rho\colon \setX\times \reals^d\to [0,\infty)$, 
 $\rho(x,y)=\lVert x-y\rVert$,    where $\lVert\,\cdot\, \rVert$ is a norm on $\reals^d$: 
\begin{enumerate}[(iii)]
\item \label{item:HD1}Lebesgue measure on $\setX=\reals^d$ satisfies \eqref{eq:regularity} with  $m=d$, $\supc=\subc$, and $\delta_0=\infty$.    
\item \label{item:HD2}Measures supported on discrete sets  $\setX\subseteq \reals^d$  satisfy \eqref{eq:regularity}  for $m=0$.
\item \label{item:HD3}The restricted Hausdorff measure $\mu=\colH^d|_\setX$ satisfies $0<\colH^d(\setX)<\infty$ and  \eqref{eq:regularity} with  $m=d$ 
for the following classes of  sets $\setX$: 
compact convex sets $\setX\subseteq\reals^d$ with $\operatorname{span}(\setX)=\reals^d$  \cite[Example 12.7]{grlu00}, surfaces of compact convex sets $\setX\subseteq\reals^{d+1}$ with $\operatorname{span}(\setX)=\reals^{d+1}$ \cite[Example 12.8]{grlu00}, 
$d$-dimensional compact $C^1$-submanifolds of $\reals^n$ \cite[Example 12.9]{grlu00},  
self-similar sets of similarity dimension $d$ satisfying the weak separation property \cite[Theorem 2.1]{frheolro15}, and finite unions 
of any such sets \cite[Lemma 12.4]{grlu00}. 
\end{enumerate}

 In all these  examples $\setX\subseteq \setY=\reals^d$,  $\rho(x,y)=\lVert x-y\rVert$, where $\lVert\,\cdot\,\rVert$ is a norm on $\reals^d$, and 
 the regularity dimension of the measure equals the Hausdorff dimension of  $\setX$  (for Items \ref{item:HD1} and \ref{item:HD2}, this  follows from  arguments in \cite[Section 3.2]{fa14};   for \ref{item:HD3} it is a consequence of the assumption $0<\colH^d(\setX)<\infty$). 
Note that  Definition \ref{def:subreg} is much more general as it does  not require $\setX\subseteq\setY$ and the distortion function need not be symmetric. 
The following simple example illustrates this aspect for  $\setX$ a  Grassmannian  and $\setY$ a hypersphere (the more complicated case where $\setX$ and $\setY$ are both Grassmannians is investigated in Section \ref{sec:grass}).

\begin{exa}\label{exa:subint}
Let $p,d\in\naturals$ with $p<d$ and denote by 
$\setG^\reals(p,d)$ the $p(d-p)$-dimensional   Grassmannian consisting of all $p$-dimensional subspaces of $\reals^d$  \cite[Section 1.3.2]{ch03}. For every $x\in\setG^\reals(p,d)$,  let 
$\operatorname{P}^\bot_x$ denote the orthogonal projection onto the orthogonal complement of $x$ in $\reals^d$. 
Let $\gamma_{p,d}$ be the unique uniformly distributed Borel regular measure on $\setG^\reals(p,d)$ with $\gamma_{p,d}(\setG^\reals(p,d))=1$ \cite[p. 49]{ma99}. Fix  $r\in(0,\infty)$ and consider  
 the distortion function $\rho\colon \setG^\reals(p,d)\times \setS^{d-1}(r)\to [0,\infty)$ defined according to  
\begin{align}
\rho(x,y)=\|\operatorname{P}^\bot_x(y)\|_2\quad\text{for $x\in \setG^\reals(p,d)$ and $y\in \setS^{d-1}(r)$.}
\end{align}
It follows that \cite[Proof of Lemma 3.11]{ma99}
\begin{align}
\gamma_{p,d}\mleft(\setB_\rho(y,\delta)\mright)= g_{p,d,r}(\delta)\quad\text{for all $y\in\setS^{d-1}(r)$ and $\delta\in(0,\infty)$,}
\end{align}
where 
\begin{align}
g_{p,d,r}(\delta)=\frac{\colH^{d-1}\mleft\{y\in\setS^{d-1}(1):\sum_{i=p+1}^dy_i^2< (\delta/r)^2\mright\}}{\colH^{d-1}(\setS^{d-1}(1))}
\end{align}
with $y_i$ denoting the $i$-th entry of $y=(y_1, \dots ,y_d)$. 
The function  $g_{p,d,r}(\delta)$ can be upper-bounded as follows:
\begin{align}
g_{p,d,r}(\delta)
&=\colL^{d} \mleft\{y\in\reals^d:\lVert y \rVert_2\leq 1, \sum_{i=p+1}^dy_i^2 <(\delta/r)^2 \mright\}\frac{\Gamma\mleft(1+\frac{d}{2}\mright)}{\pi^\frac{d}{2}}\label{eq:Lemhaus}\\
&\leq  \colL^{d} \mleft\{y\in\reals^d: \sum_{i=1}^py_i^2\leq 1, \sum_{i=p+1}^dy_i^2< (\delta/r)^2 \mright\}\frac{\Gamma\mleft(1+\frac{d}{2}\mright)}{\pi^\frac{d}{2}}\\ 
&=\frac{\Gamma\mleft(1+\frac{d}{2}\mright)}{r^{\,d-p}\Gamma\mleft(1+\frac{p}{2}\mright)\Gamma\mleft(1+\frac{d-p}{2}\mright)}\delta^{d-p}\label{eq:finalexama99}\\
&=\frac{2}{p\, r^{\,d-p} B(p/2,1+(d-p)/2)}\delta^{d-p}\quad\text{for all $\delta\in(0,\infty)$},\label{eq:usebeta}
\end{align}
where \eqref{eq:Lemhaus} is by \cite[Equation (3.6)]{ma99}, in \eqref{eq:finalexama99} we  used \eqref{eq:lebball}, 
and \eqref{eq:usebeta} follows from \eqref{eq:gammabeta} and $\Gamma(1+p/2)=(p/2)\Gamma(p/2)$. 
We can hence conclude that $\gamma_{p,d}$ is $\rho$-subregular of dimension $d-p$. 
The corresponding  subregularity constants are 
\begin{align}
c_{p,\,d,r}=\frac{2}{p\, r^{d-p} B(p/2,1+(d-p)/2)} 
\end{align}
and $\delta_0=\infty$. 
\end{exa}

We continue by listing properties of subregularity and superregularity needed throughout the paper. 
The  first result constitutes an extension of  \eqref{eq:subregularity}--\eqref{eq:regularity} to  the half-closed interval $  (0,\delta_0]$ 
applicable whenever  $\delta_0<\infty$.  
\begin{lem}\label{lem:extsub}
Let $(\setX,\colX)$ and   $(\setY,\colY)$ be measurable  spaces  equipped with the distortion function  $\rho\colon \setX\times\setY\to[0,\infty]$.  
Then, 
the following properties hold: 
\begin{enumerate}[(iii)]
\item \label{item:subregularity} If $\mu$ satisfies \eqref{eq:subregularity} for $\delta_0<\infty$, then   
\begin{align}
\mu\mleft(\setB_{\rho}\mleft(y,\delta\mright)\mright)\leq c\delta^m\quad\text{ for all $y\in\setY$ and  $\delta\in (0,\delta_0]$.}
\end{align}
\item \label{item:superregularity}
If $\nu$ satisfies \eqref{eq:superregularity} for $\delta_0<\infty$, then   
\begin{align}
\nu\mleft(\widetilde\setB_{\rho}\mleft(x,\delta\mright)\mright)\geq b\delta^m\quad\text{ for all $x\in\setX$ and  $\delta\in (0,\delta_0]$.}
\end{align}
\item \label{item:regularity}
 If $\mu$ satisfies \eqref{eq:regularity} for $\delta_0<\infty$, then 
\begin{align} 
   b\delta^m\leq\mu\mleft(\widetilde \setB_{\rho}\mleft(x,\delta\mright)\mright)\leq c\delta^m \quad\text{for all $x\in\setX$ and  $\delta\in (0,\delta_0]$.}
\end{align}   
\end{enumerate}
\end{lem}
\begin{proof}
To establish Item \ref{item:subregularity}, suppose that  \eqref{eq:subregularity} holds with $\rho_0<\infty$. Then, we have  
\begin{align}
\mu\mleft(\setB_{\rho}\mleft(y,\delta_0 \mright)\mright)
&=
\mu\mleft(\bigcup_{n=\lceil1/\rho_0 \rceil+1}^\infty\setB_{\rho}\mleft(y,\delta_0-\frac{1}{n} \mright)\mright)\\
&=\lim_{n\to\infty}\mu\mleft(\setB_{\rho}\mleft(y,\delta_0-\frac{1}{n} \mright)\mright)\label{eq:useBartle}\\
&\leq c\lim_{n\to\infty} \mleft(\delta_0-\frac{1}{n}\mright)^m\label{eq:usesub}\\ 
&=c\delta_0^m,
\end{align}
where \eqref{eq:useBartle} follows from \cite[Lemma 3.4, Item (a)]{ba95}  and \eqref{eq:usesub} is by \eqref{eq:subregularity} with $0<\delta=\delta_0-1/n<\delta_0$ for all $n\geq \lceil1/\rho_0 \rceil+1$. 
The proofs of Items \ref{item:superregularity} and \ref{item:regularity} 
follow along the same lines.
\end{proof} 

Next, we establish  scaling properties of  subregular and superregular measures. 

\begin{lem}\label{lem:simpilysub}
Let $(\setX,\colX)$ and $(\setY,\colY)$ be   measurable spaces equipped with the distortion function   $\rho\colon \setX\times\setY\to[0,\infty]$ and fix $k,\alpha, \beta\in (0,\infty)$.   
Then, for a measure $\mu$ on  $(\setX,\colX)$, the following statements  are equivalent: 
\begin{enumerate}[(ii)]
\item $\mu$ is $\rho^{1/k}$-subregular of dimension $m$ with  subregularity  constants $c\in(0,\infty)$ and $\delta_0\in(0,\infty]$. \label{property0sub}
\item $\beta \mu$ is 
$(\alpha\rho)$-subregular of dimension $m/k$  with subregularity  constants  $\tilde c=\beta c/\alpha^{m/k}$ and $\tilde \delta_0 =\alpha \delta_0^{k}$. \label{property1sub}
\end{enumerate}
Similarly, for a measure $\nu$ on  $(\setY,\colY)$, the following statements are equivalent:
\renewcommand{\theenumi}{(\alph{enumi})}
\begin{enumerate}[(a)]
\item $\nu$ is $\rho^{1/k}$-superregular of dimension $m$ with  superregularity  constants $\supc\in(0,\infty)$ and $\delta_0\in(0,\infty]$. \label{property0sup}
\item $\beta \nu$ is 
$(\alpha\rho)$-superregular of dimension $m/k$  with superregularity  constants  $\tilde \supc=\beta \supc/\alpha^{m/k}$ and $\tilde \delta_0 =\alpha \delta_0^{k}$. \label{property1sup}
\end{enumerate}
\renewcommand{\theenumi}{(\roman{enumi})}

\end{lem}
\begin{proof}
Follows directly from Definition \ref{def:subreg}.
\end{proof}
By Lemma \ref{lem:simpilysub}, applied with $k=\alpha=1$, 
we can hence assume, without loss of generality, that finite subregular measures $\mu$ on  $(\setX,\colX)$  and finite superregular measures $\nu$ on $(\setY,\colY)$ are 
normalized according to  $\mu(\setX)=1$ and $\nu(\setY)=1$, respectively.  

Next, we show that if  $\mu$ is subregular with subregularity constants  $c,\delta_0\in(0,\infty)$   and $\mu(\setX)=1$, then $c$ can be modified to make the subregularity condition \eqref{eq:subregularity} hold  for $\delta_0=\infty$. 
The corresponding formal statement is as follows.   

\begin{lem}\label{lem:rho}
Let $(\setX,\colX,\mu)$ be a finite measure space with $\mu(\setX)=1$, let  $(\setY,\colY)$ be a measurable space, and consider the   
distortion function  $\rho\colon \setX\times\setY\to[0,\infty]$. 
If $\mu$ is $\rho$-subregular of dimension $m$ with  subregularity constants $c,\delta_0\in(0,\infty)$,  
then 
$\mu$ is also $\rho$-subregular of dimension $m$ with  subregularity constants $\tilde c=\max(c,\delta_0^{-m})$ and $\tilde \delta_0=\infty$, i.e., 
\begin{align}\label{eq:global}
\mu\mleft(\setB_{\rho}\mleft(y,\delta\mright)\mright)\leq 
\tilde c\delta^m  \quad\text{for all $y\in\setY$ and  $\delta \in  (0, \infty)$.} 
\end{align}
\end{lem}
\begin{proof}
Let $y\in\setY$ and $\delta\in(0,\infty)$ be arbitrary but fixed. 
If $\delta<\delta_0$, then \eqref{eq:global} follows directly from \eqref{eq:subregularity}. 
For $\delta\geq \delta_0$,  \eqref{eq:global} holds by 
$\mu\big(\setB_{\rho}(y,\delta)\big)\leq \mu(\setX) =1 \leq  \delta_0^{-m}\delta^m$. 
\end{proof}

In light of Lemma \ref{lem:rho}, if  $\mu$ is $\rho$-subregular of dimension $m$ with subregularity constants  $c,\delta_0\in(0,\infty)$ satisfying $c\geq \delta_0^{-m}$  and  if $\mu(\setX)=1$,  then the subregularity condition \eqref{eq:subregularity} holds  for $\delta_0=\infty$ with $c$ unchanged.

The following result allows to infer subregularity of $\tilde \mu\ll\mu$ from subregularity of $\mu$. 
\begin{lem}\label{lem:transformsub}
Let $(\setX,\colX,\mu)$ be a  measure space, let  $(\setY,\colY)$ be a measurable space, and consider the   
distortion function  $\rho\colon \setX\times\setY\to[0,\infty]$. 
Suppose that  $\mu$ is $\rho$-subregular of dimension $m$ with  subregularity constants $c\in(0,\infty)$ and $\delta_0\in(0,\infty]$. If $\tilde \mu\ll \mu$ 
and  $\lVert\mathrm d \tilde \mu/\mathrm d \mu\rVert^{(\mu)}_{p/(p-1)}<\infty$ with $p\in[1,\infty)$, then $\tilde \mu$ is  $\rho$-subregular of dimension $m/p$ with subregularity constants 
$\lVert\mathrm d \tilde \mu/\mathrm d \mu\rVert^{(\mu)}_{p/(p-1)} c^{1/p}$   and $\delta_0$.
\end{lem}
\begin{proof}
We have 
\begin{align}
\tilde \mu\mleft(\setB_{\rho}\mleft(y,\delta\mright)\mright)
&=\mleft\lVert \frac{\mathrm d \tilde \mu}{\mathrm d \mu}\ \ind{\setB_{\rho}\mleft(y,\delta\mright)} \mright\rVert_1^{(\mu)}\\
&\leq \mleft\lVert  \ind{\setB_{\rho}\mleft(y,\delta\mright)} \mright\rVert_p^{(\mu)}    \mleft\lVert   \frac{\mathrm d \tilde \mu}{\mathrm d \mu}\mright\rVert_{p/(p-1)}^{(\mu)}\label{eq:useHoellder}\\
&=   \mleft(\mu\mleft(\setB_{\rho}\mleft(y,\delta\mright)\mright)\mright)^{\frac{1}{p}}  \mleft\lVert   \frac{\mathrm d \tilde \mu}{\mathrm d \mu}\mright\rVert_{p/(p-1)}^{(\mu)} \\
&\leq c^\frac{1}{p}\mleft\lVert   \frac{\mathrm d \tilde \mu}{\mathrm d \mu}\mright\rVert_{p/(p-1)}^{(\mu)} \delta^{\frac{m}{p}}\quad\text{for all $y\in\setY$ and $\delta\in(0,\delta_0)$,}\label{eq:usesubreeg}
\end{align}
where in \eqref{eq:useHoellder} we applied H\"older's inequality \cite[Theorem 1, p. 372]{ka18} and \eqref{eq:usesubreeg} is by subregularity of $\mu$. 
\end{proof}
If  $\lVert \mathrm d \tilde \mu/\mathrm d \mu\rVert_{\infty}^{(\mu)}<\infty$, then Lemma \ref{lem:transformsub} can be applied with $p=1$ to conclude that not only does $\tilde \mu$ inherit subregularity from $\mu$, but does so while preserving the subregularity dimension as a consequence of $m/p=m$.  
If $\lVert \mathrm d \tilde \mu/\mathrm d \mu\rVert_{\infty}^{(\mu)}=\infty$, the subregularity dimension need not be preserved, even when $\lVert \mathrm d \tilde \mu/\mathrm d \mu\rVert_{q}^{(\mu)}<\infty$ for all $q\in[1,\infty)$, as is   illustrated by the following example. 
\begin{exa}\label{exa:submore}
Take $\setX=\setY=[0,1]$  and $\tilde \mu\ll  \colL^1 \vert_\setX$ with 
\begin{align}
\frac{\mathrm d \tilde \mu}{\mathrm d \colL^1 \vert_\setX}(x)=-\log(x).
\end{align}
 It follows that 
\begin{align}
\lVert \mathrm d \tilde \mu/\mathrm d \colL^1\vert_\setX\rVert_{q}^{(\colL^1\vert_\setX)}
&= \mleft( \int_0^1 (-\log (x))^q\ \mathrm d x\mright)^{\frac{1}{q}}\label{eq:pq1} \\
&= \Gamma^{\frac{1}{q}}(1+q)\\
& <\infty \quad\text{for all $q\in [1,\infty)$}\label{eq:pq2}
\end{align}
and $\lVert \mathrm d \tilde \mu/\mathrm d \colL^1\vert_\setX\rVert_{\infty}^{(\colL^1\vert_\setX)}=\infty$.  
Lemma \ref{lem:transformsub} with $p=q/(q-1)$ together with  \eqref{eq:pq1}--\eqref{eq:pq2} therefore implies that $\tilde \mu$ is $\lvert\,\cdot\,\rvert$-subregular of dimension $m/p$ for all $p\in (1,\infty)$. 
But 
\begin{align}
\tilde \mu\{x\in[0,1]: x<\delta\}
&=-\int_0^\delta \log (x)\, \mathrm d x\\
&=\Gamma(2,\log(1/\delta))\\
&\geq \log(1/\delta) \delta \quad\text{for all $\delta\in (0,1)$,\label{eq:useJameson}} 
\end{align}
where \eqref{eq:useJameson} follows from  \cite[Equation (6)]{ja16}, which implies that $\tilde \mu$ is not  $\lvert\,\cdot\,\rvert$-subregular of dimension $1$. 
\end{exa}

The following result allows to deduce subregularity/superregularity of product measures from  subregularity/superregularity of their constituent  measures.  
 
\begin{prp}\label{prp:subprod}
Fix $k\in  (0, \infty)$.  For $i=1,\dots, \ell$, let $(\setX_i,\colX_i)$ and  $(\setY_i,\colY_i)$  be measurable spaces equipped  with the  distortion function  $\rho_i\colon \setX_i\times\setY_i\to[0,\infty]$. 
 Consider the weighted distortion function 
\begin{align}\label{eq:barrho}
\rho_{(\ell)}((x_1,\dots,x_\ell),(y_1,\dots,y_\ell))=\sum_{i=1}^\ell\alpha_i\rho_i(x_i,y_i) 
\end{align}
on  $(\setX_1\times\dots\times  \setX_\ell,\setY_1\times\dots\times  \setY_\ell)$, where $\alpha_i\in(0,\infty)$ for $i=1,\dots, \ell$.  Then, the following  holds:  
\begin{enumerate}[(ii)]
\item \label{item:prodsub}
Suppose that, for  $i=1,\dots,\ell$,  
$\mu_i$ is a $\sigma$-finite   $\rho_i^{1/k}$-subregular measure on $(\setX_i,\colX_i)$  of   dimension $m_i\in(0,\infty)$ with subregularity constants $c_i\in  (0, \infty)$ and $\delta_i\in(0,\infty]$. 
Then, $\mu^{( \ell)}:=\mu_1\otimes\dots\otimes \mu_\ell$ satisfies the subregularity condition 
\begin{align}\label{eq:toshowprodP}
\mu^{( \ell)} \mleft(\setB_{\rho_{(\ell)}^{1/k}}\mleft(y^{(\ell)},\delta\mright)\mright)\leq 
c_{(\ell)}\delta^{m_{(\ell)}} 
\end{align}
for all $y^{(\ell)}\in\setY_1\times\dots\times  \setY_\ell$ and $\delta\in(0,\delta_{(\ell)})$ with  subregularity dimension $m_{(\ell)}=\sum_{i=1}^\ell m_i$ 
and subregularity constants 
\begin{align}\label{eq:parsubsup}
c_{(\ell)}&=\frac{\prod_{i=1}^\ell \Gamma\mleft(1+\frac{m_i}{k}\mright)}{\Gamma\mleft(1+\sum_{i=1}^\ell\frac{m_i}{k}\mright)} \times \prod_{i=1}^\ell\frac{c_i}{\alpha_i^{\frac{m_i}{k}}}
\end{align}
and 
\begin{align}\label{eq:parsubsupb}
\delta_{(\ell)}& = \min\mleft\{\alpha_1^{1/k}\delta_1, \dots,\alpha_\ell^{1/k}\delta_\ell\mright\}.
\end{align}
\item \label{item:prodsup}
Suppose that, for  $i=1,\dots,\ell$,  
$\nu_i$ is a $\sigma$-finite  $\rho_i^{1/k}$-superregular measure on $(\setY_i,\colY_i)$  of  dimension $m_i\in(0,\infty)$ with superregularity constants $\supc_i\in(0,\infty)$ and $\delta_i\in(0,\infty]$. 
Then, $\nu^{( \ell)}:=\nu_1\otimes\dots\otimes \nu_\ell$ satisfies the superregularity condition 
\begin{align}\label{eq:toshowprodPsup}
\nu^{( \ell)}\mleft(\widetilde\setB_{\rho_{(\ell)}^{1/k}}\mleft(x^{(\ell)},\delta\mright)\mright)\geq 
\supc_{(\ell)}\delta^{m_{(\ell)}} 
\end{align}
for all $x^{(\ell)}\in\setX_1\times\dots\times  \setX_\ell$ and $\delta\in(0,\delta_{(\ell)})$ with superregularity dimension $m_{(\ell)}=\sum_{i=1}^\ell m_i$ 
and superregularity constants
\begin{align}\label{a	}
\supc_{(\ell)}&=\frac{\prod_{i=1}^\ell \Gamma\mleft(1+\frac{m_i}{k}\mright)}{\Gamma\mleft(1+\sum_{i=1}^\ell\frac{m_i}{k}\mright)} \times \prod_{i=1}^\ell\frac{\supc_i}{\alpha_i^{\frac{m_i}{k}}}
\end{align}
and 
\begin{align}\label{eq:parsubsupb2}
\delta_{(\ell)}& = \min\mleft\{\alpha_1^{1/k}\delta_1, \dots,\alpha_\ell^{1/k}\delta_\ell\mright\}.
\end{align}
\end{enumerate}
\end{prp} 
\begin{proof}See Appendix~\ref{sec:proofsubprod}.\end{proof}

The weighted distortion function $\rho_{(\ell)}$ in \eqref{eq:barrho} covers the following important special cases:
\begin{itemize}
\item
If $\alpha_i=1/\ell$ and $\rho_i=\sigma$ for $i=1,\dots,\ell$, then $\rho_{(\ell)}$ equals the average distortion function of length $\ell$ in \eqref{eq:ghol}.
\item
If $\alpha_i=1$ and $\rho_i(x,y)=\lVert x-y\rVert^k_k$ for $i=1,\dots,\ell$, then 
\begin{align}
\rho_{(\ell)}((x_1,\dots,x_\ell),(y_1,\dots,y_\ell))=\lVert(x_1-y_1,\dots,x_\ell-y_\ell)\rVert^k_k. \label{eq:direct}
\end{align}   
\end{itemize}
For illustration purposes, we show that the bounds in Proposition \ref{prp:subprod}  are tight enough to recover 
\eqref{eq:lebball} using  regularity of  Lebesgue measure $\colL^\ell$ on $\reals^\ell$. To this end, note that 
 \cite[Corollary 6.7]{kn05} 
\begin{align}\label{eq:Lebprod}
\colL^\ell= \underbrace{\colL^1 \otimes\dots \otimes\colL^1}_{\text{$\ell$ times}}. 
\end{align}
Moreover,  $\colL^1(-\delta,\delta)=2\delta$ for all $\delta\in(0,\infty)$ implies that  $\colL^1$ is $\lvert\,\cdot\,\rvert$-regular of dimension $1$ with regularity constants $b=c=2$ and $\delta_0=\infty$.
Application of Items \ref{item:prodsub}  and  \ref{item:prodsup} of Proposition \ref{prp:subprod} with   $\alpha_i=1$, $\rho_i=\lvert\,\cdot\,\rvert^k$, $m_i=1$,  $b_i=c_i=2$, and $\delta_i=\infty$, all for $i=1,\dots,\ell$,  yields 
\begin{align}
\colL^\ell(\setB_{\lVert\,\cdot\,\rVert_k}(x,\delta)) 
&= \frac{\Gamma^\ell\mleft(1+\frac{1}{k}\mright)}{\Gamma\mleft(1+\frac{\ell}{k}\mright)}2^\ell\delta^\ell
\quad\text{for all $x\in\reals^\ell$ and  $k,\delta\in(0,\infty)$,}\label{eq:indirect2}
\end{align}
which recovers the well-known volume formula for $\lVert\,\cdot\,\rVert_k$-balls in $\reals^\ell$ (see, e.g., \cite[Theorem 5]{kovy20}). Particularized to $k=2$  and using  $\Gamma(3/2)=\sqrt{\pi}/2$, this yields \eqref{eq:lebball}.

Note that for the Hausdorff measures, a factorization akin to that for Lebesgue measure in \eqref{eq:Lebprod} is not possible in general. 
There is, however, an important exception. Concretely, consider an   $\colH^s$-rectifiable set 
 $\setE\subseteq\reals^m$ and let  $\setF\subseteq\reals^n$ be such that  $\setF=f(\setA)$ with  $f\colon \reals^p\to\reals^n$ a Lipschitz mapping and $\setA\subseteq \reals^p$ compact. Then, $\setF$ is  $\colH^p$-rectifiable  by construction and  we have  \cite[Theorem 3.2.23]{fed69}
\begin{align}\label{eq:splittingH}
\colH^{s+p}\vert_{\setE\times\setF}=\colH^{s}\vert_{\setE}\otimes \colH^{p}\vert_{\setF}. 
\end{align}
The specific form of the $\colH^p$-rectifiable set $\setF$  is required as $\colH^p$-rectifiability alone is not enough for \eqref{eq:splittingH} to hold  
 \cite[Remark 3.2.24]{fed69}. Applying \eqref{eq:splittingH} inductively 
 for a set $\setF$ as above,  we find that 
\begin{align}\label{eq:splittingH2}
\colH^{\ell p}\vert_{\setF^{\ell}} 
= \underbrace{\colH^{p}\vert_{\setF}\otimes\dots\otimes  \colH^{p}\vert_{\setF}}_{\text{$\ell$ times}}\quad\text{for all $\ell\in\naturals$.}
\end{align}  
This factorization now allows the application of Proposition  \ref{prp:subprod}  (with $\alpha_i=1$ and $\rho_i=\lVert\,\cdot\,\rVert^k$ for $i=1,\dots,\ell$)  to get  
upper/lower\ bounds on $ \colH^{\ell p}\vert_{\setF^\ell}(\setB_{\lVert\,\cdot\,\rVert_k}(x,\delta))$ for all $x\in\reals^{ \ell p} $ and $\delta\in(0,\infty)$ 
in terms of   upper/lower bounds on $\colH^{p}\vert_{\setF}(\setB_{\lVert\,\cdot\,\rVert_k}(x,\delta))$ for all $x\in\reals^{p}$ and $\delta\in(0,\infty)$ upon noting that  
$\sigma$-finiteness of $\colH^{p}\vert_{\setF}$ is guaranteed by 
\cite[Lemma 15.5]{ma99}.

\section{A Lower Bound on the Rate-Distortion Function} \label{sec:slb}
The best-known lower bounds on $R_X(D)$ in \eqref{eq:RD} are the  Shannon lower bounds for i)
discrete $X$ of finite entropy and with 
distortion function  $\rho\colon\setX\times\setY\to[0,\infty]$  such that
$\sum_{x\in\setX}e^{-s\rho(x,y)}$ is independent of $y$ for all $s\in(0,\infty)$  \cite[Lemma 4.3.1]{gr90}
and ii)  continuous  $X$  of finite differential entropy and with a difference distortion function  \cite[Equation (4.6.1)]{gr90}.    
We next present  a general Shannon-type lower bound that applies to all random variables $X$ taking values in a $\sigma$-finite measure space $(\setX,\colX,\mu)$, where $\mu$ is such that   $\mu_X\ll\mu$ and $\abs{h_\mu(X)}<\infty$, but arbitrary otherwise. 
   
\begin{prp}\label{thm.ko17}
Let $X$ be a random variable taking values in  the $\sigma$-finite measure space $(\setX,\colX,\mu)$. Assume that   $\mu_X\ll\mu$ with  $\abs{h_\mu(X)}<\infty$, let  $(\setY,\colY)$ be a measurable space, and 
consider the distortion function 
  $\rho\colon \setX\times\setY\to[0,\infty]$.    
Then, $R_X(D)\geq R_X^{\text{SLB}}(D)$ for all $D\in(0,\infty)$, where 
\begin{align}
R_X^{\text{SLB}}(D)= h_\mu(X)-\inf_{s\geq 0}\mleft(sD +\log (g(s))\mright)\label{eq:SLB}
\end{align}
with 
\begin{equation}
g(s)=\sup_{y\in\setY}\int e^{-s\rho(x,y)} \mathrm d \mu(x)\quad\text{for all $s\in(0,\infty)$.} \label{eq:defns1}
\end{equation}
\end{prp}
\begin{proof}  See Appendix \ref{thm.ko17proof}.\end{proof}

Proposition \ref{thm.ko17} covers the following important special cases:
\begin{itemize}
\item 
For discrete $X$   of finite entropy, $\mu$  the counting measure,  and $\sum_{x\in\setX}e^{-s\rho(x,y)}$ independent of $y$ for all $s\in(0,\infty)$, $R_X^{\text{SLB}}(D)$ in \eqref{eq:SLB} recovers the  Shannon lower bound for discrete random variables reported in \cite[Lemma 4.3.1]{gr90}; in this case    $h_\mu(X)$ equals the  Shannon entropy  of $X$. 
\item
For $X$ continuous, $\mu$ the Lebesgue measure, $\setX=\setY=\reals^d$, and $\rho$ a difference distortion function,  \color{black} $R_X^{\text{SLB}}(D)$ in \eqref{eq:SLB} \color{black} recovers  the Shannon lower bound for continuous random variables \cite[Equation (4.6.1)]{gr90}; in this case  $h_\mu(X)$ equals  the differential entropy of $X$. 
This lower bound can be evaluated explicitly for $\rho(x,y)=\lVert x-y\rVert^k$ with $\lVert\, \cdot\,\rVert$ a semi-norm and $k\in(0,\infty)$, resulting in  the classical form of the Shannon lower bound \cite[Section VI]{yatagr80}
\begin{equation} \label{eq:shlbclassic}
R_X^{\text{SLB}}(D) = 
h_\mu(X) +F_{d,k,\kappa_d({\lVert\,\cdot\,\rVert})}(D), 
\end{equation}
where, for $m,k,c,D\in(0,\infty)$, we set    
\begin{align}\label{eq:F}
F_{m,k,c}(D)=\log\mleft(\frac{\mleft(\frac{m}{kD}\mright)^\frac{m}{k}}{c\, \Gamma\mleft(1+\frac{m}{k}\mright)}\mright)-\frac{m}{k}.
\end{align}
\item
For the class of   $m$-rectifiable  random variables \cite[Definition 11]{kopirihl16},  Proposition \ref{thm.ko17} recovers \cite[Theorem 55]{kopirihl16}.  In this case, 
$h_\mu(X)$ is the   $m$-dimensional entropy \cite[Definition 18]{kopirihl16} of the $m$-rectifiable  random variable $X$. 
\end{itemize}

The explicit expression  in \eqref{eq:shlbclassic}  is made possible by a simplification of 
 $g(s)$ exploiting that i) $\rho$ is a difference distortion function and ii) Lebesgue measure is translation invariant. Specifically, we have   
\begin{align}
g(s)
&= \sup_{y\in\setY}\int e^{-s\rho(x-y,0)} \,\mathrm d \mu(x)\\
&= \int e^{-s\rho(x,0)} \,\mathrm d \mu(x),\label{eq:defns1b}
\end{align}
which can be evaluated explicitly for $\rho(x,y)=\lVert x-y\rVert^k$  with $\lVert\, \cdot\,\rVert$ a semi-norm and $k\in(0,\infty)$ \cite[Section II]{yatagr80}.  
For  general distortion functions  $\rho$ and general measures $\mu$, a simplification of $g(s)$ and hence of the lower bound in \eqref{eq:SLB} does not seem  possible. 
It turns out, however,  that subregular measures permit an explicit  upper bound on $g(s)$, which in turn leads to the following explicit lower bound on   $R_X^{\text{SLB}}(D)$.

\begin{thm}\label{thm:new}
Let   $X$ be a random variable taking values in  the $\sigma$-finite measure space $(\setX,\colX,\mu)$. Assume that    $\mu_X\ll\mu$ with   $\abs{h_\mu(X)}<\infty$, let  $(\setY,\colY)$ be  a measurable space, and consider the distortion function  $\rho\colon \setX\times\setY\to[0,\infty]$. 
Further,  let $k\in(0,\infty)$ and  suppose that  
  $\mu$ is     $\rho^{1/k}$-subregular  of dimension $m\in(0,\infty)$ 
with subregularity constants $c\in(0,\infty)$ and $\delta_0\in(0,\infty]$.
\color{black} Finally, assume that $\mu(\setX)=1$ if  $\delta_0<\infty$. \color{black}  
Then,  
\begin{align}\label{eq:toshow1}
R_X^{\text{SLB}}(D)&\geq R_X^{\text{L}}(D)\quad\text{for all $D\in(0,\infty)$}     
\end{align}
with $R_X^{\text{L}}(D)$ defined as follows:
\begin{enumerate}[(ii)]
\item 
\color{black}If  $ \delta_0=\infty$, then \color{black}
\begin{align}
R_X^{\text{L}}(D)
=h_\mu(X)+ F_{m,k,c}(D)\label{eq:SLB1a}
\end{align}
with $F_{m,k,c}(D)$ as defined  in \eqref{eq:F}. 
\item
\color{black}If $\delta_0<\infty$, then \color{black} 
\begin{align}
R_X^{\text{L}}(D)
=h_\mu(X)-\frac{m}{k} -\log \mleft(c\mleft(\frac{m}{kD}\mright)^{- \frac{m}{k}}\Gamma\mleft(1+ \frac{m}{k}\mright) + e^{-\frac{m\delta_0^k}{kD}}\mright)\label{eq:SLB1}
\end{align}
 with 
 \begin{align}\label{eq:RLasympA}
 R_X^{\text{L}}(D) <  h_{\mu}(X) +F_{m,k,c}(D)\quad\text{for all $D\in(0,\infty)$} 
 \end{align}
 and
 \begin{align}\label{eq:RLasymp}
 \lim_{D\to 0}(R_X^{\text{L}}(D)- h_{\mu}(X) -F_{m,k,c}(D))=0.
 \end{align} 
 \end{enumerate}
\end{thm}
\begin{proof}See Appendix \ref{thm:newproof}. 
\end{proof}

The proof of Theorem \ref{thm:new} reveals that 
 if $\mu$ satisfies the  subregularity condition with equality  for    $\delta_0=\infty$, then the lower  bound in \eqref{eq:toshow1} holds with equality, i.e., we have 
 \begin{align}\label{eq:yamadared}
  R_X^{\text{SLB}}(D)=R_X^{\text{L}}(D)=h_\mu(X) +F_{m,k,c}(D)\quad\text{for all $D\in(0,\infty)$} 
\end{align}  
with $F_{m,k,c}(D)$ as defined in \eqref{eq:F}. 
In particular, this applies to Lebesgue measure so that Theorem \ref{thm:new}  generalizes  the classical Shannon lower bound 
\eqref{eq:shlbclassic} to arbitrary  distortion functions. 

We finally note that if  $X$ satisfies the assumptions of Theorem \ref{thm:new},  then the  lower R-D dimension  of order $k\in(0,\infty)$ is lower-bounded by the subregularity dimension. The corresponding formal statement is as follows. 
\begin{cor}\label{lem:cordimR}
Under the assumptions of Theorem \ref{thm:new},   
 we have 
\begin{align}
\underline{\dim}_\text{R}(X)\geq m.
\end{align} 
\end{cor}
\begin{proof}
The lower bound on $\underline{\dim}_\text{R}(X)$ follows from 
\begin{align}
\underline{\dim}_\text{R}(X)=
&\liminf_{D\to 0} \frac{R_{X}\big(D^k\big)} {\log (1/D)}\\ 
&\geq\liminf_{D\to 0} \frac{R_X^{\text{SLB}}\big(D^k\big)} {\log (1/D)}\label{eq:asyRD0}\\ 
&\geq \lim_{D\to 0} \frac{R_{X}^{\text{L}}\big(D^k\big)} {\log (1/D)} \label{eq:asyRD}\\
&= \lim_{D\to 0} \frac{F_{m,k,c}\big(D^k\big)} {\log (1/D)}\label{eq:asyRD2}\\ 
&=m, 
\end{align}
where  \eqref{eq:asyRD0} is by  Proposition \ref{thm.ko17} and in  \eqref{eq:asyRD} we applied  Theorem \ref{thm:new}.
\end{proof}

\section{Analysis of the $n$-th Quantization Error} \label{sec:rdl1}  
We first  derive  a lower bound $L_{n}(X)$  on $V_n(X)$ for  random variables  $X$  of  distribution 
$\mu_X\ll\mu$, where $\mu$ is a   $\sigma$-finite  $\rho^{1/k}$-subregular measure    satisfying $\lVert \mathrm d \mu_X/\mathrm d \mu\rVert_{p/(p-1)}^{(\mu)}<\infty$ with  $p\in [1,\infty)$. 
The corresponding formal statement is as follows.

\begin{thm}\label{thm:singleshot}
Let  $X$ be a random variable taking values in the $\sigma$-finite measure space $(\setX,\colX,\mu)$. Assume that    $\mu_X\ll\mu$ and 
\begin{align}
\Sigma_p(X):=\lVert \mathrm d \mu_X/\mathrm d \mu\rVert_{p/(p-1)}^{(\mu)}<\infty\quad\text{with  $p\in [1,\infty)$,}
\end{align}
 let  $(\setY,\colY)$ be  a measurable space, and  consider the  distortion function $\rho\colon \setX\times\setY\to[0,\infty]$.    
Further,  let $k\in(0,\infty)$ and suppose that  $\mu$ is $\rho^{1/k}$-subregular of dimension $m\in(0,\infty)$ with  subregularity
constants  $c\in(0,\infty)$ and $\delta_0 \in (0,\infty]$.
Then,  $V_{n}(X)\geq L_n(X)$ for all  $n\in\naturals$, 
where  
\begin{align}
L_{n}(X)= \min\mleft\{c^{-\frac{k}{m}}\Sigma_p^{-\frac{pk}{m}}(X)    n^{-\frac{pk}{m}} , \delta_0^k\mright\}  \frac{m}{m+pk}. 
\label{eq:directbound}
\end{align}
\end{thm}
\begin{proof}
See Appendix \ref{thm:singleshotproof}.
\end{proof}

Theorem~\ref{thm:singleshot} particularized to  $X$ taking values in a Riemannian manifold and such that $\mu_X$ itself is $\rho^{1/k}$-subregular with $\delta_0=\infty$  recovers \cite[Proposition 4.2]{kl12}.  
\color{black}We also note that the proof of  \cite[Proposition 4.2]{kl12} uses methods from optimal transport theory \cite{vi16}, whereas our proof technique is based on elementary probability and measure theory.\color{black} 

We next use Theorem~\ref{thm:singleshot}  to study   the $n\to\infty$   asymptotics of  $V_{n}(X)$ 
in terms of the lower $k$-th  quantization dimension  and the lower
 $k$-th  quantization coefficient. 
\begin{cor}\label{cor:DC}
Under the assumptions of Theorem \ref{thm:singleshot}, the following statements  hold:
\begin{enumerate}[(ii)]
\item$\underline{D}_k(X)\geq m/p$. \label{eq:boundD}
\item  \label{eq:boundC}If  ${D}_k(X)=m/p$, then 
\begin{align}
\underline{C}_k(X)\geq     c^{-\frac{k}{m}} \Sigma_p^{-\frac{pk}{m}}(X)  \frac{m}{m+pk} >0.
\end{align} 
\end{enumerate}
\end{cor} 
\begin{proof}
Item \ref{eq:boundD} follows from 
\begin{align}
\underline{D}_k(X)&=\liminf_{n\to\infty}  \frac{ k \log( n)}{\log (1/V_{n}(X))}\\
&\geq \lim_{n\to\infty}  \frac{ k \log (n)}{\log (1/L_{n}(X))}\label{eq:usetheorem}\\
&=m/p,\label{eq:usehX}
\end{align} 
where \eqref{eq:usetheorem} is by Theorem \ref{thm:singleshot}. 
Item  \ref{eq:boundC} is by
\begin{align}
\underline{C}_k(X)&=\liminf_{n\to\infty} n^\frac{pk}{m}V_n(X)\label{eq:coeff1}\\
&\geq \lim_{n\to\infty} n^\frac{pk}{m}L_n(X)\label{eq:coeff1a}\\
&= c^{-\frac{k}{m}}  \Sigma_p^{-\frac{pk}{m}}(X) \frac{m}{m+pk},
\end{align} 
where \eqref{eq:coeff1} is by \eqref{eq:Cinf} combined with the assumption ${D}_k(X)=m/p$  and  \eqref{eq:coeff1a}  follows from  Theorem \ref{thm:singleshot}.  
\end{proof}

Note that if  $\lVert \mathrm d \mu_X/\mathrm d \mu\rVert_{p/(p-1)}^{(\mu)}<\infty$ for all $p\in(1,\infty)$, then Item \ref{eq:boundD} of Corollary \ref{cor:DC} implies 
\begin{align}\label{eq:Dklower}
\underline{D}_k(X)\geq \sup_{p\in(1,\infty)}\mleft\{\frac{m}{p}\mright\}=m 
\end{align} 
even if  $\Sigma_1(X)=\infty$, see, e.g.,     Example  \ref{exa:submore} with  \color{black}$\mu_X=\nu$ and $\mu=\colL^1 \vert_\setX$. \color{black} 

We finally note that for the special case   $\setX=\setY=\reals^d$, $k\in[1,\infty)$, and $X$ of nonvanishing continuous part $\tilde\mu_X$ (in the Lebesgue decomposition of $\mu_X$) satisfying $ \lVert\mathrm d \tilde \mu_X/\mathrm d \colL^d\rVert^{(\mu)}_{d/(d+k)}<\infty$, it is known that  ${D}_k(X)= d$ and an expression for the $k$-th quantization coefficient is available   \cite[Theorem 6.2, Remark 6.3]{grlu00}.

We next use our results in R-D theory to obtain lower bounds on the lower quantization dimension. To this end, we first 
establish that  the $k$-th  lower quantization dimension   is lower-bounded by the lower R-D-dimension. 

\begin{lem}\label{lem:RDD}
Let  $(\setX,\colX)$ and $(\setY,\colY)$ be  measurable spaces equipped with  the  distortion function $\rho\colon \setX\times\setY\to[0,\infty]$. 
Further, let $X$ be a random variable taking values in  $(\setX,\colX)$. Then, $\underline{D}_k(X) \geq \underline{\dim}_R(X)$ for all $k\in(0,\infty)$.   
\end{lem} 
\begin{proof}
See Appendix \ref{lem:RDDproof}.
\end{proof}

In the special case $\setX\subseteq\setY=\reals^d$,    $\rho(x,y)=\|x-y\|^k$ with $\lVert\,\cdot\,\rVert$  a norm on $\reals^d$, and  $k\in [1,\infty)$, 
 Lemma  \ref{lem:RDD} particularizes to   \cite[Theorem  11.10]{grlu00}. 

As an immediate consequence of Lemma \ref{lem:RDD}, we obtain the following lower bound on the $k$-th  lower quantization dimension. 

\begin{cor}\label{cor:RDD}
Let   $X$ be a random variable taking values in  the $\sigma$-finite  measure space $(\setX,\colX,\mu)$. Assume that  $\mu_X\ll\mu$ with   $\abs{h_\mu(X)}<\infty$, let  $(\setY,\colY)$ be  a measurable space, and consider the distortion function  $\rho\colon \setX\times\setY\to[0,\infty]$. 
Further,  let $k\in(0,\infty)$ and suppose that  
  $\mu$ is     $\rho^{1/k}$-subregular  of dimension $m\in(0,\infty)$ 
with subregularity constants $c\in(0,\infty)$ and $\delta_0\in(0,\infty]$. \color{black}Finally, suppose that  $\mu(\setX)=1$ if $\delta_0<\infty$. \color{black}
Then, $\underline{D}_k(X)\geq m$. 
\end{cor}
\begin{proof}
Follows from  Corollary \ref{lem:cordimR} and  Lemma \ref{lem:RDD}.   
\end{proof}

While the lower bound $L_{n}(X)$ on  $V_{n}(X)$ in Theorem~\ref{thm:singleshot} is based on  a subregular  $\sigma$-finite  measure $\mu$ on $(\setX,\colX)$, we next derive an upper bound  on   $V_{n}(X)$ that requires the existence of a  finite superregular   measure $\nu$ on $(\setY,\colY)$. 
In contrast to $L_n(X)$, which needs $\mu_X\ll\mu$ and as such  depends on $X$, the  upper bound $U_n$ we obtain is universal 
in the sense of  depending only on the spaces $(\setX,\colX)$ and   $(\setY,\colY,\nu)$ and the distortion function $\rho$ but 
 not on the random variable $X$ per se.  To reflect this aspect, we will write  $U_n$ without the argument $X$.

\begin{thm}\label{thm:singleshot2}
Let   $(\setX,\colX)$  be  a measurable space, let  $(\setY,\colY,\nu)$  be a finite measure space with $\nu$ normalized according to $\nu(\setY)=1$,  and  consider the  distortion function  $\rho\colon \setX\times\setY\to[0,\infty]$. 
Further, let $k\in(0,\infty)$, suppose that  $\nu$ 
is $\rho^{1/k}$-superregular of dimension $m\in(0,\infty)$ with  superregularity
constants  $ b\in(0,\infty)$ and $ \delta_0 \in (0,\infty]$,  and assume  that 
\begin{align}
\beta:=\operatorname{ess}\sup_{\!\!\!\!\!\!\!\!\!\!\!\!\!\!\!\! y\in\setY} \sup_{x\in\setX} \rho^{{1}/{k}}(x,y)<\infty\quad\text{if $\delta_0<\infty$.} \label{eq:bounddist}
\end{align}
Then, every $X$ taking values in $(\setX,\colX)$ satisfies  $V_{n}(X)\leq U_n$ for all  $n\in\naturals$, 
\color{black}
where  
\begin{align}
U_n=\Gamma\mleft(1+\frac{k}{m}\mright)\mleft(\supc n\mright)^{-\frac{k}{m}} \quad\text{if $\beta\leq \delta_0$}
\end{align}
and 
\begin{align}
U_n=\Gamma\mleft(1+\frac{k}{m}\mright)\mleft(\supc n\mright)^{-\frac{k}{m}} +(\beta^k-\delta_0^k)e^{-\supc n\delta_0^m}\quad\text{if $\beta> \delta_0$}.
\end{align}
\end{thm}
\begin{proof}
See Appendix \ref{thm:singleshot2proof}.
\end{proof}
We now employ Theorem~\ref{thm:singleshot2}  to characterize  the  $n\to\infty$ asymptotics of  $V_{n}(X)$  
in terms of the upper $k$-th  quantization dimension  and the upper 
 $k$-th  quantization coefficient. 
 
\begin{cor}\label{cor:DC2}
Under the assumptions of Theorem \ref{thm:singleshot2}, the following statements  hold:
\begin{enumerate}[(ii)]
\item$\overline{D}_k(X)\leq m$. \label{eq:boundD2}
\item  If  ${D}_k(X)=m$, then \label{eq:boundC2} 
\begin{align}
\overline{C}_k(X)\leq \Gamma\mleft(1+\frac{k}{m}\mright)\supc^{-\frac{k}{m}}.
\end{align} 
\end{enumerate}
\end{cor} 
\begin{proof}
Item \ref{eq:boundD2} follows from 
\begin{align}
\overline{D}_k(X)&=\limsup_{n\to\infty}  \frac{ k \log (n)}{\log (1/V_{n}(X))}\\
&\leq \lim_{n\to\infty}  \frac{ k \log (n)}{\log (1/U_{n})}\label{eq:usetheorem2}\\
&=m,
\end{align} 
where \eqref{eq:usetheorem2} is by Theorem \ref{thm:singleshot2}. 
  Item \ref{eq:boundC} is by 
\begin{align}
\overline{C}_k(X)&=\limsup_{n\to\infty} n^\frac{k}{m}V_n(X)\label{eq:coeff12}\\
&\leq \lim_{n\to\infty} n^\frac{k}{m}U_n\label{eq:coeff1a2}\\
&=  \Gamma\mleft(1+\frac{k}{m}\mright)\supc^{-\frac{k}{m}},\label{eq:coeff22}
\end{align} 
where \eqref{eq:coeff12} is by  \eqref{eq:Csup}  combined with the assumption ${D}_k(X)=m$ and   \eqref{eq:coeff1a2} follows from  Theorem \ref{thm:singleshot2}.
\end{proof}

Combining Corollaries \ref{cor:DC}--\ref{cor:DC2}, we get the   following  result on the $k$-th quantization dimension.  


\begin{cor}\label{cor:dimequal}
Let  $X$ be a random variable taking values in  the $\sigma$-finite measure space $(\setX,\colX,\mu)$ with $\mu_X\ll\mu$, let  $(\setY,\colY,\nu)$ be  a finite measure space with $\nu$ normalized according to $\nu(\setY)=1$, and  consider the  distortion function $\rho\colon \setX\times\setY\to[0,\infty]$.   
Further, let  $k,m,c,b\in(0,\infty)$,  $\delta_0\in(0,\infty]$, and suppose that i) $\mu$ is  $\rho^{1/k}$-subregular of dimension $m$ with subregularity constants $c$ and $\delta_0$ and ii)
$\nu$  is $\rho^{1/k}$-superregular  of dimension $m$ with superregularity constants $b$ and $\delta_0$. Finally, assume  that  
\begin{align}
\operatorname{ess}\sup_{\!\!\!\!\!\!\!\!\!\!\!\!\!\!\!\! y\in\setY} \sup_{x\in\setX}\rho^{{1}/{k}}(x,y)<\infty\quad\text{if $\delta_0<\infty$}. \label{eq:bounddistAA} 
\end{align} 
Then, the following properties hold:
\begin{enumerate}[(ii)]
\item \label{item1dimequal}
If  $\lVert \mathrm d \mu_X/\mathrm d \mu\rVert_{\infty}^{(\mu)}<\infty$, then 
\begin{align}{D}_k(X)=m.\end{align}
\item \label{item2dimequal} If i) $\abs{h_\mu(X)}<\infty$ and ii) $\mu(\setX)=1$ for $\delta_0<\infty$,  then   
\begin{align}\underline{\operatorname{dim}}_{R}(X)={D}_k(X)=m.\end{align}
\end{enumerate}
\end{cor}
\begin{proof}
If  $\lVert \mathrm d \mu_X/\mathrm d \mu\rVert_{\infty}^{(\mu)}<\infty$, then  Item \ref{eq:boundD} of 
Corollary \ref{cor:DC} yields $\underline{D}_k(X)\geq m$. Similarly, if i) $\abs{h_\mu(X)}<\infty$ and ii) $\mu(\setX)=1$ for $\delta_0<\infty$,    then   
 Corollary  
\ref{cor:RDD} establishes  $\underline{D}_k(X)\geq m$. Finally,   
 Item \ref{eq:boundD2} of 
Corollary \ref{cor:DC2} yields  $\overline{D}_k(X)\leq m$ in both cases.  Taken together these results establish ${D}_k(X)=m$ in Items \ref{item1dimequal} and \ref{item2dimequal} as desired. The identity $\underline{\operatorname{dim}}_{R}(X)=m$ follows from Corollary \ref{lem:cordimR}, Lemma \ref{lem:RDD}, and ${D}_k(X)=m$.
\end{proof}

Corollary \ref{cor:dimequal} does not hold in general when \eqref{eq:bounddistAA} is not satisfied. The following example,   which is a slight modification of \cite[Example 6.4]{grlu00}, illustrates this by constructing a random variable with ${D}_k(X)=\infty$.


\begin{exa}
For every $\ell\in\naturals$ with $\ell\geq 2$, set $\setI_\ell=[2^\ell,\, 2^{\ell+1})$,  $\setJ_\ell=[(4/3)2^\ell,\, (5/3)2^{\ell})$, and 
\begin{align}
\mu_\ell=\frac{c\,\colL^{1}\vert_{\setJ_\ell}}{\colL^{1}(\setJ_\ell) 2^{k\ell}\ell\log^2(\ell)} 
\end{align}
with $c=(\sum_{\ell=2}^\infty \frac{1}{2^{k\ell}\ell\log^2(\ell)})^{-1}$.  
Set $\setX=\setY=(0,\infty)$  and $\mu=\colL^{1}|_\setX$, and consider the distortion function  $\rho\colon \setX\times\setY\to[0,\infty)$, $\rho(x,y)=|x-y|^k$ with $k\in(0,\infty)$.
Further, let $X$ be of distribution $\mu_X=\sum_{\ell=2}^\infty \mu_\ell$.  
Next, fix  $n\in\naturals$ and  $f\in\setF_n$ arbitrarily and set 
$\setI=\{\ell\in\naturals: \ell\geq 2\ \text{and}\ f([4,\infty))\cap \setI_\ell=\emptyset\}$. Then, we have 
\begin{align}
\opE\mleft[ \abs{f(X)-X}^k\mright]
&= \sum_{\ell=2}^\infty \int_{\setJ_\ell}\abs{f(x)-x}^k\,\mathrm d\mu_\ell(x)\\
&\geq \sum_{\ell\in\setI}\int_{\setJ_\ell} \abs{f(x)-x}^k\,\mathrm d\mu_\ell(x)\\
&\geq \sum_{\ell\in\setI} \frac{2^{k\ell}}{3^k} \mu_\ell(\setJ_\ell)\label{eq:useellI}\\
&= \frac{c}{3^k}\sum_{\ell\in\setI}\frac{1}{\ell\log^2(\ell)}\\
&\geq \frac{c}{3^k}\sum_{\ell=n+2}^\infty\frac{1}{\ell\log^2(\ell)}\\
&\geq  \frac{c}{3^k} \int_{n+2}^\infty \frac{1}{x \log^2(x)}\,\mathrm d x \\
&= \frac{c}{3^k \log(n+2)}, 
\end{align}
where \eqref{eq:useellI} follows from $\setJ_\ell\subseteq \setI_\ell$  and $f([4,\infty))\cap \setI_\ell=\emptyset$  for all $\ell\in\setI$. As  $n\in\naturals$ and $f\in\setF_n$ were arbitrary, we can conclude that 
\begin{align}
V_n(X)\geq  \frac{c}{3^k \log(n+2)} \quad\text{for all $n\in\naturals$,}
\end{align}
which in turn yields  ${D}_k(X)=\infty$. 
\end{exa}

For the particular  case  $\setX\subseteq\setY=\reals^d$   with $\setX$ compact,  distortion function $\rho\colon\setX\times\setY\to[0,\infty]$, $\rho(x,y)=\|x-y\|^k$, where  $\lVert\,\cdot\,\rVert$ is a norm on $\reals^d$ and $k\in [1,\infty)$, and  $\rho^{1/k}$-regular  $\nu$ of dimension $m$ satisfying  $\operatorname{supp}(\nu)=\setX$  and  $\nu(\setX)=1$, combining Corollary  \ref{cor:dimequal} for this $\nu$ and $\mu=\nu|_\setX$  with \cite[Theorem 12.11]{grlu00}
allows us to conclude  that the  $k$-th  quantization  dimension of $X$  with $\mu_X\ll\mu$ equals the Hausdorff dimension of $\mu$, i.e., we have 
\begin{align}\label{eq:dimchainineq}
D_k(X)=m=\dim_\mathrm{H}(\mu)\leq \dim_\mathrm{H}(\setX)
\end{align}
 with   $\dim_\text{H}(\mu)$ as  defined  in \eqref{eq:dimHX}, provided that  $\lVert \mathrm d \mu_X/\mathrm d \mu\rVert_{\infty}^{(\mu)}<\infty$ or 
 $\abs{h_\mu(X)}<\infty$.
  We note that \eqref{eq:dimchainineq}  particularized to 
   $\mu=\mu_X$   recovers  \cite[Theorem 12.18]{grlu00}.   
 The inequality $D_k(X)\leq \dim_\mathrm{H}(\setX)$ becomes  particularly relevant when  the ambient space dimension $d$ is significantly  larger than the Hausdorff dimension of the set $\setX$ containing the data to be compressed (here described by a random variable taking values in  $\setX$).  

While  the upper bound $U_n$ in Theorem~\ref{thm:singleshot2} is sharp enough to establish  ${D}_k(X)=m$ in Corollary \ref{cor:dimequal}, it is, in 
general, too weak  to yield  good upper bounds on $\overline{C}_k(X)$, simply  as   $U_n$  does not depend on the specific random variable  $X$ taking values in  $(\setX,\colX)$.   
We next derive an $X$-dependent upper bound on $V_{n}(X)$  applicable to   $\setX=\setY\subseteq\reals^d$, which requires, however, stronger technical assumptions than those imposed in 
 Theorem~\ref{thm:singleshot2}. 

\begin{thm}\label{thm:upperstrong}
Let   $\setY=\setX\subseteq\reals^d$ be Borel and consider the distortion function $\rho=\omega^k\vert_{\setX\times \setX}$,  where  $k\in(0,\infty)$ and  $\omega$  is a metric on $\reals^d$.  
Suppose  that 
$X$ takes values  in $\setX$    and assume that there exists a  finite measure  $\nu$ on $\setX$   normalized according to $\nu(\setX)=1$  with $\nu\ll\mu_X$    
and   
\begin{align}\label{eq:deffxinfA}
\mleft\lVert \frac{\mathrm d\nu }{\mathrm d\mu_X}\mright \rVert_\infty^{(\mu_X)}<\infty.
\end{align}
Further, let  $\nu$ be $\rho^{1/k}$-superregular of dimension \color{black}$m\in(k,\infty)$ \color{black} with  superregularity
constants  $ b\in(0,\infty)$ and $ \delta_0 \in (0,\infty]$ and  
assume  that  
 \begin{align}
\beta:=\sup_{x,y\in\setX}\omega(x,y)<\infty\quad\text{if $\delta_0<\infty$.} \label{eq:bounddist2}
\end{align}
Then, we have  
\begin{align}\label{limsupVUn}
\limsup_{n\to\infty} n^\frac{k}{m}V_n(X) \leq  
\Omega_{k/m}(X)\Gamma\mleft(1+\frac{k}{m} \mright)b^{-\frac{k}{m}}   
\end{align}
with 
\begin{align}\label{eq:OmegaSigma3}
\Omega_{\alpha}(X)=
\opE\mleft[\mleft( \frac{\mathrm d\nu }{\mathrm d\mu_X} (X)\mright)^{\alpha} \mright]\leq 1\quad\text{for all $\alpha\in(0,1)$}
\end{align}
\color{black}
and strict inequality   in \eqref{eq:OmegaSigma3} unless $\mu_X=\nu$. 
If, in addition,  $D_k(X)=m$, then 
\begin{align}\label{eq:chainCC3} 
\overline{C}_k(X) \leq  \Omega_{k/m}(X) \Gamma\mleft(1+\frac{k}{m} \mright)b^{-\frac{k}{m}}.
\end{align}
\end{thm}
\begin{proof}See Appendix 
\ref{thm:upperstrongproof}.\end{proof}

\color{black}
Note that  \eqref{eq:chainCC3} improves upon the upper bound 
  on $\overline{C}_k(X)$ in  
Item \ref{eq:boundC2} of Corollary \ref{cor:DC2} as  the inequality in \eqref{eq:OmegaSigma3} is strict for $\alpha=k/m<1$ unless  $\mu_X=\nu$.

We close this section by stating  a  technical result needed later in the paper. 

\begin{lem}\label{lem:dimbound}
Let   $(\setX,\colX)$ and  $(\setY,\colY)$  be  measurable spaces and  consider the  distortion function  $\rho\colon \setX\times\setY\to[0,\infty]$. For 
$X$ taking values in  $(\setX,\colX)$ and $m,k\in(0,\infty)$, the following properties hold: 
\begin{enumerate}[(ii)]
\item \label{eq:boundCcheck2} If $\liminf_{n\to\infty} n^\frac{k}{m}V_n(X)>0$, then $\underline{D}_k(X)\geq m$. 
\item \label{eq:boundCcheck} If $\limsup_{n\to\infty} n^\frac{k}{m}V_n(X)<\infty$, then $\overline{D}_k(X)\leq m$.
\end{enumerate}
\end{lem}
\begin{proof}
See Appendix \ref{proof:dimbound}.
\end{proof}

\section{Linking R-D Theory and Quantization}\label{sec:link}
The purpose of this section is to establish, for i.i.d. sequences $(X_i)_{i\in\naturals}$, a relationship between R-D theory and quantization  that is both conceptual and quantitative. 
To this end, recall the setup described in Section \ref{sec:RDIntro},  where $(\setA,\colA)$ and $(\setB,\colB)$ are   measurable spaces 
equipped with the  distortion function  $\sigma\colon \setA\times\setB\to[0,\infty]$,   
 and let $(X_i)_{i\in\naturals}$  be a 
sequence of i.i.d. random variables taking values in $(\setA,\colA)$.   
For a given rate  $R\in[0,\infty)$ and distortion $D\in[0,\infty)$, 
the  pair $(R,D)$ is said to be $\ell$-achievable if
there exists a source code $g_{(\ell)}\colon \setA^\ell \to \setB^\ell$   of length $\ell$ with $\abs{g_{(\ell)}(\setA^\ell)} \leq \lfloor e^{\ell R}\rfloor$ and  expected average distortion  
\begin{align}\label{eq:leqdfin2}
\opE\mleft[\sigma_{(\ell)}\big((X_1,\dots,X_\ell),g_{(\ell)}(X_1,\dots,X_\ell)\big)\mright] \leq D, 
\end{align} 
where 
\begin{align}\label{eq:ghol2}
\sigma_{(\ell)}((x_1,\dots,x_\ell),(y_1,\dots,y_\ell))=\frac{1}{\ell}\sum_{i=1}^\ell\sigma(x_i,y_i) 
\end{align}
is the average distortion function of length $\ell$,  and 
\begin{equation}
R_{(\ell)}(D) = \inf\{R\in[0,\infty):  \text{$(R,D)$ is $\ell$-achievable}\}. 
\label{eq:rdfin3}
\end{equation} 
The  pair $(R,D)$  is said to be achievable if  there exists an $\ell\in\naturals$ such that  $(R,D)$  is $\ell$-achievable. 
Finally, we set   
\begin{equation}
R(D) = \inf\{R\in[0,\infty):  \text{$(R,D)$ is achievable}\}.  
\label{eq:rdfin22}
\end{equation}

First, note that $R<R(D)$ in R-D theory implies $V_{\lfloor e^{R}\rfloor}(X_1)\geq D$ in quantization. 
In fact,  pairs $(R,D)$  with $R<R(D)$ are not  $1$-achievable. As the set of source codes of length $1$ 
with $|f(\setX)|\leq \lfloor e^{R}\rfloor$ is 
$\setF_{\lfloor e^{R}\rfloor}(\setX,\setY)$, we can therefore conclude  that $R<R(D)$  implies  $V_{\lfloor e^{R}\rfloor}(X_1)\geq D$.

Interestingly, under  subregularity for $\delta_0=\infty$, reasoning in 
 the opposite direction  is also possible.  Concretely, one can get a lower bound  on $R(D)$ based on results from quantization. 
This will be accomplished by applying  Theorem \ref{thm:singleshot} to  vectors  $X=(X_1,\dots,X_\ell)$   and then taking the limit $\ell\to\infty$. 
We proceed to establish lower bounds on $R_{(\ell)}(D)$ for all $\ell\in\naturals$. 

\begin{thm}\label{cor:MSiid} 
Let  $(X_i)_{i\in\naturals}$ be an i.i.d. sequence of random variables  taking values in the $\sigma$-finite measure space $(\setA,\colA,\mu)$ satisfying   $\mu_{X_1}\ll\mu$ and $\lVert \mathrm d \mu_{X_1}/\mathrm d \mu\rVert_{p/(p-1)}^{(\mu)}<\infty$ with  $p\in [1,\infty)$, let  $(\setB,\colB)$ be  a measurable space, and  consider the  distortion function $\sigma\colon \setA\times\setB\to[0,\infty]$.    
Further,  let $k\in(0,\infty)$ and suppose   that  $\mu$ is $\sigma^{1/k}$-subregular of dimension $m\in(0,\infty)$ with  subregularity
constants  $c\in(0,\infty)$ and $\delta_0 \in (0,\infty]$.
Then,  we have $ R_{(\ell)} (D) \geq \widetilde R_{(\ell)}(D)$ for all $D \in(0,D_{(\ell)})$, where
\begin{align}\label{eq:MSbound}
\widetilde R_{(\ell)}(D)
=   \frac{m}{pk} \log \bigg(   \frac{\ell m}{(\ell m + pk)D} \bigg)
-  \log (d_{(\ell)})
\end{align}
with 
\begin{align}\label{eq:parsubsup2a}
d_{(\ell)}&=\frac{\Gamma \mleft(1+\frac{m}{pk}\mright)\ell^{\frac{ m}{pk}}}{\Gamma^{\frac{1}{\ell}}\mleft(1+\frac{\ell m}{pk}\mright)}  c^{\frac{1}{p}} 
\lVert\mathrm d \mu_{X_1}/\mathrm d \mu\rVert^{(\mu)}_{p/(p-1)} 
\end{align}
and 
\begin{equation}\label{eq:Dl}
D_{(\ell)} 
=  
\frac{\delta_0^{k}}{\ell}\frac{\ell m}{\ell m+pk}. 
\end{equation} 
Moreover, $\widetilde R_{(\ell)}(D)$
in  \eqref{eq:MSbound} is strictly monotonically decreasing in $\ell$ with 
\begin{align}\label{eq:asymptoticA1}
\lim_{\ell\to\infty} \widetilde R_{(\ell)}(D) =-\log \mleft(\lVert\mathrm d \mu_{X_1}/\mathrm d \mu\rVert^{(\mu)}_{p/(p-1)}\mright)+ F_{m/p,k,c^{1/p}}(D)
\end{align}
for all $D\in(0,\infty)$ 
and  with $F_{m,k,c}(D)$ as defined  in \eqref{eq:F}.  
\end{thm}
\begin{proof}
See Appendix \ref{cor:MSiidproof}.
\end{proof}

 Now, if $\mu$ in Theorem \ref{cor:MSiid} is $\sigma^{1/k}$-subregular  with  $\delta_0=\infty$,  then   $D_{(\ell)} =\infty$ for all $\ell\in\naturals$ so that 
%
%
\begin{align}\label{eq:ellratebound}
R_{(\ell)}(D) \geq   \widetilde R_{(\ell)}(D) \geq-\log \mleft(\lVert\mathrm d \mu_{X_1}/\mathrm d \mu\rVert^{(\mu)}_{p/(p-1)}\mright)+ F_{m/p,k,c^{1/p}}(D)
\end{align}
for all $D\in(0,\infty)$ and $\ell\in\naturals$.
We conclude that $R(D)\geq \widetilde R_{X_1}^\mathrm{L}(D)$ for all $D\in(0,\infty)$ with 
\begin{align}\label{eq:lowerRDiid}
\widetilde R_{X_1}^\mathrm{L}(D)= -\log \mleft(\lVert\mathrm d \mu_{{X_1}}/\mathrm d \mu\rVert^{(\mu)}_{p/(p-1)}\mright)+ F_{m/p,k,c^{1/p}}(D). 
\end{align}

 
 Alternatively,  suppose that the assumptions of  Theorem \ref{thm:new} (applied to $X=X_1$) are satisfied for $\delta_0=\infty$. In particular, we need $\abs{h_\mu({X_1})}<\infty$. Then,  the R-D function $R_{{X_1}}(D)$ is lower-bounded according to  
$R_{{X_1}}(D)\geq R_{{X_1}}^\mathrm{L}(D)$ with (see \eqref{eq:SLB1a})
 \begin{align}\label{eq:lowerRDiid2}
 R_{X_1}^\mathrm{L}(D)= h_\mu({X_1}) + F_{m,k,c}(D)\quad\text{for all $D\in(0,\infty)$}. 
\end{align}
Since $R_{X_1}(D)=R_{(X_i)_{i\in \naturals}}(D)$ thanks to the i.i.d. assumption, the converse \cite[Theorem 7.2.5]{be71} implies that achievability of $(R,D)$ requires $R\geq R_{X_1}(D)$, which in turn yields
$R(D)\geq R_{{X_1}}^\mathrm{L}(D)$  for all $D\in(0,\infty)$. 

   For $p=1$, the lower bound  $\widetilde R_{X_1}^\mathrm{L}(D)$ in \eqref{eq:lowerRDiid}  differs from   $R_{{X_1}}^\mathrm{L}(D)$  in  \eqref{eq:lowerRDiid2}
in the term $-\log \mleft(\lVert\mathrm d \mu_{{X_1}}/\mathrm d \mu\rVert^{(\mu)}_{\infty}\mright)$, which is  replaced by    $h_\mu({X_1})$.  
 In terms of applicability, validity of $R(D)\geq \widetilde R_{X_1}^\mathrm{L}(D)$ for $p=1$ requires $\lVert \mathrm d \mu_{X_1}/\mathrm d \mu\rVert_{\infty}^{(\mu)}<\infty$, whereas 
$R(D)\geq R_{X_1}^\mathrm{L}(D)$ is based on the assumption $\abs{h_\mu({X_1})}<\infty$.  
As illustrated in Examples \ref{exa:one} and \ref{exa:two} below, there
  are cases where  $\lVert \mathrm d \mu_{X_1}/\mathrm d \mu\rVert_{\infty}^{(\mu)}<\infty$ and  $h_\mu({X_1})=\infty$ and vice versa.

\begin{exa}\label{exa:one}
Set $\setX=\setY=[e,\infty)$  and let the random variable  $X$ take values in $\setX$. Further, set  $g(x):=(\mathrm d \mu_X/\mathrm d \colL^1 \vert_\setX)(x)=1/(x \log^2(x))$. It then follows that 
$\lVert g \rVert_{\infty}^{(\colL^1\vert_{\setX})}=1/e<\infty$, but 
\begin{align}
h_{\colL^1 \vert_\setX}(X) &=-\int_e^\infty g(x) \log (g(x))\,\mathrm d x\\
&\geq \int_e^\infty  \frac{1}{x \log(x)} \, \mathrm d x\\
&=\infty. 
\end{align}
\end{exa}

\begin{exa}\label{exa:two}
Set $\setX=\setY=[0,1]$  and let the random variable  $X$ take values in $\setX$. Further, set  $g(x):=(\mathrm d \mu_X/\mathrm d \colL^1 \vert_\setX)(x)=1/(2\sqrt{x})$. It then follows that 
$\lVert g \rVert_{2}^{(\colL^1\vert_{\setX})}=\infty$, which implies $\lVert g \rVert_{\infty}^{(\colL^1\vert_{\setX})}=\infty$, while 
\begin{align}
h_{\colL^1\vert_\setX}(X) 
&= -\int_0^1 g(x) \log (g(x)) \,\mathrm d x\\
&=\log(2)+\frac{1}{2}\int_0^1 \frac{\log (\sqrt{x})}{\sqrt{x}}  \,\mathrm d x\\ 
&=\log(2)+\int_0^1\log (u)  \,\mathrm d u\\
&= \log(2) -1.  
\end{align}
\end{exa}

\section{R-D Theory and Quantization  for Compact Manifolds}\label{sec:manifold}
In this section, we  particularize our results on R-D theory and quantization  to  random variables taking values in compact manifolds, specifically hyperspheres and Grassmannians.  
Hyperspheres are  prevalent in many areas of data science including  spherical quantization \cite{swth83,mahu10,egla19}, hypersphere learning \cite[Section 4]{joneruty19}, and directional statistics  \cite{maju00}.   
Grassmannians find application in, e.g.,  code design \cite{zhts02, he05,daliri08,piwetiho18}, computer vision \cite{tuvech08},  deep neural network theory \cite{huwugo18}, and the completion of low-rank matrices \cite{boab15}.   
 What we need here in order to apply the program developed above are  suitable sub/super-regularity  conditions, which in turn requires  volume estimates of balls.   
For hyperspheres and Grassmannians, these  estimates are   obtained by direct computation.  
For general complete Riemannian manifolds under suitable curvature assumptions,  such volume estimates can be derived  using the Bishop-G\"unther volume bounds \cite[Theorem 3.101]{gahula93} (see also \cite[Section III.A]{he05} 
for an application of this method to  Grassmannians  and Stiefel manifolds).

\subsection{ R-D Theory and Quantization for Hyperspheres} \label{sec:hyper}
We first consider random variables  taking values in the hypersphere $\setS^{d-1}(r)$  
and  derive corresponding  lower bounds on the R-D function and lower and upper bounds on the 
$n$-th quantization error. The bounds on the $n$-th quantization error 
 will allow us  to 
  obtain lower and upper bounds on the lower and upper $k$-th quantization coefficient, respectively. 
Finally, as an example, we evaluate our results for a random variable of von Mises-Fisher distribution. 
 
We start by establishing  subregularity and  superregularity for the measure  $\colH^{d-1}\vert_{\setS^{d-1}(r)}/a^{(d-1)}(r)$. 

\begin{lem}\label{lem:sphere}
Fix  $d\in\naturals\setminus\!\{1\}$ and $r\in(0,\infty)$ and consider the restricted normalized Hausdorff measure $\mu=\colH^{d-1}\vert_{\setS^{d-1}(r)}/a^{(d-1)}(r)$. Further, let   $\setS^{d-1}(r)\subseteq \setY\subseteq\reals^d$  with distortion function  $\rho\colon\setS^{d-1}(r)\times \setY\to [0,\infty)$, $\rho(x,y)=\lVert x-y\rVert^2_{2}$. Then, the following   holds:  

\begin{enumerate}[(ii)]
\item 
The measure $\mu$ satisfies the following family of subregularity conditions: 
\label{itemsubR}
\begin{align}
\mu\mleft(\setB_{\rho^{1/2}}\big(y,\delta\big)\mright)
&\leq   c_{\delta_0}\delta^{d-1}\quad\text{for all $y\in\setY$ and $\delta\in(0,\delta_0]$ } \label{eq:subsphere}
\end{align}
 parametrized by $\delta_0\in (0,r]$, where $ c_{\delta_0}=G(\delta_0^2/r^2)/a^{(d-1)}(r)$ with 
\begin{align} \label{eq:Gfunction}
G(\alpha)=\frac{a^{(d-1)}(1)}{2} \frac{I_{\frac{d-1}{2},\frac{1}{2}} \mleft(\alpha \mright)}{\alpha^\frac{d-1}{2}}
\end{align}
 continuous and strictly monotonically increasing on $(0,1]$ with  
\begin{align}\label{eq:limitGa}
\lim_{\alpha\to 0} G(\alpha)=
\frac{\pi^{\frac{d-1}{2}}}{\Gamma\mleft(\frac{d+1}{2}\mright)}=v^{(d-1)}(1).
\end{align} 
\item  \label{itemsubS}
The measure $\mu$ satisfies the following family of superregularity conditions: 
\begin{align}\label{eq:supernew}
\mu\mleft(\widetilde \setB_{\rho^{1/2}}\big(x,\delta\big)\mright)\geq b_{\delta_0}\delta^{d-1}\quad\text{for all $x\in\setS^{d-1}(r)$ and $\delta\in(0,\delta_0]$ }
\end{align}
 parametrized by $\delta_0\in (0,\sqrt{2}r]$, where 
\begin{align}\label{eq:bsuberS}
b_{\delta_0}=\frac{v^{(d-1)}(1)}{a^{(d-1)}(r)}\mleft(1-\frac{\delta_0^2}{4r^2}\mright)^\frac{d-1}{2}.
\end{align}

\item
We have the following limits: 
\begin{align}\label{eq:limitcb}
c_0:= \lim_{\delta_0\to 0} c_{\delta_0}  =b_0:=  \lim_{\delta_0\to 0} b_{\delta_0}=\frac{v^{(d-1)}(1)}{a^{(d-1)}(r)}.  
\end{align}
\end{enumerate}
\end{lem} 
\begin{proof}
See Appendix \ref{lem:sphereproof}. 
\end{proof}
Note that   \eqref{eq:limitcb} implies that the bounds in \eqref{eq:subsphere} and \eqref{eq:supernew}
become tight for $\delta_0\to 0$. 
  The specific value  for the limiting constant, namely $v^{(d-1)}(1)/a^{(d-1)}(r)= \kappa_{d-1}(\lVert \,\cdot\,\rVert_2)/a^{(d-1)}(r)$,    can be explained using the following deep result on  $\colH^{m}$-rectifiable sets   in geometric measure theory.   
\color{black}

\begin{thm}\cite[Theorem 2.63]{amfupa00}  \label{th:rec}
An  $\colH^m$-measurable set  $\setA\subseteq\reals^d$ with 
$0<\colH^m(\setA)<\infty$ is  $\colH^m$-rectifiable if and only if 
\begin{equation}\label{eq:limc0rec}
\lim_{\delta\to 0} \frac{\colH^m\vert_{\setA}\mleft(\setB_{\lVert \,\cdot\,\rVert}\big(y,\delta\big)\mright)}{\delta^m}=\kappa_m(\lVert \,\cdot\,\rVert)
\quad\text{for $\colH^m$-almost all $y\in\setA$}. 
\end{equation}
\end{thm} 
Intuitively, Theorem  \ref{th:rec}  says that viewed from close up, every $\colH^m$-rectifiable set looks almost everywhere like $\reals^m$. 


In the  following, we fix $d\in\naturals\!\setminus\!\{1\}$, set 
\begin{align}
\mu:=\frac{\colH^{d-1}\vert_{\setS^{d-1}(r)}}{a^{(d-1)}(r)},  
\end{align}
 let $X$ be a random variable taking values in $\setX=\setS^{d-1}(r)\subseteq \setY\subseteq \reals^d$, and  consider the distortion function  
$\rho\colon \setX\times \setY\to[0,\infty)$, $\rho(x,y)=\lVert x-y\rVert_2^2$.


We first  derive a lower bound on the R-D function  under the assumptions $\mu_X\ll\mu$ and  
 $h_\mu(X)>-\infty$. 
 Before starting in earnest, we note that 
Jensen's inequality  \cite[Theorem 2.3]{peprto92} combined with $\mu(\setX)=1$ yields $h_\mu(X) \leq \log (\mu(\setX))=0$, which in turn implies  
  $\abs{h_\mu(X)}<\infty$. 
Fix $\alpha\in(0,1/2)$  
 and set, for every $D\in(0,\infty)$, $\delta_D=\min\{D^{\alpha}, r \}$ and 
\begin{align}
c_{\delta_D} =\frac{1}{2} \frac{I_{\frac{d-1}{2},\frac{1}{2}} \mleft(\frac{\delta_D^2}{r^2} \mright)}{\delta_D^{d-1}}. 
\end{align}
Application of Theorem \ref{thm:new} with $k=2$ and $m=d-1$ 
combined with Item \ref{itemsubR}  of  Lemma \ref{lem:sphere} for $\delta_0=\delta_{D}$
then  yields 
$R_X^{\text{SLB}}(D)\geq R_X^{\text{L}}(D)$ for all $D\in(0,\infty)$ with       
\begin{align}
R_X^{\text{L}}(D)
&=h_\mu(X)-\frac{d-1}{2}\label{eq:slb1}\\
 &\ \ \ -\log \mleft(c_{\delta_D}\mleft(\frac{d-1}{2D}\mright)^{- \frac{d-1}{2}}\Gamma\mleft(\frac{d+1}{2}\mright) + e^{-\frac{(d-1)\delta_D^2}{2D}}\mright).  \label{eq:slb2}
\end{align}
Moreover, since  $\delta_D\to0$ as $D\to0$ and  $e^{-(d-1)\delta_{D}^2/(2D)}\to 0$  as $D\to 0$, we have     
\begin{align}
\lim_{D\to0}(R_X^{\text{L}}(D)- h_{\mu}(X) -F_{d-1,2,c_0}(D))=0, 
\end{align} 
where 
 we used \eqref{eq:limitcb} to conclude that 
\begin{align}\label{eq;divdA}
c_0=\frac{v^{(d-1)}(1)}{a^{(d-1)}(r)}.
\end{align}
Thus, as $D\to 0$, the lower bound $R_X^{\text{L}}(D)$ approaches the classical Shannon lower bound  $R_W^{\text{SLB}}(D)$ in \eqref{eq:shlbclassic} for a continuous random variable $W$ taking values in  $\reals^{d-1}$ and of differential entropy 
\begin{equation} 
h_{\colL^{d-1}}(W) = \log (a^{(d-1)}(r)) +h_\mu(X) = h_{\tilde \mu}(X)
\end{equation}
with $\tilde \mu=a^{(d-1)}(r)\mu=\colH^{d-1}\vert_{\setS^{d-1}(r)}$. This reflects the fact that from close up the hypersphere $\setS^{d-1}(r)$ looks like $\reals^{d-1}$ (see Theorem  \ref{th:rec}).

Next, we derive a lower bound on the $n$-th quantization error $V_{n}(X)$    under the assumptions $\mu_X\ll\mu$  and  
\begin{align}
\Sigma_p(X):=\mleft\lVert \frac{\mathrm d \mu_X}{\mathrm d \mu}\mright\rVert_{p/(p-1)}^{(\mu)}<\infty \quad\text{with $p\in[1,\infty).$}\label{eq:gp}
\end{align}    
 To this end, for  every $n\in\naturals$ with $n\geq n_0:=2^{1/p}/\Sigma_p(X)$, we set 
\begin{align}
\delta_{n}
&= r \sqrt{I_{\frac{d-1}{2},\frac{1}{2}}^{-1}\mleft(\frac{2}{n^p \Sigma^p_p(X)}\mright)}   \label{eq:rhon} 
\end{align}
and 
\begin{align}
c_{\delta_n} =\frac{1}{2} \frac{I_{\frac{d-1}{2},\frac{1}{2}} \mleft(\frac{\delta_n^2}{r^2} \mright)}{\delta_n^{d-1}}. 
\end{align}
These  choices for  $\delta_n$  and $c_{\delta_n}$ guarantee that   $\delta_{n}\leq r$  and 
\begin{align}
\delta_n= c_{\delta_n}^{-\frac{1}{d-1}} \Sigma_p^{-\frac{p}{d-1}}(X)\, n^{-\frac{p}{d-1}} \quad\text{ for all  $n\geq n_0.$}
\end{align}  
Application of Theorem \ref{thm:singleshot} with $k=2$ and $m=d-1$ combined with Item \ref{itemsubR}  of  Lemma \ref{lem:sphere} for $\delta_0=\delta_{n}$
then  yields  $V_n(X)\geq   L_n(X)$ for all $n\ge n_0$ with 
\begin{align}
  L_n(X)\label{eq:lowerboundLn}
&= \frac{d-1}{d-1+2p}r^2\, I_{\frac{d-1}{2},\frac{1}{2}}^{-1}\mleft(\frac{2}{n^p \Sigma_p^p(X)}\mright).
\end{align} 
This allows us to establish  
\begin{align} \label{eq:limsumS} 
 \liminf_{n\to\infty} n^\frac{2p}{d-1}V_{n}(X)
\geq  \frac{d-1}{d-1+2p}\, r^2\, \Sigma_p^{-\frac{2p}{d-1}}(X)\, k_d >0
\end{align}
with 
\begin{align} \label{eq:Kd}
k_d=\mleft(\frac{2\sqrt{\pi}\,\Gamma\mleft(\frac{d+1}{2}\mright)}{\Gamma\mleft(\frac{d}{2}\mright)}\mright)^\frac{2}{d-1}.  
\end{align} 
Indeed, we have 
\begin{align}
 \liminf_{n\to\infty} n^\frac{2p}{d-1}V_{n}(X)
 &\geq  \lim_{n\to\infty} n^\frac{2p}{d-1}L_{n}(X)\\
 &=   \frac{d-1}{d-1+2p}\,r^2\, \lim_{n\to\infty} n^\frac{2p}{d-1}I_{\frac{d-1}{2},\frac{1}{2}}^{-1}\mleft(\frac{2}{n^p \Sigma_p^p(X)}\mright)\\
 &= \frac{d-1}{d-1+2p}\, r^2 \,\Sigma_p^{-\frac{2p}{d-1}}(X) \mleft(\lim_{\alpha\to 0}\frac{2\alpha^\frac{d-1}{2}}{I_{\frac{d-1}{2},\frac{1}{2}}(\alpha)}\mright)^\frac{2}{d-1},\label{eq:alpha} 
\end{align}
where in \eqref{eq:alpha} we set 
\begin{align}
\alpha=I_{\frac{d-1}{2},\frac{1}{2}}^{-1}\mleft(\frac{2}{n^p \Sigma^p_p(X)}\mright).
\end{align} Using  L'H\^opital's rule, we obtain  
\begin{align}\label{eq:alphalim}
\lim_{\alpha\to 0}\frac{2\alpha^\frac{d-1}{2}}{I_{\frac{d-1}{2},\frac{1}{2}}(\alpha)}
&= \frac{2\sqrt{\pi}\,\Gamma\mleft(\frac{d+1}{2}\mright)}{\Gamma\mleft(\frac{d}{2}\mright)}, 
\end{align} 
which when inserted  in \eqref{eq:alpha} establishes \eqref{eq:limsumS}.

We next derive an upper bound $U_n$ on the $n$-th quantization error $V_{n}(X)$.  
Fix $\alpha\in(0,1)$  
 and set, for  every $n\in\naturals\!\setminus\!\{1\}$, $\delta_{n}=\sqrt{2}r n^{-{\alpha}/(d-1)}$ 
and 
\begin{align}
b_{\delta_{n}}=\frac{\Gamma\mleft(\frac{d}{2}\mright)}{2\sqrt{\pi}r^{d-1}\Gamma\mleft(\frac{d+1}{2}\mright)}\mleft(1-\frac{\delta_{n}^2}{4r^2}\mright)^\frac{d-1}{2}. 
\end{align}
As
\begin{align}
\sup_{x,y\in\setS^{d-1}(r)} \lVert x-y\rVert_2^2 = 4r^2,
\end{align}
application of Theorem \ref{thm:singleshot2} with $\setX=\setS^{d-1}(r)$\color{black}, $\nu=\mu$, $k=2$, and $m=d-1$  combined with Item \ref{itemsubS}  of  Lemma \ref{lem:sphere} for $\delta_0=\delta_{n}$
then yields 
$V_n(X)\leq U_n$ for all  $n\in\naturals$ with  
\begin{align}\label{eq:Un}
U_n=\Gamma\mleft(\frac{d+1}{d-1} \mright)\mleft(\supc_{\delta_{n}} n\mright)^{-\frac{2}{d-1}} +\mleft(4r^2-\delta_{n}^2\mright)e^{- \supc_{\delta_{n}} n\, \delta_{n}^{d-1}}.   
\end{align} 
Next, note that $\lim_{n\to\infty} \delta_{n}=0$ and  $\lim_{n\to\infty} n^\frac{2}{d-1}e^{- \supc_{\delta_{n}} n\, \delta_{n}^{d-1}} =0$. 
Therefore, $V_n(X)\leq U_n$ implies    
 \begin{align}
 \limsup_{n\to\infty} n^\frac{2}{d-1}V_{n}(X)
 &\leq   \Gamma\mleft(\frac{d+1}{d-1}\mright)r^2k_d <\infty\label{eq:laststepC2}
\end{align}   
with $k_d$  as defined  in \eqref{eq:Kd}. 
Combining \eqref{eq:limsumS} and \eqref{eq:laststepC2}  with 
Items \ref{eq:boundCcheck2} and \ref{eq:boundCcheck} of Lemma \ref{lem:dimbound}, respectively,   
now implies  
\begin{align}\label{eq:Dk}
\frac{d-1}{p}\leq\underline{D}_2(X)\leq \overline{D}_2(X)=d-1. 
\end{align}
If  $p=1$, then the $2$-nd quantization dimension  ${D}_2(X)$ exists and  equals  the geometric dimension $d-1$ of the hypersphere.
Similarly, if  $h_\mu(X)>-\infty$ (Jensen's inequality  \cite[Theorem 2.3]{peprto92} combined with $\mu(\setX)=1$ yields $h_\mu(X) \leq \log (\mu(\setX))=0$), 
then Item \ref{item2dimequal} of Corollary \ref{cor:dimequal} (with  $\setX=\setS^{d-1}(r)$ and $\nu=\mu$)  allows us to conclude that 
\begin{align}\underline{\operatorname{dim}}_R(X)={D}_2(X)=d-1.\end{align} 
In particular, we have the following results on the upper and lower $2$-nd quantization coefficient.  If ${D}_2(X)=d-1$, then   
by  \eqref{eq:laststepC2}, we have    
\begin{align}\label{eq:chain1}
 \overline{C}_2(X) \leq   \Gamma\mleft(\frac{d+1}{d-1}\mright) r^2 k_d<\infty,  
\end{align}
and if  \eqref{eq:gp} holds for $p=1$, then  
 \eqref{eq:limsumS}    yields 
\begin{align}\label{eq:chain1a}
\underline{C}_2(X)\geq  \frac{d-1}{d+1} r^2\,   \Sigma_1^{-\frac{2}{d-1}}(X)\, k_d>0.   
\end{align}
For $\setY=\setS^{d-1}(r)$,  if in addition to ${D}_2(X)=d-1$, we have $d\geq 4$,  $\mu\ll\mu_X$, and $\lVert \mathrm d\mu /\mathrm d\mu_X\rVert_\infty^{(\mu_X)}<\infty$, 
then we can apply Theorem \ref{thm:upperstrong}  with $\nu=\mu$, $k=2$, and  $m=d-1$  
to obtain the improved upper bound 
\begin{align}\label{eq:Cupperultra}
 \overline{C}_2(X)\leq  \Omega_{2/(d-1)}(X) \Gamma\mleft(\frac{d+1}{d-1}\mright)  r^2 k_d,   
\end{align}
where 
\begin{align}\label{eq:OmegavM}
\Omega_{2/(d-1)}(X)
&=\opE\mleft[\mleft(\frac{\mathrm d \mu}{\mathrm d\mu_X}(X)\mright)^{\frac{2}{d-1}}\mright]\leq 1\quad\text{for all $d\geq 4$}    
\end{align}
with strict inequality in \eqref{eq:OmegavM} unless  $\mu_X=\mu$.  


In order to endow the results  just obtained with a more specific, and, in particular,  quantitative flavor, we 
 evaluate $h_\mu(X)$, $\Sigma_{1}(X)$, and $\Omega_{2/(d-1)}(X)$ for  $X$  of  von Mises-Fisher distribution. 
\begin{exa} 
Fix $d\in\naturals\!\setminus\!\{1\} $ and let   $X$ be a random variable of  von Mises-Fisher distribution $\mu_X$ 
 with mean direction $y\in\setS^{d-1}(1)$ and concentration parameter $\kappa\in(0,\infty)$, which is determined according to  \cite[Equation (9.3.4)]{maju00}
\begin{align}\label{eq;RNdervM}
\frac{\mathrm d \mu_X}{\mathrm d\mu}(x)= c_d(\kappa) e^{\kappa \tp{y}x}, 
\end{align}
where $\mu=\colH^{d-1}\vert_{\setS^{d-1}(1)}/a^{(d-1)}(1)$ and 
\begin{align}
c_d(\kappa) 
&:= \mleft(\int e^{\kappa \tp{y}x} \,\mathrm d\mu(x)\mright)^{-1}\label{eq:cd1}\\
&= \frac{\kappa^{\frac{d}{2}-1}}{\Gamma\mleft(\frac{d}{2}\mright) 2^{\frac{d}{2}-1} I_{\frac{d}{2}-1}(\kappa)} \label{eq:cd2}  
\end{align}
with 
\begin{align}
I_{\alpha}(\kappa)=\frac{1}{2\pi} \int_{0}^{2\pi} \cos(\alpha t)e^{\kappa \cos(t)}\,\mathrm d t. 
\end{align}
This distribution   plays an important role in directional statistics \cite[Section 9.3.2]{maju00}. 
Specifically, it is one of the simplest parametric distributions on  $\setS^{d-1}(1)$ and has an entropy-maximizing property akin to that of the  multivariate Gaussian distribution on  $\reals^d$. Concretely,  among all random variables $Z$ taking values in $\setS^{d-1}(1)$, of distribution $\mu_Z\ll\mu=\colH^{d-1}\vert_{\setS^{d-1}(1)}/a^{(d-1)}(1)$, and with $\opE[Z]$ fixed, the von Mises-Fisher distribution with $y$ and $\kappa$ determined by 
 \cite[Equation (9.3.7)]{maju00}
\begin{align}\label{eq:meanx}
\opE\mleft[X\mright] = y \frac{I_{\frac{d}{2}}(\kappa)}{I_{\frac{d}{2}-1}(\kappa)} 
\end{align} 
 maximizes $h_\mu(Z)$ \cite[Section 2.3]{ma75}. 
 The  generalized  entropy $h_\mu(X)$ of $X$ with von Mises-Fisher distribution 
 can be derived as follows.  First note that 
\begin{align}
h_\mu(X)
&=-\opE\mleft[ \log \mleft(\frac{\mathrm d\mu_X}{\mathrm d\mu}(X)\mright) \mright]\\
&= -\log (c_d(\kappa))-\kappa \opE\mleft[  \tp{y}X \mright].  \label{eq:expeval} 
\end{align}
Now,  \eqref{eq:meanx} together with $y\in\setS^{d-1}(1)$ implies 
\begin{align}
\opE\mleft[ \tp{y} X\mright] = \frac{I_{\frac{d}{2}}(\kappa)}{I_{\frac{d}{2}-1}(\kappa)}. \label{eq:meanxy} 
\end{align}
Using \eqref{eq:meanxy}
in \eqref{eq:expeval} results in 
\begin{align}\label{eq:entVM}
h_\mu(X)= -\log (c_d(\kappa))-\kappa \frac{I_{\frac{d}{2}}(\kappa)}{I_{\frac{d}{2}-1}(\kappa)}. 
\end{align}
Moreover, 
\begin{align}
\Sigma_{1}(X)&= \mleft\lVert \frac{\mathrm d \mu_X}{\mathrm d \mu}\mright\rVert_{\infty}^{(\mu)}\\
&= c_d(\kappa)\sup_{x\in \setS^{d-1}(1)} e^{\kappa \tp{y}x}\\
&=c_d(\kappa)e^{\kappa},  
\end{align}
which implies (see \eqref{eq:chain1a})
\begin{align}
\underline{C}_2(X)\,\geq \,\frac{d-1}{d+1} r^2\,  \mleft(c_d(\kappa)e^{\kappa}\mright)^{-\frac{2}{d-1}}  \, k_d\, > \, 0.  
\end{align}
Finally, if $d\geq4$, then 
\begin{align}\label{eq:OmegavMaaa22}
\Omega_{2/(d-1)}(X)
&=\opE\mleft[\mleft(\frac{\mathrm d \mu}{\mathrm d\mu_X}\mright)^{\frac{2}{d-1}}(X)\mright]\\
&=  \int \mleft(\frac{\mathrm d\mu_X}{\mathrm d\mu}\mright)^{\frac{d-3}{d-1}}(x)\mathrm\, d\mu(x)\\
&=\frac{c_d^\frac{d-3}{d-1}(\kappa)}{c_d\mleft(\kappa\frac{d-3}{d-1} \mright)}.
\end{align}
\end{exa}

 We conclude  this discussion  by noting that   the lower bound on  $\underline{C}_2(X)$ in  \eqref{eq:chain1a}  evaluated for $d=2$ and $\mu_X=\mu$, i.e., $X$  uniformly distributed on the circle of radius $r$, is sharp enough to establish that $C_2(X)=r^2\pi^2/3$, which is shown in the following example. 

\begin{exa}
Let $X$ be a random variable with uniform distribution $\mu_X$ on the circle of radius $r$. 
Using $d=2$ and $\mu_X=\mu$ in \eqref{eq:chain1a} yields
\begin{align}\label{eq:laststep11}
 \underline{C}_2(X) \geq  \frac{r^2\pi^2}{3}.    
\end{align}
To find a matching upper bound on $\overline{C}_2(X)$, we first derive, for every $n\in\naturals$, an upper bound $U_{n}(X)$ on the $n$-th quantization error $V_{n}(X)$ for $\setY=\setS^{1}(r)$.
Since   more flexibility in placing the quantization points can only reduce $V_{n}(X)$,  
this upper bound applies to $\setS^{1}(r)\subseteq\setY\subseteq\reals^2$ as well. 
Concretely, for every $n\in\naturals$ and $x\in\setS^1(r)$, we set $f_n(x)=a_i$, where $i=\operatorname{argmin}_{j\in\{1,\dots,n\}}\|x-a_j\|_2$ and $a_j=r(\cos(2\pi j/n),\,  \sin(2\pi j/n))$ for $j=1,\dots,n$.  %
This yields
\begin{align}
V_{n}(X)&\leq\opE[\|X-f_{n}(X)\|_2^2]\label{eq:upbound1b}\\
&=\frac{2r^2n}{ \pi}\int_0^\frac{ \pi }{n}\mathrm (1-\cos(\alpha))\,\mathrm d\alpha \label{upbound1a}\\
&=2r^2\mleft(1-\sinc\mleft(\frac{1}{n}\mright)\mright)\quad\text{for all $n\in\naturals$,} \label{upbound1}
\end{align}
where \eqref{upbound1a} follows from the formula for the chord length corresponding to a circle segment of central angle $\alpha$.  
We thus have 
\begin{align}
\limsup_{n\to\infty}n^2V_{n}(X)
&\leq 2r^2\lim_{n\to\infty}n^2\mleft(1-\sinc\mleft(\frac{1}{n}\mright)\mright)\label{eq:asy1a}\\ 
&=
2r^2\pi^2 \lim_{\varepsilon\to 0}\frac{\varepsilon-\sin (\varepsilon)}{\varepsilon^3}\label{eq:anpi}\\
&=\frac{r^2\pi^2}{3}, \label{eq:lop}
\end{align}
where in \eqref{eq:asy1a} we   used  \eqref{eq:upbound1b}--\eqref{upbound1}, 
in \eqref{eq:anpi} we substituted 
$\varepsilon=\pi/n$,  
and in \eqref{eq:lop} we applied L'H\^opital's rule  three times.  As $D_2(X)=1$, see \eqref{eq:Dk} for $d=2$ and $p=1$, 
by \eqref{eq:asy1a}--\eqref{eq:lop}   we get  
\begin{align}\label{eq:laststep22}
 \overline{C}_2(X) \leq  \frac{r^2\pi^2}{3}.    
\end{align}
Combining    \eqref{eq:laststep11} and  \eqref{eq:laststep22}  yields the desired result  $C_2(X)=r^2\pi^2/3$. 
\end{exa}

\subsection{R-D Theory and Quantization   for Grassmannians} \label{sec:grass}
We now consider random variables taking values in a Grassmannian.  
Specifically, we  derive corresponding  lower bounds on the R-D function and lower and upper bounds on the 
$n$-th quantization error.  
We start with preparatory material on Grassmannians  largely following the exposition in 
 \cite[Sections II-III]{daliri08}.  
For  $\mathbb{F}\in\{\mathbb{R},\mathbb{C}\}$  and 
 $r,d\in\naturals$ with $1\leq r\leq d$, let  %
$\setG^\mathbb{F}(r,d)$ denote the Grassmannian  consisting of all $r$-dimensional subspaces of $\mathbb{F}^d$ and 
designate by  $\gamma_{r,d}$  the unique uniformly distributed Borel regular measure on $\setG^\mathbb{F}(r,d)$ with $\gamma_{r,d}(\setG^\mathbb{F}(r,d))=1$.  
The dimension of $\setG^\mathbb{F}(r,d)$ is given by $\beta r(d-r)$ with $\beta=1$ if $\mathbb{F}=\mathbb{R}$ and $\beta=2$ if $\mathbb{F}=\mathbb{C}$. 
For  Grassmannians $\setG^\mathbb{F}(r,d)$ and $\setG^\mathbb{F}(s,d)$ with   $r,s,d\in\naturals$ and  $1\leq r,s\leq d$, 
the  chordal distance $\rho_\mathrm{c}$ is defined according to 
\begin{align}\label{eq:chordaldist}
\rho_\mathrm{c}
\colon \setG^\mathbb{F}(r,d)\times \setG^\mathbb{F}(s,d)&\to [0,\sqrt{\min\{r,s\}}]\\
(x,y)&\mapsto \sqrt{\sum_{i=1}^{\min\{r,s\}} \sin^2(\theta_i(x,y))}, 
\end{align}
where   $\theta_1(x,y),\dots,\theta_{\min\{r,s\}}(x,y)$ are the principal angles between  the subspaces $x\in\setG^\mathbb{F}(r,d)$ and $y\in\setG^\mathbb{F}(s,d)$. 
 The   chordal distance can  now be used to state the following sub/super-regularity conditions for the measures  $\gamma_{r,d}$ on $\setG^\mathbb{F}(r,d)$ and $\gamma_{s,d}$ on $\setG^\mathbb{F}(s,d)$. 

\begin{lem}\cite[Equation (6) and Corollaries 1 and 2]{daliri08}\label{lem:Grassmann} 
Consider the Grassmannians $\setX=\setG^\mathbb{F}(r,d)$ and $\setY=\setG^\mathbb{F}(s,d)$ with $1\leq r, s\leq d$.   
Set $a=\min\{r,s\}$, $b=\max\{r,s\}$,  $m_\mathrm{G}=\beta a(d-b)$, and 
\begin{align}
c_{a,b,d,\beta}=
\begin{cases}
\frac{1}{\Gamma\mleft(\frac{\beta}{2}a(d-b)+1\mright)}\prod_{i=1}^{a}\frac{\Gamma\mleft(\frac{\beta}{2}(d-i+ 1)\mright)}{\Gamma\mleft(\frac{\beta}{2}(b-i+1)\mright)}&\quad\text{if $a+b\leq d$}\\[3mm]
\frac{1}{\Gamma\mleft(\frac{\beta}{2}a(d-b)+1\mright)}\prod_{i=1}^{d-b}\frac{\Gamma\mleft(\frac{\beta}{2}(d-i+ 1)\mright)}{\Gamma\mleft(\frac{\beta}{2}(d-a-i+1)\mright)}&\quad\text{else,}
\end{cases}
\end{align}
where $\beta=1$ if $\mathbb{F}=\mathbb{R}$ and $\beta=2$ if $\mathbb{F}=\mathbb{C}$. 
Then, we have 
\begin{align}\label{eq:defV}
v^{(d)}_{r,s}(\delta) :=\gamma_{r,d}\mleft(\setB_{\rho_\mathrm{c}}(y,\delta)\mright) =   \gamma_{s,d}\mleft(\widetilde \setB_{\rho_\mathrm{c}}(x,\delta)\mright) 
\end{align}
for all $x\in\setX$ and $y\in\setY$ 
and the following  holds: 
\begin{enumerate}[(iii)]
\item \label{item0Grass} If ($\mathbb{F}=\mathbb{R}$  and $b=a+1$) or ($\mathbb{F}=\mathbb{C}$  and $b=a$), then 
\begin{align}
v^{(d)}_{r,s}(\delta) =c_{a,b,d,\beta}\delta^{m_\mathrm{G}}\quad \text{for all $\delta\in(0,1]$.}  
\end{align}

\item \label{item1Grass} If $\mathbb{F}=\mathbb{R}$ and $a=b$, then  
\begin{align}
c_{a,a,d,1}\,\delta^{m_\mathrm{G}}   \leq v^{(d)}_{r,s}(\delta)  \leq \frac{c_{a,a,d,1}}{(1-\delta_0^2)^{\frac{a}{2}}}\,\delta^{m_\mathrm{G}}
\end{align}
 for all $\delta \in(0,\delta_0]$ parametrized by $\delta_0\in (0,1]$.   
\item \label{item2Grass}If  ($\mathbb{F}=\mathbb{R}$ and $b\notin\{a,a+1\}$) or ($\mathbb{F}=\mathbb{C}$ and $a\neq b$), then 
\begin{align}
 c_{a,b,d,\beta}(1-\delta_0^2)^{\frac{\beta}{2}a(b-a+1)-a}\,\delta^{m_\mathrm{G}}     \leq   v^{(d)}_{r,s}(\delta)  \leq c_{a,b,d,\beta}\,\delta^{m_\mathrm{G}}  
\end{align}
for all $\delta \in(0,\delta_0]$ parametrized by $\delta_0\in (0,1]$. 
\end{enumerate}
\end{lem}

In the following, fix $r,s,d\in \naturals$ with  $1\leq r, s\leq d$  and let  $\setX=\setG^\mathbb{F}(r,d)$ and $\setY=\setG^\mathbb{F}(s,d)$  be equipped with 
the distortion function $\rho(x,y)=\rho^2_\mathrm{c}(x-y)$. Further, set   $a=\min\{r,s\}$, $b=\max\{r,s\}$, and 
 $m_\mathrm{G}=\beta a(d-b)$, again with $\beta=1$ if $\mathbb{F}=\mathbb{R}$ and $\beta=2$ if $\mathbb{F}=\mathbb{C}$. 
 Note that, unless $r=s$,  $\setX$ and $\setY$ constitute different manifolds. If $r=s$, then  $m_\mathrm{G}$ equals the  dimension of the Grassmannian  $\setX$.  
 Further,  
let $X$ be a random variable taking values in   $\setX$. 

We proceed to derive a lower bound on the R-D function  under the assumptions $\mu_X\ll\gamma_{r,d}$ and  $h_{\gamma_{r,d}}(X)>-\infty$.  To this end, we start by noting that Jensen's inequality  \cite[Theorem 2.3]{peprto92} combined with $\gamma_{r,d}(\setX)=1$ yields $h_{\gamma_{r,d}}(X)\leq \log(\gamma_{r,d}(\setX))=0 $, which implies $\abs{h_{\gamma_{r,d}}(X)}<\infty$.

Suppose first that $\mathbb{F}=\mathbb{R}$ and $a=b$. 
Fix $\alpha\in(0,1/2)$ and set, for every $D\in(0,1)$, $\delta_D=D^{\alpha}$ and 
\begin{align}
d_{\delta_D}=
\frac{c_{a,a,d,1}}{(1-\delta_D^2)^{\frac{a}{2}}}. 
\end{align}
Application of Theorem \ref{thm:new} for $k=2$, $m=m_\mathrm{G}$, and $\mu=\gamma_{r,d}$ combined with Item \ref{item1Grass} of 
Lemma \ref{lem:Grassmann}
 then yields 
$R_X^{\text{SLB}}(D)\geq R_X^{\text{L}}(D)$ for all $D\in(0,1)$ with 
\begin{align}
R_X^{\text{L}}(D)
=h_\mu(X)-\frac{m_\mathrm{G}}{2}
   -\log \mleft(d_{\delta_D}\mleft(\frac{m_\mathrm{G}}{2D}\mright)^{- \frac{m_\mathrm{G}}{2}}\Gamma\mleft(1+ \frac{m_\mathrm{G}}{2}\mright) + e^{-\frac{m_\mathrm{G}\delta_D^2}{2D}}\mright).
\end{align}
Moreover, since   $\lim_{D\to0}\delta_D=0$ and $\lim_{D\to0}e^{- m_\mathrm{G} \delta_{D}^2/(2D)}= 0$, we have 
\begin{align}
\lim_{D\to0}(R_X^{\text{L}}(D)- h_{\mu}(X) -F_{m_\mathrm{G},2,c_{a,a,d,1}}(D))=0
\end{align}
 with $F_{m,k,d}(D)$  as defined  in \eqref{eq:F}.  

Next, suppose that $\mathbb{F}=\mathbb{C}$ or $a\neq b$.  
Application of Theorem \ref{thm:new} with $k=2$, $m=m_\mathrm{G}$, $\mu=\gamma_{r,d}$, and $\delta_0=1$ combined with Items \ref{item0Grass} and \ref{item2Grass} of 
Lemma \ref{lem:Grassmann}
 then yields  
$R_X^{\text{SLB}}(D)\geq R_X^{\text{L}}(D)$ for all $D\in(0,1)$ with
\begin{align}
R_X^{\text{L}}(D)
&=h_{\gamma_{r,d}}(X)-\frac{m_\mathrm{G}}{2} -\log \mleft(c_{a,b,d,\beta}\mleft(\frac{m_\mathrm{G}}{2D}\mright)^{- \frac{m_\mathrm{G}}{2}}\Gamma\mleft(1+ \frac{m_\mathrm{G}}{2}\mright) + e^{-\frac{m_\mathrm{G}}{2D}}\mright).
\end{align}
Moreover, we have  
\begin{align}
\lim_{D\to0}(R_X^{\text{L}}(D)- h_{\gamma_{r,d}}(X) -F_{m_\mathrm{G},2,c_{a,b,d,\beta}}(D))=0.
\end{align}

Next, we derive a lower bound on the $n$-th quantization error $V_{n}(X)$ under the assumptions $\mu_X\ll\gamma_{r,d}$  and  
\begin{align}
\Sigma_p(X):=\mleft\lVert \frac{\mathrm d \mu_X}{\mathrm d \gamma_{r,d}}\mright\rVert_{p/(p-1)}^{(\gamma_{r,d})}<\infty \quad\text{with $p\in[1,\infty).$}\label{eq:gp2}
\end{align}

Suppose first that  $\mathbb{F}=\mathbb{R}$ and $a=b$ and  
consider the strictly monotonically increasing function 
\begin{align}
h\colon [0,1)&\to [0,\infty)\\
u& \mapsto \frac{c_{a,a,d,1} u^{m_\mathrm{G}}}{(1-u^2)^\frac{a}{2}}.\label{eq:htt} 
\end{align}
As   $h(0)=0$ and $\lim_{u\to1}h(u)=\infty$, 
$h$ is bijective. 
For every $n\in\naturals$, set  \begin{align}\label{eq:deltanGr}
\delta_{n}=h^{-1}\mleft(\frac{1}{n^p\Sigma_p^p(X)}\mright)
\end{align} 
and 
\begin{align}
c_{\delta_n}=\frac{c_{a,a,d,1}}{(1-\delta_n^2)^{\frac{a}{2}}}=h(\delta_n) \delta_n^{-m_\text{G}}. 
\end{align}
These   choices for  $\delta_n$  and $c_{\delta_n}$ ensure   $\delta_{n}< 1$  and 
\begin{align}
\delta_n=
c_{\delta_n}^{-\frac{1}{m_\text{G}}} \Sigma_p(X)^{-\frac{p}{m_\text{G}}} n^{-\frac{p}{m_\text{G}}} \quad\text{for all  $n\in\naturals$.}
\end{align} 
Application of Theorem \ref{thm:singleshot} with $k=2$, $m=m_\mathrm{G}$, and $\mu=\gamma_{r,d}$ combined with Item \ref{item1Grass} of Lemma \ref{lem:Grassmann} 
then  yields  $V_n(X)\geq   L_n(X)$ for all $n\in\naturals$ with 
\begin{align}
 L_n(X)\label{eq:lowerboundLnG}
&= \frac{m_\mathrm{G}}{m_\mathrm{G}+2p}\mleft(h^{-1}\mleft(\frac{1}{n^p\Sigma_p^p(X)}\mright)\mright)^2. 
\end{align} 
This lower bound can   be further simplified as follows. 
First, note that 
\begin{align}
\frac{1}{n^p\Sigma_p^p(X)}=h(\delta_n)\geq c_{a,a,d,1}\delta_n^{m_\mathrm{G}} 
\end{align}
implies $\delta_n\leq (n^p\,\Sigma_p^p(X)\,c_{a,a,d,1})^{-1/{m_\mathrm{G}}}$ so that 
\begin{align}
\frac{1}{n^p\Sigma_p^p(X)}=h(\delta_n)&=  \frac{c_{a,a,d,1} \delta_n^{m_\mathrm{G}}}{(1-\delta_n^2)^\frac{a}{2}}\\
&\leq \frac{c_{a,a,d,1} \delta_n^{m_\mathrm{G}}}{(1-(n^p\,\Sigma_p^p(X)\,c_{a,a,d,1})^{-\frac{2}{{m_\mathrm{G}}}})^\frac{a}{2}}  
\end{align}
for all $n>n_0:=1/\big(\Sigma_p(X)\, c_{a,a,d,1}^{1/p}\big)$.
This yields 
\begin{align}\label{eq:simlyfiedLN}
\delta_n^2\geq \mleft(c_{a,a,d,1}\, n^p\,\Sigma_p^p(X) \mright)^{-\frac{2}{m_\mathrm{G}}}\mleft(1-( n^p\,\Sigma_p^p(X)\, c_{a,a,d,1})^{-\frac{2}{m_\mathrm{G}}}\mright)^\frac{1}{d-a} 
\end{align}
for all $n>n_0$. 
Using \eqref{eq:deltanGr} and 
 \eqref{eq:simlyfiedLN}  in \eqref{eq:lowerboundLnG}, we finally obtain
\begin{align}
 L_n(X) 
&\geq  \frac{m_\mathrm{G}}{m_\mathrm{G}+2p} \mleft(c_{a,a,d,1}\, n^p\,\Sigma_p^p(X) \mright)^{-\frac{2}{m_\mathrm{G}}}\mleft(1-( n^p\,\Sigma_p^p(X)\, c_{a,a,d,1})^{-\frac{2}{m_\mathrm{G}}}\mright)^\frac{1}{d-a} \label{eq:lowerboundLnG2}
\end{align}
for all $n>n_0$.  

Next, suppose that $\mathbb{F}=\mathbb{C}$ or $a\neq b$.  
Application of Theorem \ref{thm:singleshot} with $k=2$, $m=m_\mathrm{G}$, $\mu=\gamma_{r,d}$, and $\delta_0=1$   combined with Items     
\ref{item0Grass} and 
 \ref{item2Grass} of Lemma \ref{lem:Grassmann} 
  yields   $V_n(X)\geq  L_n(X)$ for all $n\geq 1/\big(\Sigma_p(X) c_{a,b,d,\beta}^{1/p}\big)$ with 
\begin{align}
 L_n(X)\label{eq:lowerboundLnG3}
 &=\frac{m_\mathrm{G}}{m_\mathrm{G}+2p}\mleft(c_{a,b,d,\beta}\, n^p\,\Sigma_p^p(X)\mright)^{-\frac{2}{m_\mathrm{G}}}. 
\end{align}
In particular, the lower bounds in \eqref{eq:lowerboundLnG2} and \eqref{eq:lowerboundLnG3} yield  
\begin{align}
 \liminf_{n\to\infty} n^\frac{2p}{m_\mathrm{G}}V_{n}(X)
 &\geq   \frac{m_\mathrm{G}}{m_\mathrm{G}+2p} \,\Sigma_p^{-\frac{2p}{m_\mathrm{G}}}(X)\,   c_{a,b,d,\beta}^{-\frac{2}{m_\mathrm{G}}}>0.    \label{eq:laststepCG}
\end{align}

We next derive an upper bound $U_n$ on the $n$-th quantization error $V_{n}(X)$. 
Suppose first that   $\mathbb{F}=\mathbb{R}$  with  $b\in\{a,a+1\}$ or $\mathbb{F}=\mathbb{C}$  with  $b=a$. 
Since  (see \eqref{eq:chordaldist} and recall that $\rho(x,y)=\rho_\mathrm{c}^2(x-y)$)
\begin{align}\label{eq:supforU}
\sup_{x\in\setX, y\in\setY} \rho(x,y) \leq a,
\end{align}
application of Theorem \ref{thm:singleshot2} with  $\nu=\gamma_{s,d}$, $m=m_\mathrm{G}$,  $k=2$, and $\delta_0=1$ combined with  Items \ref{item0Grass}  and \ref{item1Grass} of Lemma \ref{lem:Grassmann} yields 
$V_n(X)\leq U_n$ for all  $n\in\naturals$ with  
\begin{align}\label{eq:upperboundLnG}
U_n=\Gamma\mleft(1+\frac{2}{m_\mathrm{G}}\mright) (c_{a,b,d,\beta}n) ^{-\frac{2}{m_\mathrm{G}}}  +\mleft(a-1\mright)e^{-n c_{a,b,d,\beta}}.  
\end{align} 
Next, suppose that  $\mathbb{F}=\mathbb{R}$ with $b\notin\{a,a+1\}$ or $\mathbb{F}=\mathbb{C}$ with $a\neq b$.  
Fix  $\alpha\in(0,1)$ and set, for  every $n\in\naturals$,  
$\delta_{n}=n^{-\alpha/ m_\mathrm{G}}$ and 
\begin{align}
b_{\delta_{n}}=c_{a,b,d,\beta}\mleft(1-\delta_{n}^2\mright)^{\frac{\beta}{2}a(b-a+1)-a}. 
\end{align}
Application of Theorem \ref{thm:singleshot2} with  $\nu=\gamma_{s,d}$, $m=m_\mathrm{G}$,   $k=2$, and $\delta_0=\delta_n$ combined with \eqref{eq:supforU} and 
Item \ref{item2Grass} of Lemma \ref{lem:Grassmann} then implies  
$V_n(X)\leq U_n$ for all  $n\in\naturals$ with  
\begin{align}\label{eq:upperboundLnG2}
U_n=\Gamma\mleft(1+\frac{2}{m_\mathrm{G}}  \mright)\mleft(n \supc_{\delta_{n}}\mright)^{-\frac{2}{m_\mathrm{G}}} +\mleft(a-\delta_{n}^2\mright)e^{-n \supc_{\delta_{n}}\delta_{n}^{m_\mathrm{G}}} . 
\end{align} 
Next, note that $\lim_{n\to\infty}\delta_n=0 $ and  $\lim_{n\to\infty} n^\frac{2}{m_\mathrm{G}}e^{-n \supc_{\delta_{n}}\delta_{n}^{m_\mathrm{G}}}=0$.
Therefore,  \eqref{eq:upperboundLnG} and  \eqref{eq:upperboundLnG2} imply  that  
 \begin{align}\label{eq:laststepC2a}
 \limsup_{n\to\infty} n^\frac{2}{m_\mathrm{G}}V_{n}(X)
 &\leq \Gamma\mleft(1+\frac{2}{m_\mathrm{G}  }\mright) c_{a,b,d,\beta} ^{-\frac{2}{m_\mathrm{G}}}<\infty. 
\end{align}

The bounds in \eqref{eq:lowerboundLnG2}, \eqref{eq:lowerboundLnG3}, \eqref{eq:upperboundLnG}, and  \eqref{eq:upperboundLnG2}  generalize \cite[Theorem 4]{daliri08}
in the sense of applying to more general, i.e., not necessarily   uniformly distributed, random variables, and, in addition,  do not require the condition of $n$ being larger than an unspecified natural number.

Combining \eqref{eq:laststepCG} and \eqref{eq:laststepC2a} with 
Items \ref{eq:boundCcheck2} and \ref{eq:boundCcheck} of Lemma \ref{lem:dimbound}, respectively,  
now implies  
\begin{align}\label{eq:DkAA}
\frac{m_\mathrm{G}}{p} \leq \underline{D}_2(X) \leq \overline{D}_2(X)=m_\mathrm{G}.
\end{align}
If  $p=1$, then the $2$-nd quantization dimension  ${D}_2(X)$ exists and  equals  $m_\text{G}$, which for $r=s$, in turn  equals the geometric dimension of the Grassmannian $\setX$. 
 Similarly, if  $h_{\gamma_{r,d}}(X)>-\infty$ (Jensen's inequality  \cite[Theorem 2.3]{peprto92} combined with $\gamma_{r,d}(\setX)=1$ yields $h_{\gamma_{r,d}}(X) \leq \log (\gamma_{r,d}(\setX))=0$), 
then Item \ref{item2dimequal} of Corollary \ref{cor:dimequal} with $\mu=\gamma_{r,d}$ and  $\nu=\gamma_{s,d}$ allows us to conclude that 
\begin{align}\underline{\operatorname{dim}}_R(X)={D}_2(X)= m_\text{G}.\end{align}
In particular, if ${D}_2(X)=m_\text{G}$, then we have the following results for the upper and lower $2$-nd quantization coefficient. By 
\eqref{eq:laststepC2a}, it follows that  
\begin{align}\label{eq:ineqC}
\overline{C}_2(X)\leq  \Gamma\mleft(1+\frac{2}{m_\mathrm{G}}\mright)c_{a,b,d,\beta}^{-\frac{2}{m_\mathrm{G}}}<\infty,  
\end{align}
and if  \eqref{eq:gp2} holds for $p=1$, then 
 \eqref{eq:laststepCG} yields  
\begin{align}\label{eq:ineqCa}
 \underline{C}_2(X)\geq   \frac{m_\mathrm{G}}{m_\mathrm{G}+2} \,\Sigma_1^{-\frac{2}{m_\mathrm{G}}}\,   c_{a,b,d,\beta}^{-\frac{2}{m_\mathrm{G}}}>0.    
 \end{align}

\section{R-D Theory and Quantization for Self-Similar Sets}\label{sec:fractal}
\color{black}
We now   particularize our results on R-D theory and quantization  to  random variables taking values in self-similar sets. 
As in the case of compact manifolds treated in Section \ref{sec:manifold}, what we need here to apply the  program   developed are suitable sub/super-regularity  conditions, which again requires the computation of volume estimates of balls.  
\color{black}
To this end, we  start with preparatory material  on contracting similarities largely following the exposition in \cite{frheolro15}. 
 Throughout this section, we work in the ambient space   $\reals^d$  equipped with a  norm $\lVert\,\cdot\, \rVert$.   
A bijection $s\colon \reals^d\to\reals^d$ is called a  similarity if there exists a $\kappa \in(0,\infty)$ such that 
\begin{align}
\lVert s(u)-s(v)\rVert= \kappa\lVert u-v\rVert\quad \text{for all}\ u,v\in\reals^d. \label{eq:contractions}
\end{align} 
A contracting similarity is a similarity with  $\kappa \in (0,1)$, which in this case is  referred to as contraction parameter. 

Let $\setI=\{1,\dots,\abs{\setI}\}$ be a finite set of indices.  An iterated function system (IFS) is a  finite collection  $\{s_1,\dots, s_{\,\abs{\setI}}\}$ of contracting similarities $s_i\colon \reals^d\to\reals^d$  with  
contraction parameters $\kappa_i\in(0,1)$, $i\in\setI$.  Let $\{s_1,\dots, s_{\,\abs{\setI}}\}$ be  an IFS with  
contraction parameters $\kappa_i\in(0,1)$, $i\in\setI$. 
The unique positive real number $m$ satisfying  
\begin{equation}\label{eq:simdim}
\sum_{i\in\setI} \kappa_i^m=1
\end{equation}
is referred to as similarity dimension. 
By \cite[Theorem 4.1.3]{ed08}, there exists a unique nonempty and compact self-similar set $\setK$ satisfying 
\begin{equation} 
\setK=\bigcup_{i\in\setI} s_i(\setK)\,\subseteq\reals^d. 
\end{equation}
Let $\setI^\ast=\bigcup_{j\in\naturals}\setI^j$.  
For every $j\in\naturals$ and $\alpha=(i_1,\dots,i_j)\in\setI^\ast$, we set 
\begin{align}
\bar\alpha=
\begin{cases}
(i_1,\dots,i_{j-1})&\text{if $j>1$}\\
\omicron&\text{if $j=1$}
\end{cases}
\end{align}  
with $\omicron$ denoting the  empty sequence, which we declare to have  length zero.     
We designate the identity mapping on $\reals^d$ by $s_\omicron$, set $\kappa_\omicron=1$,  
and define  
\begin{align}
s_\alpha&=s_{i_1}\circ s_{i_2}\circ\dots\circ s_{i_j}\label{eq:direct1}\\
\kappa_\alpha&=\kappa_{i_1}\kappa_{i_2}\dots \kappa_{i_j}\label{eq:direct2}
\end{align}   
for all $\alpha\in\setI^\ast$. 
As $\{s_1,\dots, s_{\,\abs{\setI}}\}$ are contracting similarities $s_i\colon \reals^d\to\reals^d$  with  
corresponding contraction parameters $\kappa_i\in(0,1)$, $i\in\setI$, it
 follows  from \eqref{eq:direct1} and \eqref{eq:direct2} that  $s_\alpha$ is a contracting similarity with  contraction parameter $\kappa_\alpha$ for all $\alpha\in\setI^\ast$.  
 For every $\delta\in(0,\infty)$ and $y\in\reals^d$, we set  
\begin{align}
\setJ_\delta&=\{\alpha\in\setI^\ast:\kappa_\alpha\leq \delta<\kappa_{\bar\alpha}\}
\end{align}
and 
\begin{align}
\setJ_\delta(y)&=\mleft\{\alpha\in\setJ_\delta:\setB_{\lVert\,\cdot\,\rVert}\big(y,\delta \big)\cap s_\alpha(\setK)\neq\emptyset\mright\}. \label{eq:setJd}
\end{align}
Finally, let 
\begin{align}
\setE=\{s_\alpha\circ s_\beta^{-1}:\alpha,\beta\in\setI^\ast, \alpha\neq\beta\}
\end{align}
and equip the group of all similarities on $\reals^d$ with the topology induced by pointwise convergence.  
The IFS is said to satisfy the weak separation property if (see \cite[Definition on p. 3533]{ze96})
\begin{align}
s_\omicron\notin\overline{\setE\setminus\{s_\omicron\}}. \label{eq:wsp}
\end{align}
The weak separation property guarantees that the IFS does not admit infinitely many overlaps in the following sense:
\begin{lem}\cite[Theorem 1, Items (3a) and (5a)]{ze96}
Let $\setI=\{1,\dots,\abs{\setI}\}$ be a finite set of indices. For every $i\in\setI$, let 
$s_i\colon\reals^d\to\reals^d$  be a contracting similarity with  contraction parameter $\kappa_i\in(0,1)$. 
Suppose that the  self-similar set corresponding to the IFS $\{s_1,\dots, s_{\,\abs{\setI}}\}$  is not contained in any $(d-1)$-dimensional hyperplane. 
 Then, the IFS $\{s_1,\dots, s_{\,\abs{\setI}}\}$ 
satisfies the weak separation property if and only if for every $y\in\reals^d$, there exists an $\ell(y)\in\naturals$ such that 
\begin{align}
\abs{\{ s_\alpha\circ s_\beta: \lVert(s_\alpha\circ s_\beta)(y) -x\rVert<\delta, \alpha\in \setJ_\delta\}}\leq \ell(y)
\end{align}
for all $x\in\reals^d$, $\delta\in(0,\infty)$, and $\beta\in \setI^\ast$. 
\end{lem}
The next  result establishes  subregularity and superregularity for the $m$-dimensional Hausdorff measure restricted to self-similar sets of  similarity dimension $m$. 
  
\begin{lem}\label{lem:frheolro15} 
Let $\setI=\{1,\dots,\abs{\setI}\}$ be a finite set of indices   and let $\{s_1,\dots, s_{\,\abs{\setI}}\}$ be an IFS with corresponding  contraction parameters $\kappa_i$, $i\in\setI$.  Let  furthermore 
\begin{equation} 
\setK=\bigcup_{i\in\setI} s_i(\setK) 
\end{equation}
be the corresponding self-similar set, denote   its similarity dimension by  $m$,   and set $\mu=\colH^m\vert_\setK/\colH^m(\setK)$. Finally, let  $k\in(0,\infty)$,  $\setK\subseteq\setY\subseteq\reals^d$, and consider the distortion function  $\rho\colon \setK\times\setY\to [0,\infty), \rho(x,y)=\lVert x-y\rVert^k$, where $\lVert \,\cdot\,\rVert$ is a norm on $\reals^d$. 
   Then, the following statements hold: 
\begin{enumerate}[(ii)]
\item \label{eq:cantorsub2a} Suppose that  there exists a $c\in(0,\infty)$ such that $\abs{\setJ_\delta(y)}\leq c$ for all $y\in\setY$ and $\delta\in (0,\infty)$. Then, the  measure $\mu$ satisfies the following subregularity condition: 
\begin{align}
\mu\mleft(\setB_{\rho^{1/k}}\big(y,\delta\big)\mright)&\leq c\delta^m\quad\text{for all $y\in\setY$ and $\delta\in (0,\infty)$.} 
\end{align}
\item\label{eq:cantorsub2c} 
The  measure $\mu$ satisfies the following superregularity condition:
\begin{align}
\mu\mleft(\setB_{ \rho^{1/k}}\big(x,\delta\big)\mright)
&\geq \mleft(\frac{ \kappa_{\mathrm{min}}}{\operatorname{diam}(\setK)}\mright)^m\delta^m\quad\text{for all $x\in\setK$ and $\delta\in (0,\operatorname{diam}(\setK))$,}  \label{eq:cantorsup2a} 
\end{align}
 where 
\begin{align}
\operatorname{diam}(\setK)=\sup_{x,y\in\setK} \lVert x-y \rVert 
\end{align} 
and $\kappa_{\mathrm{min}}=\min\{\kappa_1,\dots,\kappa_{\,\abs{\setI}}\}$. 
\item \label{eq:cantorsub2b} If the IFS satisfies the  weak separation property \eqref{eq:wsp}  and $\setK$ 
is not contained in any hyperplane of dimension $d-1$, then $0<\colH^m(\setK)<\infty$ 
and there exists a $c\in(0,\infty)$ such that $\abs{\setJ_\delta(y)}\leq c$ for all $y\in\reals^d$ and $\delta\in (0,\infty)$.  
\end{enumerate}
\end{lem}
\begin{proof}
The proof follows from the corresponding parts of the proof of 
\cite[Theorem 2.1]{frheolro15}. 
\end{proof}
In the remainder of this section, we fix a norm $\lVert\, \cdot\,\rVert$   on $\reals^d$ and a $k\in(0,\infty)$. 
Further,
 we consider an IFS satisfying  the weak separation poperty \eqref{eq:wsp}  and such that the corresponding self-similar set  $\setK\subseteq\reals^d$  
is not contained in any hyperplane of dimension $d-1$.\color{black} {}  
In particular, by Item \ref{eq:cantorsub2b} of Lemma  \ref{lem:frheolro15}, we then have $0<\colH^m(\setK)<\infty$, and hence  $\dim_\mathrm{H}(\setK)=m$ with $m$ denoting the similarity dimension of $\setK$. Moreover, again by Item \ref{eq:cantorsub2b} of   Lemma \ref{lem:frheolro15}, 
there must exist  a $c\in(0,\infty)$ such that 
\begin{align}\label{eq:c_1}
 \abs{\setJ_\delta(y)}\leq c\quad\text{for all $y\in\reals^d$ and $\delta\in (0,\infty)$.}
\end{align}
Finally, we consider a random variable $X$  taking values in $\setK\subseteq \setY\subseteq \reals^d$ and  the distortion function  
$\rho\colon \setK\times \setY\to[0,\infty)$, $\rho(x,y)=\lVert x-y\rVert^k$. 

We first  derive a lower bound on the R-D function  under the assumptions   
\begin{align}
\mu_X\ll\mu:=\colH^m\vert_\setK/\colH^m(\setK)
\end{align}  
and $h_\mu(X)>-\infty$. 
Jensen's inequality  \cite[Theorem 2.3]{peprto92} combined with $\mu(\setK)=1$ yields $h_{\mu}(X)\leq \log(\mu(\setK))=0 $, which in turn implies $\abs{h_{\mu}(X)}<\infty$. 
Application of Theorem \ref{thm:new} for $\delta_0=\infty$ combined with \eqref{eq:c_1}  and Item \ref{eq:cantorsub2a} of Lemma \ref{lem:frheolro15} 
 yields  $R_X^{\text{SLB}}(D)\geq R_X^{\text{L}}(D)$ with 
\begin{align}\label{eq:Rlowerbound}
R_X^{\text{L}}(D) =
h_{\mu}(X) +F_{m,k,c}(D) \quad\text{for all $D\in(0,\infty)$.}
\end{align}

Next, we derive a lower bound on the $n$-th quantization error under the assumptions $\mu_X\ll\mu$  and  
\begin{align}
\Sigma_p(X):=\mleft\lVert \frac{\mathrm d \mu_X}{\mathrm d \mu}\mright\rVert_{p/(p-1)}^{(\mu)}<\infty \quad\text{with $p\in[1,\infty).$}\label{eq:gp2a}
\end{align}    
Application of  Theorem \ref{thm:singleshot}  for $\delta_0=\infty$  combined 
 with \eqref{eq:c_1}  and Item \ref{eq:cantorsub2a} of Lemma \ref{lem:frheolro15} 
 yields 
 $V_{n}(X)\geq L_n(X)$   for all  $n\in\naturals$ with   
\begin{align}
L_{n}(X)= \frac{m}{m+pk} c^{-\frac{k}{ m}}\Sigma_p^{-\frac{pk}{ m}}(X) n^{-\frac{pk}{ m}}.
\label{eq:directbounda}
\end{align}
In particular, we have  
\begin{align}
 \liminf_{n\to\infty} n^\frac{p k}{m}V_{n}(X)
 &\geq \frac{m}{m+pk} c^{-\frac{k}{ m}}\Sigma_p^{-\frac{pk}{ m}}(X)>0. \label{eq:laststepD}
\end{align}

We now derive an upper bound $U_n$ on the $n$-th quantization error $V_n(X)$. 
Application of Theorem \ref{thm:singleshot2} with  $\beta=\operatorname{diam}(\setK)$, $\nu= \mu$, and $\delta_0= \operatorname{diam}(\setK)$ combined with  Item  \ref{eq:cantorsub2c} of  Lemma  \ref{lem:frheolro15}  yields 
 $V_n(X)\leq U_n$ for all  $n\in\naturals$ with  
\begin{align}\label{eq:directbounda2}
U_n&=\Gamma\mleft(1+\frac{k}{m}\mright)  \mleft(\frac{\operatorname{diam}(\setK)}{\kappa_{\mathrm{min}}} \mright)^{k}   n^{-\frac{k}{m}}.   
\end{align} 
In particular, we have  
 \begin{align}
 \limsup_{n\to\infty} n^\frac{k}{m}V_{n}(X)
 &\leq   \Gamma\mleft(1+\frac{k}{m}\mright)\mleft(\frac{\operatorname{diam}(\setK)}{\kappa_{\mathrm{min}}} \mright)^{k}<\infty. \label{eq:laststepC2b}
\end{align} 
Combining \eqref{eq:laststepD} and  \eqref{eq:laststepC2b}   with 
Items \ref{eq:boundCcheck2} and \ref{eq:boundCcheck} of Lemma \ref{lem:dimbound}, respectively,  
now yields 
\begin{align}\label{eq:Dk22}
\frac{m}{p}\leq \underline{D}_k(X) \leq \overline{D}_k(X)\leq m. 
\end{align}
If  $p=1$, then the $k$-th quantization dimension  ${D}_k(X)$ exists and  equals  the similarity dimension $m$. 
Similarly, if  $h_\mu(X)>-\infty$ (Jensen's inequality  \cite[Theorem 2.3]{peprto92} combined with $\mu(\setX)=1$ yields $h_\mu(X) \leq \log (\mu(\setX))=0$), 
then Item \ref{item2dimequal} of Corollary \ref{cor:dimequal} (with  $\setX=\setK$ and $\nu=\mu$)  allows us to conclude that 
\begin{align}
\underline{\operatorname{dim}}_R(X)={D}_k(X)=m.
\end{align} 
\color{black}
In particular, if ${D}_k(X)=m$, then we obtain the following results for the upper and lower $k$-th quantization coefficients.  
By \eqref{eq:laststepC2b}, we have 
\begin{align}\label{eq:chainineq}
 \overline{C}_k(X) \leq  \Gamma\mleft(1+\frac{k}{m}\mright)\mleft(\frac{\operatorname{diam}(\setK)}{\kappa_{\mathrm{min}}} \mright)^{k}<\infty,  
\end{align}
and if \eqref{eq:gp2a} holds for $p=1$, then   
\eqref{eq:laststepD} yields 
\begin{align}\label{eq:chainineqa}
\underline{C}_k(X) \geq 
\frac{m}{m+k} c^{-\frac{k}{ m}}\Sigma_1^{-\frac{k}{ m}}(X)>0. 
\end{align}
For $\setY=\setK$,  if  in addition to ${D}_k(X)=m$, we have $m> k$,  $\mu\ll\mu_X$,  and $
\lVert \mathrm d\mu /\mathrm d\mu_X \rVert_\infty^{(\mu_X)}<\infty$, 
then we can apply Theorem \ref{thm:upperstrong}  with $\nu=\mu$ 
to obtain the improved upper bound 
\begin{align}\label{eq:CupperultraB}
 \overline{C}_k(X)\leq  \Omega_{k/m}(X) \Gamma\mleft(1+\frac{k}{m}\mright)\mleft(\frac{\operatorname{diam}(\setK)}{\kappa_{\mathrm{min}}} \mright)^{k},   
\end{align}
where
\begin{align}\label{eq:OmegavMB}
\Omega_{k/m}(X)
&=\opE\mleft[\mleft(\frac{\mathrm d \mu}{\mathrm d\mu_X}(X)\mright)^{\frac{k}{m}}\mright]\quad\text{for all $m>k$}
\end{align}
with strict inequality in \eqref{eq:OmegavMB} unless $\mu_X=\mu$. 

In order to further quantify our results, the following example  considers a random variable   $X$ taking values in the middle third Cantor set.

\begin{exa}

Consider the  middle third Cantor set 
 $\setC\subseteq [0,1]$,  i.e., the self-similar  set corresponding to the choice   $\setI=\{1,2\}$, $\kappa_1=\kappa_2=1/3$, $s_1(x)=x/3$, and $s_2(x)=x/3+2/3$. 
This self-similar set has similarity dimension  $m_\setC=\log (2)/\log (3)$ and satisfies   $0<\mathscr{H}^{m_\setC}(\setC)<\infty$ \cite[Example 4.5]{fa14}.   
Let $\mu=\colH^{m_\setC}\vert_\setC/\colH^{m_\setC}(\setC)$.  
We now employ  Lemma \ref{lem:frheolro15} to establish  the subregularity dimension and the subregularity constants of  $\mu$ 
for  $\setY=\setC$ and $\setY=\reals$ with 
\begin{align}
\rho\colon\setC\times \setY&\to[0,\infty),\ \ \rho(x,y)= \lvert x-y\rvert^2.
\end{align}
Specifically, note that  
$\kappa_\alpha=3^{-j}$ for  all $\alpha=(i_1,\dots,i_j)$ and  $j\in\naturals$.  Thus,    
\begin{align}
\setJ_{\delta}
&=\{\alpha\in\setI^\ast:\kappa_\alpha\leq \delta <\kappa_{\bar\alpha}\}\\
&=\{\alpha:\abs{\alpha}=j\}\quad\text{for all $\delta\in\big[3^{-j},3^{-j+1}\big)$ and $j\in\naturals$,}  
\end{align}
which implies (see Figure \ref{fig:cantor}) 
\begin{align}
\abs{\setJ_{\delta}(y)}\leq c_\setC\quad  \text{for all $\delta\in (0,1)$ and $y\in\setY$}
\end{align}
with 
\begin{align}\label{eq:defcC}
c_\setC
&=
\begin{cases}
2 &\quad \text{if $\setY=\setC$}\\
3 &\quad \text{if $\setY=\reals$}.
\end{cases}
\end{align}
Therefore, by Item \ref{eq:cantorsub2a} in Lemma \ref{lem:frheolro15}, the subregularity dimension  of $\mu$ is $m_\setC$ and  the corresponding subregularity constants are given by $c_\setC$ as defined in \eqref{eq:defcC} and $\delta_0=\infty$.

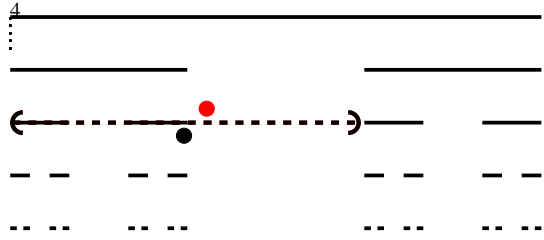
\begin{figure}[!t]
\begin{center}
\begin{tikzpicture}[scale=2]
  \foreach \order in {0,...,4}
    \draw[line width=0.5mm, yshift=-\order*10pt]  l-system[l-system={cantor set, axiom=F, order=\order, step=100pt/(3^\order)}];  
\put (10,5){\draw[{Arc Barb[]}-{Arc Barb[]}, ultra thick, red, dashed] (0,-20pt) -- (66pt,-20pt);};
\filldraw[red] (1.3,-0.61) circle (0.5mm);
\filldraw[black] (1.15,-0.79) circle (0.5mm);
\put (0,-5){\draw[{Arc Barb[]}-{Arc Barb[]}, ultra thick, black, dashed] (0,-20pt) -- (66pt,-20pt);};
\put (0,0){
\put (4,-90){\draw [dotted, line width=0.3mm] (0pt,0pt) -- (0pt,-7pt);};
\put (18,-90){\draw [dotted, line width=0.3mm] (0pt,0pt) -- (0pt,-7pt);};
};
\put (45,0){
\put (4,-90){\draw [dotted, line width=0.3mm] (0pt,0pt) -- (0pt,-7pt);};
\put (18,-90){\draw [dotted, line width=0.3mm] (0pt,0pt) -- (0pt,-7pt);};
};
\put (134,0){
\put (4,-90){\draw [dotted, line width=0.3mm] (0pt,0pt) -- (0pt,-7pt);};
\put (18,-90){\draw [dotted, line width=0.3mm] (0pt,0pt) -- (0pt,-7pt);};
};
\put (178,0){
\put (4,-90){\draw [dotted, line width=0.3mm] (0pt,0pt) -- (0pt,-7pt);};
\put (18,-90){\draw [dotted, line width=0.3mm] (0pt,0pt) -- (0pt,-7pt);};
};
\put (220,-22){$|\alpha|=1$};
\put (220,-42){$|\alpha|=2$};  
\put (220,-62){$|\alpha|=3$};   
\put (220,-82){$|\alpha|=4$};
\end{tikzpicture}
\vspace*{10truemm}
\caption{The sets $s_\alpha([0,1])$ with $\abs{\alpha}=j$ have length $3^{-j}$. 
If $y\in\reals$, then 
at most three different  sets $s_\alpha([0,1])$ with $\abs{\alpha}=j$ intersect an open interval of  length $2(3^{-j+1})$ centered at $y$ (depicted in red  for $j=2$). 
If $y\in\setC$, then 
at most two different  sets $s_\alpha([0,1])$ with $\abs{\alpha}=j$ intersect an open interval of  length $2(3^{-j+1})$ centered at $y$ (depicted in black  for $j=2$). 
\label{fig:cantor}} 
\end{center}
\end{figure}

Now, suppose that the random variable   $X$ takes values in  $\setC$. If   $\mu_X\ll\mu$ with $h_\mu(X)>-\infty$,    
then \eqref{eq:Rlowerbound} evaluated for $\setK=\setC$, $m=m_\setC$, $c=c_\setC$, and $k=2$
  yields  the R-D lower bound 
\begin{align}\label{eq:SLB1aC}
R_X^{\text{L}}(D) = h_{\mu}(X)- 
\frac{m_\setC}{2} + \frac{m_\setC}{2}\log \mleft(\frac{m_\setC}{2 D} \mright)
-\log
\mleft(c_\setC
\Gamma\mleft(1+\frac{m_\setC}{2}\mright)
\mright). 
\end{align}
If $\mu_X\ll\mu$ with   \eqref{eq:gp2a} satisfied  for $p=1$  
or  $\abs{h_\mu(X)}<\infty$,    
then   ${D}_2(X)=m_\setC$  owing to \eqref{eq:Dk22}. Furthermore,   \eqref{eq:chainineq} and  \eqref{eq:chainineqa}
particularized to  $\setK=\setC$, $m=m_\setC$, $c=c_\setC$,  $k=2$, $\kappa_{\mathrm{min}}=1/3$,  and $\operatorname{diam}(\setC)=1$
yields the following chain of inequalities for the lower and upper $2$-nd quantization coefficient 
\begin{align}\label{eq:boundC1C2}
0\,<\, \frac{m_\setC}{m_\setC+2}\,\Sigma_{1}^{-\frac{2}{ m_\setC}}(X) \,c_\setC^{-\frac{2}{ m_\setC}} \leq \,  \underline{C}_2(X)\, \leq \, \overline{C}_2(X)\, \leq \, 9\Gamma\mleft(1+\frac{2}{m_\setC}\mright)\,<\,\infty.
\end{align}


We finally note that 
the $n$-th quantization error for the special case of $X$ taking values in $\setC$ with  uniform distribution $\mu_X=\mu$ and  $\setY=\reals$ is known explicitly and equals \cite[Theorem 5.2]{grlu97} 
\begin{align}\label{eq:VnC}
V_{n}(X)=\frac{1}{18^{l_n}}\frac{1}{8}\mleft(2^{l_n+1}-n+\frac{1}{9}\mleft(n-2^{l_n}\mright)\mright),
\end{align}
where $l_n=\lfloor\log (n)/\log (2)\rfloor$ for all $n\in\naturals$.
 The set of accumulation points of the sequence $V_{n}(X)n^{2/{m_\setC}}$ is given by  the interval  $[1/8, f(17/(8+4m_\setC))]$, where $f(s)=(1/72) s^{2/m_\setC}(17-8s)$  \cite[Theorem 6.3]{grlu97}, which implies  $\underline{C}_2(X) ={1}/{8}$ and $\overline{C}_2(X) =f(17/(8+4m_\setC))$. In particular,  ${C}_2(X)$ does not exist. 
\end{exa}

\section{Future Work and Open Problems}\label{sec:fwop} 

An interesting research direction is the identification of concrete examples of non-i.i.d. stationary ergodic processes $(X_i)_{i\in\naturals}$ 
which allow explicit expressions for  the corresponding   Shannon lower bounds according to Theorem \ref{thm:new}.  Specifically, suppose that there exists a subregular measure $\mu$ on $(\setX,\colX)$ such that $\mu_{X_1}\times\dots\times \mu_{X_\ell}\ll \mu^{(\ell)}$ with 
\begin{align}
\mu^{(\ell)}= \underbrace{\mu\times\dots\times \mu}_{\text{$\ell$ times}}\quad\text{for all $\ell\in\naturals$}. 
\end{align} 
Then,  Proposition  \ref{prp:subprod} can be used to infer subregularity of $\mu^{(\ell)}$  from subregularity of $\mu$ for all $\ell\in\naturals$, and the lower bound  $R_{X^{(\ell)}}^{\text{L}}(D)$ in \eqref{eq:toshow1}  can be evaluated for all  sections $X^{(\ell)}=(X_1,\dots,X_\ell)$ of $(X_i)_{i\in\naturals}$. The resulting 
Shannon lower bound is then explicit and depends on the subregularity dimension and the subregularity parameters of $\mu$. 
To the best of our knowledge, results on such Shannon lower bounds for non-i.i.d. stationary ergodic  processes have not been reported in the literature. 

Another open problem is related to the assumptions in  Theorem \ref{thm:singleshot}. Specifically, in \eqref{eq:useHoelder} in the proof of 
Theorem \ref{thm:singleshot}, 
subregularity of $\mu_X$ is obtained from subregularity of $\mu$ by application of  H\"older's inequality \cite[Theorem 1, p. 372]{ka18}.  
A corollary to Theorem \ref{thm:singleshot}, namely Corollary \ref{cor:DC}, then allows us to conclude that, if there exists a $p\in[1,\infty)$ such that 
\begin{align}  \label{eq:prevent2}
 \lVert\mathrm d \mu_X/\mathrm d \mu\rVert^{(\mu)}_{p/(p-1)}<\infty, 
\end{align} 
 we get  the lower bound 
 \begin{align}\label{eq:resultD}
 \underline{D}_k(X)\geq m/p.  
\end{align}
For $p>1$,  this  bound is strictly smaller than the subregularity dimension $m$. Note that if 
\begin{align} \label{eq:prevent}
 \lVert\mathrm d \mu_X/\mathrm d \mu\rVert^{(\mu)}_{q/(q-1)}=\infty\quad\text{for all  $q\in [1,p)$,}
\end{align} 
then \eqref{eq:resultD}  is also the largest possible lower bound on $\underline{D}_k(X)$ that can be obtained by application of Theorem \ref{thm:singleshot}. 
It would now be  interesting to understand if there is a  concrete example
saturating \eqref{eq:resultD}  for  $p\in(1,\infty)$  or whether the somewhat crude step of applying H\"older's inequality in the proof of Theorem \ref{thm:singleshot} renders the lower bound in Item \ref{eq:boundD} of Corollary \ref{cor:DC} structurally loose.  
 Two broad classes of random variables can, however, already be excluded from consideration as candidates by the following arguments: 
\begin{itemize}
\item
 If $\setX=\setY=\reals^d$,  $\mu=\colL^d$, and $k\in[1,\infty)$, then  \cite[Theorem 6.2, Remark 6.3]{grlu00} imply ${D}_k(X)= d$ provided that 
 \begin{align}\label{eq:contrulesout} \lVert\mathrm d \mu_X/\mathrm d \mu\rVert^{(\mu)}_{d/(d+k)}<\infty. 
 \end{align} 
  This hence rules out all continuous random variables 
 satisfying \eqref{eq:contrulesout}.  
 \item   
If  $\mu$ is subregular of subregularity dimension $m$ with $\mu(\setX)=1$ for $\delta_0<\infty$, then  Corollary \ref{cor:RDD} implies 
$\underline{D}_k(X)\geq m$ provided that $\abs{h_\mu(X)}<\infty$. This hence  rules out every $X$ of finite generalized entropy. 
\end{itemize} 

Finally, it would be interesting to understand whether the link between quantization and R-D theory established in Section \ref{sec:link} extends to the case where the subregularity parameter $\delta_0$ is finite, which is, e.g.,  the case for  the  compact manifolds described in Section \ref{sec:manifold}.  
 The problem  with $\delta_0<\infty$ here is that the  lower bound $ R_{(\ell)} (D) \geq \widetilde R_{(\ell)}(D)$  in 
Theorem \ref{cor:MSiid} is valid for  $D \in(0,D_{(\ell)})$, but unfortunately  $D_{(\ell)}\to 0$ as $\ell\to\infty$ unless $\delta_0=\infty$. 



\appendix

\section{Proof of Proposition~\ref{prp:subprod}}\label{sec:proofsubprod}
We prove Item \ref{item:prodsub} only as the proof of Item \ref{item:prodsup} follows along the exact same lines using superregularity of the measures $\nu_i$ instead of subregularity of the  $\mu_i$. 
First, we consider the  special case  $\ell=2$, $\alpha_1=\alpha_2=k=1$, which constitutes  the core piece  of the proof. The general case  then follows by induction over $\ell$ and subsequent application of Lemma \ref{lem:simpilysub} to extend the result to arbitrary $k, \alpha_1,\dots,\alpha_\ell\in(0,\infty)$. 

\begin{lem}\label{lem:subprod}
 For $i=1,2$, let $(\setX_i,\colX_i)$ and $(\setY_i,\colY_i)$ be   measurable spaces and consider the distortion function   $\rho_i\colon \setX_i\times\setY_i\to[0,\infty]$.    
Suppose that, for $i=1,2$,  
$\mu_i$ is a $\sigma$-finite  $\rho_i$-subregular measure on $(\setX_i,\colX_i)$ of dimension $m_i\in(0,\infty)$ with subregularity  constants $c_i\in(0,\infty)$ and $\delta_i\in(0,\infty]$,
and set 
\begin{align}
\bar\rho((x_1,x_2),(y_1,y_2))=\rho_1(x_1,y_1)+\rho_2(x_2,y_2). 
\end{align}
Then, $\bar\mu=\mu_1\otimes\mu_2$
is $\bar\rho$-subregular of dimension $m_1+m_2$ with subregularity constants $c_1c_2
\Gamma\mleft(1+m_1\mright)\Gamma\mleft(1+m_2\mright)/\Gamma\mleft(1+m_1+m_2\mright)$ and $\min(\delta_1,\delta_2)$, i.e.,
\begin{align}\label{eq:toshowprod}
\bar\mu\mleft(\setB_{\bar\rho}\mleft(y,\delta\mright)\mright)\leq c_1c_2
\frac{\Gamma\mleft(1+m_1\mright)\Gamma\mleft(1+m_2\mright)}{\Gamma\mleft(1+m_1+m_2\mright)}
\delta^{m_1+m_2} 
\end{align}
for all $y\in\setY_1\times \setY_2$ and $\delta\in(0,\min(\delta_1,\delta_2))$. 
\end{lem} 
\begin{proof}
Fix $y=(y_1,y_2)\in \setY_1\times \setY_2$ and $\delta\in(0,\min(\delta_1, \delta_2))$ arbitrarily. We have 
\begin{align}
\bar\mu\mleft(\setB_{\bar\rho}\mleft(y,\delta\mright)\mright)
&=\bar\mu\mleft((x_1,x_2)\in\setX_1\times \setX_2:\rho_1(x_1,y_1)+\rho_2(x_2,y_2)<\delta\mright)\\
&=\int  I(x_1)
\,\mathrm d\mu_1(x_1)\label{eq:applyFub},
\end{align}
where \eqref{eq:applyFub} is by Tonelli's theorem for characteristic functions \cite[Theorem 10.9]{ba95} with 
\begin{align}
I(x_1)
&=\mu_2\mleft(\setB_{\rho_2}\mleft(y_2,\mleft(\delta-\min\mleft(\rho_1(x_1,y_1),\delta\mright)\mright)\mright)\mright). 
\end{align}
Now, 
$\rho_2$-subregularity of $\mu_2$ implies 
\begin{align}\label{eq:appsub2}
I(x_1)\leq c_2 g(x_1)
\end{align}
with 
\begin{align}
g(x_1)=\mleft(\delta-\min\mleft(\rho_1(x_1,y_1),\delta\mright)\mright)^{m_2}, 
\end{align}
which, when used in \eqref{eq:applyFub}, yields 
\begin{align}\label{eq:interm1}
\bar\mu\mleft(\setB_{\bar\rho}\mleft(y,\delta\mright)\mright)
&\leq c_2\int  g(x_1) \,\mathrm d\mu_1(x_1).
\end{align}
Next, note that 
\begin{align}
 &\int  g(x_1)\,\mathrm d\mu_1(x_1)\label{eq:stepuu0}\\
 &=\int_0^{\delta^{m_2}}\mu_1(\{x_1:g(x_1)>t\})\,\mathrm d t\label{eq:stepuu1} \\
 &=\int_0^{\delta^{m_2}}\mu_1\mleft(\setB_{\rho_1}\mleft(y_1,\mleft(\delta-t^{\frac{1}{m_2}}\mright)\mright)\mright)\,\mathrm d t\\
 &\leq c_1 \int_0^{\delta^{m_2}}\mleft(\delta-t^{\frac{1}{m_2}}\mright)^{m_1}\label{eq:stepuu2} \,\mathrm d t\\
 &= c_1\delta^{m_1} \int_0^{\delta^{m_2}}\mleft(1-\delta^{-1}t^{\frac{1}{m_2}}\mright)^{{m_1}}\label{eq:stepuu3} \,\mathrm d t\\
 &= c_1\delta^{m_1+m_2} m_2\int_0^{1}\mleft(1-s\mright)^{{m_1}}s^{{m_2}-1}\label{eq:stepuu4} \,\mathrm d s\\
 &=c_1\delta^{m_1+m_2} m_2\,B_{1+m_1,\,m_2}(1),\label{eq:stepuu5}
\end{align}
where \eqref{eq:stepuu1} follows from Lemma \ref{lem:fubI} upon noting that $0\leq g(x_1)\leq \delta^{m_2}$ for all $x_1\in\setX_1$,    in \eqref{eq:stepuu2}  we employed $\rho_1$-subregularity of $\mu_1$, and in \eqref{eq:stepuu4}  we changed variables according  to 
$s=\delta^{-1}t^{\frac{1}{m_2}}$. 
Using \eqref{eq:stepuu0}--\eqref{eq:stepuu5}  in \eqref{eq:interm1}, we obtain
\begin{align}\label{eq:interm2}
\bar\mu\mleft(\setB_{\bar\rho}\mleft(y,\delta\mright)\mright)
&\leq c_1c_2\delta^{m_1+m_2} m_2\,B_{1+m_1,\,m_2}(1)\\
&=c_1c_2
\frac{\Gamma\mleft(1+m_1\mright)\Gamma\mleft(1+m_2\mright)}{\Gamma\mleft(1+m_1+m_2\mright)}
\delta^{m_1+m_2}, 
\end{align} 
where the last step follows from  \eqref{eq:gammabeta} 
%
%
%
 and $m_2\Gamma(m_2)=\Gamma(1+m_2)$.  
\end{proof}

The generalization of  Lemma  \ref{lem:subprod}  to arbitrary $\ell$ is effected by induction as  follows.

\begin{lem}\label{lem:subprod1}
For $i=1,\dots, \ell$, let $(\setX_i,\colX_i)$ and  $(\setY_i,\colY_i)$ be   measurable spaces  and consider the distortion function $\rho_i\colon \setX_i\times\setY_i\to[0,\infty]$.   
Suppose that, for $i=1,\dots, \ell$, $\mu_i$ is a $\sigma$-finite $\rho_i$-subregular measure of dimension $m_i\in(0,\infty)$ with subregularity constants $c_i\in(0,\infty)$ and $\delta_i\in(0,\infty]$, i.e., 
\begin{align}\label{eq:subregularity1}
\mu_i\mleft(\setB_{\rho_i}\mleft(y_i,\delta\mright)\mright)\leq c_i\delta^{m_i}\quad\text{for all $y_i\in\setY_i$ and $\delta\in (0,\delta_i)$}. 
\end{align}
Then, $\mu^{(\ell)}:=\mu_1\otimes\dots\otimes\mu_\ell$ satisfies 
\begin{align}\label{eq:toshowprod1}
\mu^{(\ell)}\mleft(\setB_{\tilde{\rho}_{(\ell)}}\mleft(y^{(\ell)},\delta\mright)\mright)\leq 
\tilde{c}_{(\ell)}\delta^{\sum_{i=1}^\ell m_i}
\end{align}
for all $y^{(\ell)}\in\setY_1\times\dots\times  \setY_\ell$ and $\delta\in(0,\min(\delta_1, \dots,\delta_\ell))$
with 
\begin{align}
\tilde{\rho}_{(\ell)}((x_1,\dots,x_\ell),(y_1,\dots,y_\ell))=\sum_{i=1}^\ell\rho_i(x_i,y_i),
\end{align}
and
\begin{align}
\tilde{c}_{(\ell)}=
\frac{\prod_{i=1}^\ell \Gamma\mleft(1+m_i\mright)}{\Gamma\mleft(1+\sum_{i=1}^\ell m_i\mright)} \times \prod_{i=1}^\ell c_i .
\end{align}
\end{lem} 
\begin{proof}
The proof is by induction on $\ell$. The base case  $\ell=2$  was established in Lemma \ref{lem:subprod}.  
For the induction step, assume that \eqref{eq:toshowprod1} holds for $\ell-1$, i.e.,   
\begin{align}\label{eq:tildemu1}
\mu^{(\ell-1)}\mleft(\setB_{\tilde{\rho}_{(\ell-1)}}\mleft(y^{(\ell-1)},\delta\mright)\mright)\leq \tilde{c}_{(\ell-1)} \delta^{\sum_{i=1}^{\ell-1} m_i}
\end{align}
for all  $y^{(\ell-1)}\in \setY_1\times\dots\times  \setY_{\ell-1}$ and $\delta\in(0,\min(\delta_1, \dots,\delta_{\ell-1}))$. 
Since 
\begin{align}
&\tilde{\rho}_{(\ell)}((x_1,\dots,x_\ell),(y_1,\dots,y_\ell))\\
&=\tilde{\rho}_{(\ell-1)}((x_1,\dots,x_{\ell-1}),(y_1,\dots,y_{\ell-1}))+\rho_{\ell}(x_\ell,y_\ell)
\end{align}
and $\mu^{(\ell)}=\mu^{(\ell-1)}\otimes\mu_\ell$ 
with $\mu^{(\ell-1)}$ satisfying \eqref{eq:tildemu1} and $\mu_\ell$ obeying \eqref{eq:subregularity1}, 
Lemma \ref{lem:subprod}  applied to $\mu^{(\ell-1)}$, $\mu_\ell$, $\tilde{\rho}_{(\ell-1)}$, and $\rho_\ell$ yields  
\begin{align}
\mu^{( \ell)}\mleft(\setB_{\tilde{\rho}_{( \ell)}}\mleft(y^{( \ell)},\delta\mright)\mright)
&\leq \tilde{c}_{( \ell-1)}c_ \ell\frac{\Gamma\mleft(1+\sum_{i=1}^{ \ell-1} m_i\mright)\Gamma\mleft(1+m_ \ell\mright)}{\Gamma\mleft(1+\big(\sum_{i=1}^{ \ell-1} m_i\big)+m_ \ell\mright)}  \delta^{(\sum_{i=1}^{ \ell-1} m_i)+m_ \ell}\\
&= \tilde{c}_{( \ell)}\delta^{\sum_{i=1}^{ \ell} m_i}  
\end{align}
for all $y^{( \ell)}\in\setY_1\times\dots\times  \setY_ \ell$ and $\delta\in(0,\min(\delta_1, \dots,\delta_ \ell))$. 
\end{proof}

To complete the proof of Proposition~\ref{prp:subprod}, we need to generalize  Lemma  \ref{lem:subprod1} to arbitrary $k\in(0,\infty)$ and arbitrary $\alpha_1,\dots,\alpha_\ell\in(0,\infty)$. 
To this end, note that 
 $\rho_i^{1/k}$-subregularity of  $\mu_i$, 
by Lemma \ref{lem:simpilysub}  implies   
that 
$\mu_i$ is $(\alpha_i\rho_i)$-subregular of dimension $m_i/k$ with subregularity constants $\tilde c_i=c_i\alpha_i^{- m_i/k}$ and $\tilde \delta_i=\alpha_i\delta_i^{k}$.  
Lemma \ref{lem:subprod1}    applied to the  $(\alpha_i\rho_i)$-subregular  measures $\mu_i$ and the distortion functions $\alpha_i\rho_i$ 
therefore yields 
that $\mu^{(\ell)}$ is 
$\rho_{(\ell)}$-subregular of dimension $m_{(\ell)}/k$ with subregularity constants $c_{(\ell)}$ (as defined in \eqref{eq:parsubsup}) and $\min(\alpha_1\delta_1^{k}, \dots,\alpha_\ell\delta_\ell^{k})$. 
Another application of Lemma \ref{lem:simpilysub} then establishes that 
$\mu^{( \ell)}=\mu_1\otimes\dots\otimes\mu_\ell$ is 
$\rho_{(\ell)}^{1/k}$-subregular of dimension $m_{(\ell)}$ with subregularity constants $c_{(\ell)}$ and $\delta_{(\ell)}$ as defined in \eqref{eq:parsubsupb}. \qed

\section{Proof of Proposition \ref{thm.ko17}}\label{thm.ko17proof}

By \cite[Lemma A.1]{kade94} we can lower-bound $R_X(D)$ according to\footnote{Imposing slightly stronger constraints on the distortion function $\rho$ than just measurability,  \eqref{eq:maxRd} can actually be made to hold with  equality  \cite[Theorem 2.3]{cs74}.}   
\begin{align}
R_X(D)\geq \sup_{\alpha,s}\mleft(\opE[\log (\alpha(X))]-sD\mright)\quad\text{for all $D\in(0,\infty)$}, \label{eq:maxRd}
\end{align}
where the supremization is over all measurable functions $\alpha\colon \setX\to (0,\infty)$ and  $s\in[0,\infty)$ such that 
\begin{equation}
\opE\mleft[\alpha(X)e^{-s\rho(X,y)}\mright]\leq 1\quad\text{for all $y\in\setY$.} \label{eq:constraint}
\end{equation}
The main idea of the proof is to lower-bound the supremum in \eqref{eq:maxRd} 
using a judiciously  parametrized family $(\alpha_s)_{s\geq 0}$ of measurable functions $\alpha_s\colon \setX\to (0,\infty)$ satisfying 
\begin{align}\label{eq:alphasexp}
\opE\big[\alpha_s(X)e^{-s\rho(X,y)}\big]\leq 1\quad\text{for all $s\in[0,\infty)$ and $y\in\setY$.} 
\end{align}
Specifically, for every $s\in [0,\infty)$, we set 
\begin{align}\label{eq:defas1}
\alpha_s(x)=
\begin{cases}
\displaystyle \mleft(\frac{\mathrm d\mu_X}{\mathrm d\mu}(x)g(s)\mright)^{-1}&\quad \text{if $\ \ \ \displaystyle\frac{\mathrm d\mu_X}{\mathrm d\mu}(x)>0$ }\\
1&\quad\text{else,} 
\end{cases}
\end{align}
where $(\mathrm d\mu_X/\mathrm d\mu)(x)$ exists and is nonnegative and measurable by  the Radon--Nikod\'ym theorem \cite[Theorem 8.9]{ba95}.   
Now  \eqref{eq:alphasexp} is established by noting that 
\begin{align}
\opE\mleft[\alpha_s(X)e^{-s\rho(X,y)}\mright]
&=\int \alpha_s(x)e^{-s\rho(x,y)}\,\mathrm d \mu_X(x)\\
&=\int_\setA \alpha_s(x)e^{-s\rho(x,y)} \frac{\mathrm d\mu_X}{\mathrm d\mu}(x) \,\mathrm d \mu(x)\label{eq:defasA}\\
&=g(s)^{-1}\int_\setA e^{-s\rho(x,y)}\,\mathrm d \mu(x)\label{eq:defas}\\
&\leq 1\quad\text{for all $s\in[0,\infty)$ and  $y\in \setY$},\label{eq:defgs} 
\end{align}
where in \eqref{eq:defasA} we set 
\begin{equation}
\setA=\Big\{x\in\setX:\frac{\mathrm d\mu_X}{\mathrm d\mu}(x)>0\Big\},\label{eq:setA}
\end{equation}
\eqref{eq:defas} follows from \eqref{eq:defas1}, and \eqref{eq:defgs} is by \eqref{eq:defns1}. 
It therefore follows from    \eqref{eq:maxRd} that 
\begin{align}
R_X(D)\geq \sup_{s\geq 0}\mleft(\opE[\log (\alpha_s(X))]-sD\mright)\quad\text{for all $D\in(0,\infty)$}.  \label{eq:maxRd2}
\end{align}
Finally,
\begin{align}
&\opE[\log (\alpha_s(X))]\label{eq:aa1}\\
&=\int \log (\alpha_s(x)) \,\mathrm d \mu_X(x)\\
&=\int \log (\alpha_s(x)) \frac{\mathrm d\mu_X}{\mathrm d\mu}(x)  \,\mathrm d \mu(x)\\
&=-\int\log \mleft(\frac{\mathrm d\mu_X}{\mathrm d\mu}(x)\mright)\frac{\,\mathrm d\mu_X}{\,\mathrm d\mu}(x)  \,\mathrm d \mu(x)-\log (g(s))\int \frac{\mathrm d\mu_X}{\mathrm d\mu}(x)  \,\mathrm d \mu(x)\label{eq:defas2}\\
&=h_\mu(X) -\log (g(s))\quad\text{for all $s\in[0,\infty)$}, \label{eq:aa2}
\end{align}
where \eqref{eq:defas2} follows from \eqref{eq:defas1} 
and \eqref{eq:aa2} holds owing to  $\int \frac{\mathrm d\mu_X}{\mathrm d\mu}(x)  \,\mathrm d \mu(x)= \mu_X(\setX)=1$.
Inserting \eqref{eq:aa1}--\eqref{eq:aa2} into \eqref{eq:maxRd2} concludes the proof. \qed

\section{Proof of Theorem \ref{thm:new}}\label{thm:newproof}
We start by upper-bounding  $g(s)$ in Proposition \ref{thm.ko17} as follows: 
\begin{align}
g(s)
&=\sup_{y\in\setY}\int e^{-s\rho(x,y)} \,\mathrm d \mu(x)\label{stepa1}\\
&=\sup_{y\in\setY}\int_0^1\mu \mleft(\mleft\{x\in\setX:e^{-s\rho(x,y)}>t \mright\}\mright)\,\mathrm d t  \label{stepa} \\
&=\sup_{y\in\setY}\int_0^\infty e^{-u}\mu \mleft(\mleft\{x\in\setX:u> s\rho(x,y)\mright\}\mright)\,\mathrm d u  \label{stepb} \\
&=\sup_{y\in\setY}\int_0^\infty e^{-u}\mu\mleft(\setB_{\rho}\mleft(y,\frac{u}{s}\mright) \mright) \,\mathrm d u\label{step1}\\
&=\sup_{y\in\setY}\mleft(\int_0^{s \delta_0^k} e^{-u}\mu\mleft(\setB_{\rho}\mleft(y,\frac{u}{s}\mright)\mright)\,\mathrm d u+\int_{s \delta_0^k}^\infty e^{-u}\mu\mleft(\setB_{\rho}\mleft(y,\frac{u}{s}\mright)\mright)\,\mathrm d u\mright)\label{step2}\\
&\leq cs^{- \frac{m}{k}}\int_0^{s\delta_0^k} e^{-u}u^{\frac{m}{k}} \,\mathrm d u+ e^{-s\delta_0^k}\label{step4} \\
&= cs^{- \frac{m}{k}}\gamma\mleft(1+ \frac{m}{k}, s\delta_0^k\mright) + e^{-s\delta_0^k}\quad \text{for all $s\in(0,\infty)$}, \label{step5} 
\end{align}
where \eqref{stepa}  follows from Lemma \ref{lem:fubI}  upon noting that $e^{-s\rho(x,y)}\in [0,1]$, 
in  \eqref{stepb} we applied the change of  variables  $t=e^{-u}$, 
and in \eqref{step4} we used  that i) thanks to Lemma \ref{lem:simpilysub}, $\mu$ is $\rho$-subregular of dimension $m/k$   with subregularity  constants $c$ and $\delta_0^k$ and ii) the assumption $\mu(\setX)=1$ if $\delta_0<\infty$.

Inserting \eqref{stepa1}--\eqref{step5} into   the right-hand side of \eqref{eq:SLB} yields 
\begin{align}
R_X^{\text{SLB}}(D)
& \geq h_\mu(X)-\inf_{s\geq 0}\mleft(sD +\log \mleft(cs^{- \frac{m}{k}}\gamma\mleft(1+ \frac{m}{k}, s\delta_0^k\mright) + e^{-s\delta_0^k}\mright) \mright)\label{laststeptheoa}\\
& \geq h_\mu(X)-\frac{m}{k} -\log \mleft(c\mleft(\frac{m}{kD}\mright)^{- \frac{m}{k}}\gamma\mleft(1+ \frac{m}{k}, \frac{m\delta_0^k}{kD}\mright) +e^{-\frac{m\delta_0^k}{kD}}\mright), \label{laststeptheo}
\end{align} 
where in   \eqref{laststeptheo} we set $s=m/(kD)$.   
Since $\gamma(1+ {m}/{k}, {m\delta_0^k}/(kD)) \leq \Gamma(1+ {m}/{k})$, this  
establishes the lower bounds in   \eqref{eq:SLB1a} and \eqref{eq:SLB1}. Note that the only bounding step in the derivation of \eqref{eq:SLB1a} arises from  the application of $\rho$-subregularity of $\mu$ in \eqref{step4} as for $\delta_0=\infty$ we have $\gamma(1+ {m}/{k}, {m\delta_0^k}/(kD)) = \Gamma(1+ {m}/{k})$  and the infimum in \eqref{laststeptheoa} is attained at $s=m/(kD)$.     
The bound in \eqref{eq:RLasympA} and 
the limit   in  \eqref{eq:RLasymp}  follow trivially.    
\qed

\section{Proof of Theorem \ref{thm:singleshot}}\label{thm:singleshotproof}
Let  $n\in \naturals$ and $f\in\setF_n(\setX,\setY)$ be arbitrary but fixed.
We have 
\begin{align}
 \mu_X\mleft(\mleft\{x\in \setX: \rho(x,f(x))\geq \delta\mright\}\mright)
&\geq \mu_X\mleft(\mleft\{x\in \setX: \rho(x,y)\geq \delta \ \text{for all}\ {y\in f(\setX)} \mright\}\mright)\label{eq:beginopP}\\
&=\mu_X(\setX)- \mu_X\mleft(\bigcup_{y\in f(\setX)}\setB_{\rho}\mleft(y,\delta\mright)\mright)\\
&\geq 1-\sum_{y\in f(\setX)}\mu_X\mleft(\setB_{\rho}\mleft(y,\delta\mright)\mright)\label{eq:unionb}\\
&\geq 1-\sum_{y\in f(\setX)} c^{\frac{1}{p}}\mleft\lVert\frac{\mathrm d \mu_X}{\mathrm d \mu} \mright\rVert^{(\mu)}_{p/(p-1)} \delta^{\frac{m}{pk}}
\label{eq:useHoelder} \\
&\geq 1-n c^{\frac{1}{p}}\mleft\lVert\frac{\mathrm d \mu_X}{\mathrm d \mu} \mright\rVert^{(\mu)}_{p/(p-1)} \delta^{\frac{m}{pk}}\label{eq:endopP}\\
&= 1- \mleft(\frac{\delta}{h(n)}\mright)^{\frac{m}{pk}}\quad\text{for all $\delta\in(0,\delta_0^k)$,}\label{eq:endopP2}
\end{align}
where \eqref{eq:unionb} follows from  $\mu_X(\setX)=1$ and the union bound, in \eqref{eq:useHoelder} we applied 
Lemmata  \ref{lem:simpilysub} and    \ref{lem:transformsub}, 
 in 
   \eqref{eq:endopP} we used 
 $\lvert f(\setX)\rvert \leq n$, and in  \eqref{eq:endopP2} we set 
 \begin{align}
 h(n)=c^{-\frac{k}{m}} \mleft( \mleft\lVert\frac{\mathrm d \mu_X}{\mathrm d \mu} \mright\rVert^{(\mu)}_{p/(p-1)}\mright)^{-\frac{pk}{m}} n^{-\frac{pk}{m}}.
 \end{align}

We thus have 
\begin{align}
\opE[\rho(X,f(X))]
&=
\int_{0}^\infty \mu_X\mleft(\mleft\{x\in \setX: \rho(x,f(x))\geq \delta \mright\}\mright)\,\mathrm d \delta\label{eq:FubiniEP}\\
&\geq
\int_{0}^{\min\{h(n), \delta_0^k\}}
\mu_X\mleft(\mleft\{x\in \setX: \rho(x,f(x))\geq \delta\mright\}\mright)\,\mathrm d \delta \label{eq:cv2}\\
&\geq 
\int_{0}^{\min\{h(n), \delta_0^k\}}
\mleft( 1-  \mleft(\frac{\delta}{h(n)}\mright)^{\frac{m}{pk}}      \mright)\,\mathrm d \delta\label{eq:useopP}\\
&=\min\mleft\{h(n), \delta_0^k\mright\}\mleft(1 -\frac{pk}{(m+pk)}\min\mleft\{1, \frac{\delta_0^k}{h(n)}\mright\}^{\frac{m}{pk}}\mright)\\  
&\geq\min\mleft\{h(n), \delta_0^k\mright\}\mleft(1 - \frac{pk}{m+pk} \mright)\\
&=\min\mleft\{h(n), \delta_0^k\mright\} \frac{m}{m+pk},\label{eq:finalstep}
\end{align}
where \eqref{eq:FubiniEP} follows from Lemma \ref{lem:fubI} and 
\eqref{eq:useopP} is by \eqref{eq:beginopP}--\eqref{eq:endopP2} upon  noting that $\delta < \delta^k_0$ on the integration domain. 
Since $n\in \naturals$ and $f\in\setF_n(\setX,\setY)$  were assumed to be arbitrary, we can conclude that 
\begin{align}
V_n(X)&=\inf_{f\in\setF_n(\setX,\setY)}\opE[\rho(X,f(X))]\\
&\geq \min\mleft\{h(n), \delta_0^k\mright\} \frac{m}{m+pk}.
\end{align}
\qed

\section{Proof of Lemma \ref{lem:RDD}}\label{lem:RDDproof}

If $\lim_{n\to\infty}V_{n}(X) >0$, then the definition of  $\underline{D}_k(X)$ in \eqref{eq:Dinf} implies   ${D}_k(X)=\infty$, which establishes  the claim. Next, consider the case    $\lim_{n\to\infty}V_{n}(X) =0$ and fix  $n\in\naturals$ and an $n$-quantizer $f\in\setF_n$ arbitrarily. Define $g\colon \setX\to\setX\times\setY$ according to $g(x)=(x,f(x))$ and consider the pushforward measures 
$\mu_f:=f_\ast(\mu_X)$ and $\mu_g:= g_\ast(\mu_X)$.  
We first show that $\mu_g\ll\mu_X\otimes\mu_f$ with   Radon--Nikod\'ym derivative  
\begin{align}
\frac{\mathrm d \mu_g}{\mathrm d (\mu_X\otimes\mu_f)}(x,y)=h(x,y):=\frac{\ind{\{(x,y)\in \setX\times\setY:f(x)=y\}}(x,y)}{\opP[f(X)=y]}.
\end{align} 
To this end, note that for every $\setA\in \colX\otimes\colY$, we have 
\begin{align}
\int_\setA h(x,y) \,\,\mathrm d (\mu_X\otimes \mu_f)(x,y)
&=\int  \mleft( \int_{\setA_y} h(x,y) \, \mathrm d \mu_X(x)  \mright)     \mathrm d \mu_f(y)  \label{eq:RDfub} \\
&=\int  \mleft( \int _{\setA_{f(z)}}h(x,f(z)) \, \mathrm d \mu_X(x)  \mright)     \mathrm d \mu_X(z)\\
&= \int  \frac{\opP[f(X)=f(z) \ \text{and}\ X\in\setA_{f(z)} ]}{\opP[f(X)=f(z)] }  \mathrm d \mu_X(z) \\
&= \sum_{a\in f(\setX)}  \frac{\opP[f(X)=a \ \text{and}\ X\in\setA_{a} ]}{\opP[f(X)=a] }  \int_{f^{-1}(\{a\})}   \mathrm d \mu_X(z) \\
&= \sum_{a\in f(\setX)}  \opP[f(X)=a \ \text{and}\ X\in\setA_{a} ]\\
&= \sum_{a\in f(\setX)}  \opP[f(X)=a \ \text{and}\ g(X)\in\setA ]\\
&=  \opP[ g(X)\in\setA ]\\
&= \mu_g(\setA), \label{eq:RDfuba}
\end{align} 
where \eqref{eq:RDfub} follows from  Tonelli's theorem  \cite[Theorem 10.9]{ba95} and we set, for every $y\in\setY$,  $\setA_y=\{x\in\setX:(x,y)\in\setA\}$. 

Next,  note that
\begin{align}
\mu_g(\setA)&= \int_\setA h(x,y) \,\,\mathrm d (\mu_X\otimes \mu_f)(x,y)\label{eq:RNpr}\\
&=\int \mu_{f(X)\vert X}(\setA_x\vert x) \mathrm d \mu_X(x)\quad \text{for all $\setA\in \colX\otimes\colY$},\label{eq:RNprU}  
\end{align}
where \eqref{eq:RNpr} is by \eqref{eq:RDfub}--\eqref{eq:RDfuba} and in  \eqref{eq:RNprU} we again applied    Tonelli's theorem 
with $\setA_x=\{y\in\setY:(x,y)\in\setA\}$ and   
\begin{align}
\mu_{f(X)\vert X}(\setB\vert x):=\int_{\setB} h(x,y) \, \mathrm d \mu_f(y)\quad \text{for all $x\in\setX$ and $\setB\in \colY$}.  
\end{align}
Tonelli's theorem  also guarantees that   $\mu_{f(X)\vert X}(\setB\vert \cdot)$  is measurable for all $\setB\in \colY$. Moreover,  $\mu_{f(X)\vert X}(\cdot\vert x)$ is a probability measure as 
\begin{align}
\mu_{f(X)\vert X}(\setY\vert x)
&=\int h(x,y) \, \mathrm d \mu_f(y)\\
&=\int h(x,f(z)) \, \mathrm d \mu_X(z)\\
&=\sum_{a\in f(\setX)} h(x,a)  \opP[f(X)=a]  \\
&=1\quad \text{for all $x\in\setX$.} 
\end{align}

Finally, we have 
\begin{align}
I(X;f(X))
&=\int \log  (h(x,y)) \mathrm d  \mu_g(x,y)\\
&=\int \log  (h(g(x)) ) \mathrm d  \mu_X(x)\\
&\leq  \log \mleft (\int   h(g(x))  \mathrm d  \mu_X(x)  \mright) \label{eq:JensenBBB} \\
&=  \log \mleft( \sum_{a\in f(\setX)} \int_{f^{-1}(a)}   h(x,a)  \mathrm d  \mu_X(x)  \mright)\\
&=\log \mleft( \sum_{a\in f(\setX)} 1\mright)\\
&\leq \log(n),    
\end{align}
where \eqref{eq:JensenBBB}  is by Jensen's inequality  \cite[Theorem 2.3]{peprto92}. 
We can now  conclude that (see  \eqref{eq:RD})
\begin{align}\label{eq:RD111}
R_{X}(D) &=\inf_{\mu_{Y\mid X}:\,\opE[\rho(X,Y)]\,\leq\, D} I(X;Y)\\
&\leq \log (n)\quad\text{if $\opE[\rho(X,f(X))]\leq D$}. \label{eq:RD1116}
\end{align}
Next, note that the definition of $V_{n}(X)$ in \eqref{eq:quanterr} implies that for every  $\delta > 0$, there exists an $f_\delta\in\setF_n$ such
 that $\opE[\rho(X,f_\delta(X))]\leq V_{n}(X)+\delta $. Combined with  \eqref{eq:RD111}--\eqref{eq:RD1116} this yields 
 \begin{align}\label{eq:RD1112}
R_{X}(V_{n}(X)+\delta) \leq \log (n)\quad\text{for all  $\delta>0$}. 
\end{align}
If there exists an $n_0\in\naturals$ such that $V_{n_0}(X)=0$, then \eqref{eq:RD1112} implies $R_{X}(D) \leq \log (n_0)$ for all $D>0$ so that ${\dim}_R(X)=0$ and the statement to be established follows  trivially. 
Now, suppose that $V_{n}(X)>0$ for all $n\in\naturals$. 
Using continuity of  $R_{X}(D)$ \cite[Lemma 1.1]{cs74} and the fact that   $n$  in \eqref{eq:RD1112} was assumed to be arbitrary, it follows that   
 \begin{align}\label{eq:RD1112A}
R_{X}(V_{n}(X)) 
& \leq \log (n)\quad\text{for all $n\in\naturals$}.  
\end{align}
 We can therefore conclude that 
\begin{align}
\underline{D}_k(X)
&=\liminf_{n\to\infty}  \frac{ k \log (n)}{\log (1/V_{n}(X))}\\ 
&\geq \liminf_{n\to\infty}\frac{ k R_{X}(V_{n}(X))}{\log (1/V_{n}(X))}\label{eq:useAAA}\\
&\geq \liminf_{D \to 0}\frac{k R_{X}(D)}{\log (1/D)} \label{eq:XXX}\\
&= \liminf_{D \to 0}\frac{R_{X}(D^k)}{\log (1/D)}\\ 
&=\underline{\dim}_R(X), 
\end{align} 
\color{black}
where \eqref{eq:useAAA} follows from \eqref{eq:RD1112A} and in  \eqref{eq:XXX} we used that for every  function $f\colon (0,\infty)\to (0,\infty)$ and   sequence $(x_n)_{n\in\naturals}$ in $(0,\infty)$ with $\lim_{n\to\infty}x_n=0$,  we have 
\begin{align}
\liminf_{n\to\infty} f(x_n)&= \sup_{n\in\naturals} \inf_{k\geq n} f(x_k) \geq \sup_{n\in\naturals} \inf_{D\in(0,x_n]} f(D)\\
& \geq  \sup_{a\in (0,\infty)} \inf_{D\in(0,a]} f(D)\label{eq:useargliming}= \liminf_{D\to 0} f(D),
\end{align}
where the inequality in \eqref{eq:useargliming} follows from the observation that for every $a\in (0,\infty)$, thanks to $\lim_{n\to\infty}x_n=0$, there exists an $n_0\in\naturals$ such that $x_{n_0}<a$ so that  $\inf_{D\in(0,x_{n_0}]} f(D) \geq \inf_{D\in(0,a]} f(D)$.  
\color{black}
\qed

\section{Proof of Theorem  \ref{thm:singleshot2}}\label{thm:singleshot2proof}

Fix $n\in\naturals$ arbitrarily. We first show that  
\begin{align}
V_{n}(X)&\leq  \opE\mleft[\min_{i=1,\dots,n}\tilde \rho(X,y_i)\mright]\quad\text{for all $y_1,\dots,y_n\in\setY$.}\label{eq:upperreal} 
\end{align}
To this end, fix  $y_1,\dots,y_n\in\setY$ arbitrarily and set 
\begin{align}
\setA_i=\bigcap_{j=1}^n\{x\in\setX:\rho(x,y_i)\leq\rho(x,y_j) \} \quad\text{for $i=1,\dots,n$}.\label{eq:defAi}
\end{align}
The sets $\setA_i$ are measurable by virtue of  being finite intersections of preimages of  $[0,\beta^k]$ under measurable functions.
By construction, for every $i=1,\dots,n$, $\setA_i$ consists of all points $x\in\setX$  that have $y_i$ as nearest point among  $\{y_1,\dots,y_n\}$. 
Next, set $\setB_1=\setA_1$ and 
\begin{align}\label{eq:322}
\setB_j=\setA_j\setminus \bigcup_{i=1}^{j-1}\setA_j\quad\text{for $j=2,\dots,n$.}
\end{align}
By construction, the sets $\setB_1,\dots,\setB_n$ are pairwise disjoint measurable sets with $\setB_i\subseteq\setA_i$ for $i=1,\dots,n$ and 
\begin{align}
\setX=\bigcup_{i=1}^n\setB_i.
\end{align}
Now, consider the  $n$-quantizer $f\colon\setX\to\setY$ defined according to $f(x)=y_i$ for all $x\in\setB_i$ and $i=1,\dots,n$.  Then, we have  
\begin{align}
V_{n}(X)&\leq \opE\, [\rho(X,f(X))]\label{eq:firstVN}\\
&=\sum_{i=1}^n\int_{\setB_i} \rho(x,f(x))\,\mathrm d \mu_X(x) \\
&=\sum_{i=1}^n\int_{\setB_i} \rho(x,y_i)\,\mathrm d \mu_X(x) \\
&=\int \min_{i=1,\dots,n}\rho(x,y_i)\,\mathrm d \mu_X(x)\label{eq:useAi}\\
&=\opE \mleft[ \min_{i=1,\dots,n}\rho(X,y_i)\mright],\label{eq:lastVN}
\end{align}
where \eqref{eq:useAi}  follows from   \eqref{eq:defAi} combined with \eqref{eq:322}. Since $y_1,\dots,y_n\in\setY$ were assumed to be  arbitrary, \eqref{eq:upperreal} follows from 
\eqref{eq:firstVN}--\eqref{eq:lastVN}. 

Now, consider i.i.d. random variables $Y_1,\dots,Y_n$ taking  values in  \setY, of distribution $\nu$, \color{black}and independent from $X$\color{black}. It follows from \eqref{eq:upperreal} that 
\begin{align}
V_{n}(X)&\leq \opE\mleft[\min_{i=1,\dots,n} \rho(X,Y_i)\mright]\label{eq:upperreal2}. 
\end{align}
Fix $x\in\setX$ arbitrarily and set $Z_n(x)= \min_{i=1,\dots,n} \rho(x,Y_i)$. Then, we have 
\begin{align}
\opP[Z_n(x)\geq \delta]
&=\opP\mleft[ \rho(x,Y_i)\geq \delta \ \ \ \text{for}\ i=1,\dots,n\mright]\label{eq:firstZxA}\\
&=\opP\mleft[ \rho(x,Y_1)\geq \delta \mright]^n\\
&=\mleft(1-\opP\mleft[ \rho(x,Y_1)< \delta \mright]\mright)^n\\
&\leq e^{-n\opP\mleft[ \rho(x,Y_1)\,<\, \delta \mright]}\label{eq:endusefurtwo}\\
&=e^{-n\nu\mleft(\widetilde\setB_{ \rho}\mleft(x,\delta \mright)\mright)}\\
&\leq
\begin{cases}
e^{-\supc n \delta^\frac{m}{k}}\quad\text{if $\delta <\delta_0^k$}\\
e^{-\supc n \delta_0^m}\quad\text{\color{black} if $\delta_0^k\leq \delta<\infty$,\color{black}}\label{eq:usesup}
\end{cases}
\end{align}
where in \eqref{eq:endusefurtwo} we used that $1-t\leq e^{-t}$ for all  $t\in[0,1]$, the  case $\delta<\delta_0^k$ in \eqref{eq:usesup} follows from the fact that $\nu$, by Lemma \ref{lem:simpilysub}, is $ \rho$-superregular of dimension $m/k$ with superregularity constants $\supc$ and $\delta_0^k$,   and for $\delta_0^k\leq \delta<\infty$,  \eqref{eq:usesup}  is by 
\begin{align}
\nu\mleft(\widetilde\setB_{ \rho}\mleft(x,\delta \mright)\mright)
\, \geq\, \nu\mleft(\widetilde\setB_{ \rho}\mleft(x,\delta_0^k \mright)\mright)
\, \geq\, \supc \delta_0^m,
\end{align}
which, in turn,  follows  from monotonicity and superregularity of $\nu$ combined with   Item  \ref{item:superregularity} in Lemma \ref{lem:extsub}. 
 We can therefore upper-bound $\opE[Z_n(x)]$ as follows: 
\begin{align}
\opE[Z_n(x)]
&=\int_0^{\beta^k} \opP\mleft[Z_n(x)\geq \delta\mright]\,\mathrm d \delta\label{eq:usefubini}\\
&\leq \int_0^{\delta_0^k} e^{-\supc n\delta^\frac{m}{k}}\,\mathrm d \delta+ e^{-\supc n\delta_0^m}\int_{\min\{\delta_0^k,\beta^k\}}^{\beta^k} \,\mathrm d \delta
 \label{eq:useProb}\\
&=\frac{k}{m}\mleft(\supc n\mright)^{-\frac{k}{m}}\int_0^{\supc n\delta_0^m} e^{-t}t^{\frac{k}{m}-1}\,\mathrm d t +e^{-\supc n\delta_0^m}\int_{\min\{\delta_0^k,\beta^k\}}^{\beta^k} \,\mathrm d \delta \label{eq:changetot}\\
&=\frac{k}{m}\gamma\mleft(\frac{k}{m},\, \supc n\delta_0^m\mright)\mleft(\supc n\mright)^{-\frac{k}{m}} +e^{-\supc n\delta_0^m}\int_{\min\{\delta_0^k,\beta^k\}}^{\beta^k} \,\mathrm d \delta\\
&\leq  U_n,\label{eq:useZx}
\end{align}
where \eqref{eq:usefubini} follows from Lemma \ref{lem:fubI} combined with \eqref{eq:bounddist}, \eqref{eq:useProb} is by \eqref{eq:firstZxA}--\eqref{eq:usesup}, and in \eqref{eq:changetot} 
we  changed variables according to $t=\supc n\delta^{m/k}$. 
The claim now follows from \eqref{eq:upperreal2}
and 
\begin{align}
\opE\mleft[\min_{i=1,\dots,n} \rho(X,Y_i)\mright]
&= \opE\mleft[Z_n(X)\mright]\\
&=\int \opE\mleft[Z_n(x)\mright] \,\mathrm d \mu_X(x)\\
&\leq \int U_n  \,\mathrm d \mu_X(x)\label{eq:LASTstep}\\
&= U_n,
\end{align}
where \eqref{eq:LASTstep} is by \eqref{eq:usefubini}--\eqref{eq:useZx}.
\qed

\section{Proof of Theorem \ref{thm:upperstrong}}\label{thm:upperstrongproof}
We first show that
\begin{align}
\nu\mleft( \setB_{\rho}\mleft(x,\delta \mright)\mright)
&\geq
\begin{cases}
 \supc \delta^\frac{m}{k}\quad \text{if $\delta<\delta_0^k$}\\
 \tilde \supc \delta^\frac{m}{k}\quad\text{\color{black}if $\delta_0^k\leq\delta \leq \beta^k<\infty$\color{black}}
\end{cases}\quad\text{for all $x\in\setX$,} \label{eq:supext}
\end{align}
where $\tilde \supc=\supc (\delta_0/\beta)^m$ for $\delta_0<\infty$. 
The case $\delta<\delta_0^k$ follows from the fact that  $\rho^{1/k}$-superregularity of $\nu$, by  Lemma \ref{lem:simpilysub} implies 
that 
$\nu$ is $\rho$-superregular of dimension $m/k$ with superregularity constants $\supc$ and $\delta_0^k$.
If  $\delta_0^k\leq\delta \leq \beta^k<\infty$, then we have  
\begin{align}
\nu\mleft(\setB_{\rho}\mleft(x,\delta \mright)\mright)
&\geq \nu\mleft(\setB_{\rho}\mleft(x,\delta_0^k \mright)\mright)\label{eq:ball}\\
& \geq \supc\,\delta_0^m\label{eq:usesuprescape}\\
& \geq  \tilde \supc\, \delta^\frac{m}{k}\quad \text{for all $x\in\setX$},\label{eq:debtildeb} 
\end{align}
where in \eqref{eq:ball} we used that $\widetilde\setB_{\rho}(x,\delta)=\setB_{\rho}(x,\delta)$ as a consequence of $\setX=\setY$ and symmetry of $\rho$, and \eqref{eq:usesuprescape}  is by  Item \ref{item:superregularity} in Lemma  \ref{lem:extsub}. 
Now, let 
\begin{align}\label{eq:defgg}
g(x,\delta)=\frac{\mu_X\mleft(\setB_{\rho}\mleft(x,\delta \mright)\mright)}{\nu\mleft(\setB_{\rho}\mleft(x,\delta \mright)\mright)}   
\end{align}
and note that there exists a $t\in(0,\infty)$ such that 
\begin{align}\label{eq:deffxinf}
f(x) :=\frac{\mathrm d \nu}{\mathrm d \mu_X}(x)\leq 1/t\quad \text{for $\mu_X$-almost all $x\in\setX$} 
\end{align}
owing to \eqref{eq:deffxinfA}. We therefore have
\begin{align}
g(x,\delta)
&=\frac{\mu_X\mleft(\setB_{\rho}\mleft(x,\delta \mright)\mright)}{ \int_{\setB_{\rho}\mleft(x,\delta \mright)}f(x) \,\mathrm d \mu_X(x)}\label{eq:gg1} \\ 
&\geq t \quad\text{for all $x\in\setX$  and  $\delta\in(0,\infty)$.}\label{eq:gg2}
\end{align}
Next, note that
\begin{align}
\lim_{\delta\to 0} \frac{1}{g(x,\delta)}
&=\lim_{\delta\to 0} \frac{1}{\mu_X\mleft(\setB_{\rho}\mleft(x,\delta \mright)\mright)}\int_{\setB_{\rho}\mleft(x,\delta \mright)}{f(x)} \,\mathrm d \mu_X(x)\label{eq:usemattila1}\\
& = {f(x)} \quad \text{for $\mu_X$-almost all $x\in\setX$,}\label{eq:usemattila}  
\end{align}
where \eqref{eq:usemattila} follows from \cite[Corollary 2.14, Item (2)]{ma99}   upon noting that  $\mu_X$ is a  Radon measure \cite[Definition 6.6]{cimo12}  
thanks to  \cite[Theorem 6.1 and Proposition 6.7]{cimo12} and $f(x)$ is bounded  by \eqref{eq:deffxinf}. 
The upper bound in \eqref{eq:deffxinf} now implies 
\begin{align}\label{eq:OmegaSigma3A}
\Omega_{k/m}(X)=\opE\mleft[f^{\frac{k}{m}}(X)\mright]\ \leq\  t^{-\frac{k}{m}}\ <\ \infty. 
\end{align}

Next,  let    $(X_n)_{n\in\naturals}$ be a sequence of i.i.d. copies of $X$ and define, for every  $n\in\naturals$ and $x\in\setX$, 
the random variable $U_n(x)= n^{k/m}\min_{i=1,\dots,n}\rho(x,X_i)$.   
Then, we have 
\begin{align}
\opP\mleft[U_n(x)\geq \delta\mright] 
&=\opP\mleft[\rho(x,X_i)\geq n^{-k/m}\delta \ \ \ \text{for}\ i=1,\dots,n\mright]\label{eq:firstZx}\\
&=\opP\mleft[\rho(x,X_1)\geq n^{-k/m}\delta \mright]^n\label{eq:firstZx2}\\
&=\mleft(1-\opP\mleft[\rho(x,X_1)< n^{-k/m}\delta \mright]\mright)^n\\
&\leq e^{-n\opP\mleft[X_1\, \in\, \setB_{\rho}\mleft(x,\, n^{-\frac{k}{m}}\delta \mright)\mright]}\label{eq:item2a}\\
&\leq
\begin{cases}
e^{-\supc\, \delta^\frac{m}{k} g\big(x,\, n^{-\frac{k}{m}}\delta\big)}&\quad\text{if $\delta<\delta_0^kn^{\frac{k}{m}} $}\\
e^{-\tilde \supc\, \delta^\frac{m}{k}g\big(x,\,n^{-\frac{k}{m}\delta}\big)}&\quad\text{if $\delta_0^kn^{\frac{k}{m}}\leq \delta \leq \beta^kn^{\frac{k}{m}}<\infty$}
\end{cases}\label{eq:item2b}
\end{align}
for all $n\in\naturals$, 
where  \eqref{eq:item2b} follows from  \eqref{eq:supext} and \eqref{eq:defgg}. Moreover, \eqref{eq:bounddist2} implies 
\begin{align}
\opP\mleft[U_n(x)\geq \delta\mright]=0\quad\text{if $\delta> \beta^kn^{\frac{k}{m}}$.}\label{eq:item2c}
\end{align} 
We can now upper-bound $\limsup_{n\to\infty} n^\frac{k}{m}V_n(X) $ as follows:  
\begin{align}
\limsup_{n\to\infty}n^\frac{k}{m} V_{n}(X)
&\leq  \limsup_{n\to\infty}n^\frac{k}{m}\opE\mleft[\min_{i=1,\dots,n}\rho(X,X_i)\mright]\label{eq:useupperZX}\\
&=\limsup_{n\to\infty}  \opE\mleft[U_n(X) \mright]\\
&=\limsup_{n\to\infty} \int \mleft(\int_0^{\infty }\opP\mleft[U_n(x)\geq \delta\mright]\,\mathrm d \delta\mright)\,\mathrm d\mu_X(x) \label{eq:usefubini2}\\
&\leq  \int \mleft(\int_0^{\infty} \limsup_{n\to\infty}\opP\mleft[U_n(x)\geq \delta\mright]\,\mathrm d \delta\mright)\,\mathrm d\mu_X(x), \label{eq:usefubini2a}
\end{align}
 where \eqref{eq:useupperZX} is by \eqref{eq:upperreal}  with $\setY=\setX$,  
 \eqref{eq:usefubini2} follows from Lemma \ref{lem:fubI}, and in \eqref{eq:usefubini2a} we applied Fatou's lemma for $\limsup$ 
 \cite[Theorem 1.20]{amfupa00} twice 
 noting that $\opP[U_n(x)\geq \delta] \leq  h(\delta)$  for all $x\in\setX$ and $\delta\in(0,\infty)$ and 
 \begin{align}
 \mleft(\int_0^{\infty}h(\delta)\,\mathrm d \delta \mright) \,\mathrm d\mu_X(x)  =   \int_0^{\infty}h(\delta)\,\mathrm d \delta<\infty 
 \end{align}
 with 
 \begin{align}
 h(\delta)=
 \begin{cases}
 e^{-\supc t\, \delta^\frac{m}{k}}&\quad\text{if  $\delta\in(0,\delta_0^kn^{{k}/{m}})$}\\
 e^{-\tilde\supc t\, \delta^\frac{m}{k}}&\quad\text{if  $\delta_0^kn^{{k}/{m}}\leq \delta < \infty$}
 \end{cases}
 \end{align}
 thanks to  \eqref{eq:firstZx}--\eqref{eq:item2c} together with \eqref{eq:gg1}--\eqref{eq:gg2}. 
Next, note that 
\begin{align}
\limsup_{n\to\infty}\opP\mleft[U_n(x)\geq \delta\mright]
&\leq
\lim_{n\to\infty} e^{-\supc\, \delta^\frac{m}{k} g\big(x,\, n^{-\frac{k}{m}}\delta\big)}\label{eq:detail1}\\
&=e^{-\supc \delta^\frac{m}{k}/f(x)}\quad\text{for all $x\in\setX$ and $\delta\in(0,\infty)$,}\label{eq:detail2}
\end{align}
where \eqref{eq:detail1} is by  \eqref{eq:firstZx}--\eqref{eq:item2b} and in   \eqref{eq:detail2}  we applied \eqref{eq:usemattila1}--\eqref{eq:usemattila}. Using \eqref{eq:detail1}--\eqref{eq:detail2} in \eqref{eq:usefubini2a}
 yields  
\begin{align}
\limsup_{n\to\infty} n^\frac{k}{m}V_n(X) 
&\leq  \int \mleft(\int_0^{\infty } e^{-\supc\, \delta^\frac{m}{k}/f(x)}   \,\mathrm d \delta\mright)\,\mathrm d\mu_X(x)\\
&=\opE\mleft[ (f(X)/b)^{\frac{k}{m}} \frac{k}{m}  \int_0^{\infty } e^{-s}  s^{\frac{k}{m}-1}  \,\mathrm ds  \mright]\label{eq:usegammafinal}\\
&=\opE\mleft[ f^{\frac{k}{m}}(X)\mright] \Gamma\mleft(1+\frac{k}{m} \mright)b^{-\frac{k}{m}} ,  
\end{align}
where in \eqref{eq:usegammafinal}  we changed variables according to $s=b\, \delta^{m/k}/f(x)$. 
This establishes   \eqref{limsupVUn}.  
\color{black}
Finally, we have 
\begin{align}
\Omega_{\alpha}(X)
&=\opE\mleft[ f^{\alpha}(X)\mright]\\
&\leq \mleft(\opE\mleft[ f(X)\mright]\mright)^{\alpha}\label{eq:JJJensen}\\
&=\mleft(\int  f(x) \,\mathrm d\mu_X(x)\mright)^{\alpha}\\
&=(\nu(\setX))^\alpha\\
&= 1\quad\text{for all $\alpha\in(0,1)$},
\end{align}
where \eqref{eq:JJJensen} follows from Jensen's inequality  \cite[Theorem 2.3]{peprto92} and strict concavity of  $t\to t^{\alpha}$ on $(0,\infty)$ for all $\alpha\in(0,1)$ with strict inequality in \eqref{eq:JJJensen}  unless 
 $f(x)=\opE\mleft[ f(X)\mright]$ for $\mu_X$-almost all $x\in\setX$ \cite[Example 6]{ba81}.  \color{black}
The statement in \eqref{eq:chainCC3} follows trivially from $\opE[f(X)]=1$ and the definition of the lower $k$-th quantization coefficient in \eqref{eq:Cinf}. 
\qed

\section{Proof of Lemma \ref{lem:dimbound}}\label{proof:dimbound}
We first prove Item \ref{eq:boundCcheck2}.  
Toward a contradiction, assume that $\liminf_{n\to\infty} n^\frac{k}{m}V_n(X)>0$ with $\underline{D}_k(X)< m$ and set $\delta=(m-\underline{D}_k(X))/2$.  
The definition of $\underline{D}_k(X)$ in \eqref{eq:Dinf} then implies that 
\begin{align}\label{eq:limsupconta}
\frac{ k \log (n)}{\log (1/V_{n}(X))}\leq m-\delta\quad\text{for $\infty$-many $n\in\naturals$.}
\end{align}
Straightforward algebraic manipulations show that  \eqref{eq:limsupconta} is equivalent to 
\begin{align}
n^\frac{k}{m} V_{n}(X)\leq n^{\frac{k}{m}-\frac{k}{m-\delta}}\quad\text{for $\infty$-many $n\in\naturals$},  
\end{align}
which in turn implies  the existence of an increasing  sequence $(n_j)_{j\in\naturals}$ satisfying 
\begin{align}
\liminf_{j\to\infty} n_j^\frac{k}{m}V_{n_j}(X)\leq \lim_{j\to\infty} {n_j}^{\frac{k}{m}-\frac{k}{m-\delta}}=0 
\end{align}
and  thereby contradicts the assumption $\liminf_{n\to\infty} n^\frac{k}{m}V_n(X)>0$.

Next, we prove  Item \ref{eq:boundCcheck}. 
  Toward a contradiction, assume that $\limsup_{n\to\infty} n^\frac{k}{m}V_n(X)<\infty$ with $\overline{D}_k(X)> m$ and set $\delta=(\overline{D}_k(X)- m)/2$. 
The definition of $\overline{D}_k(X)$ in \eqref{eq:Dsup} then implies that 
\begin{align}\label{eq:limsupcont}
\frac{ k \log (n)}{\log (1/V_{n}(X))}\geq m+\delta\quad\text{for $\infty$-many $n\in\naturals$.}
\end{align}
Straightforward algebraic manipulations reveal that  \eqref{eq:limsupcont} is equivalent to 
\begin{align}
n^\frac{k}{m} V_{n}(X)\geq n^{\frac{k}{m}-\frac{k}{m+\delta}}\quad\text{for $\infty$-many $n\in\naturals$},  
\end{align}
which in turn implies  the existence of an increasing  sequence $(n_j)_{j\in\naturals}$ satisfying 
\begin{align}
\limsup_{j\to\infty} n_j^\frac{k}{m}V_{n_j}(X)\geq \lim_{j\to\infty} {n_j}^{\frac{k}{m}-\frac{k}{m+\delta}}=\infty
\end{align}
and  thereby contradicts the assumption $\limsup_{n\to\infty} n^\frac{k}{m}V_n(X)<\infty$.


\section{Proof of Theorem \ref{cor:MSiid}}\label{cor:MSiidproof}
Lemma \ref{lem:simpilysub} combined with Lemma \ref{lem:transformsub} implies that every $\mu_{X_i}$ is $\sigma$-subregular  of dimension $m/(pk)$ with subregularity constants 
$\tilde c=\lVert\mathrm d \mu_{X_1}/\mathrm d \mu\rVert^{(\mu)}_{p/(p-1)} c^{1/p}$   and $\delta_0^k$.  
Fix $\ell\in\naturals$ arbitrarily and set $X^{(\ell)}=(X_1,\dots,X_\ell)$. 
It follows from Item \ref{item:prodsub} of Proposition~\ref{prp:subprod} applied to the measures $\mu_{X_1},\dots,\mu_{X_\ell}$ with  $\alpha_i=1/\ell$ for $i=1,\dots,\ell$ that $\mu_{X^{(\ell)}}:=\mu_{X_1}\otimes\dots\otimes\mu_{X_\ell}$  is 
$\sigma_{(\ell)}$-subregular of dimension $\ell m/(pk)$
with  subregularity constants 
$d_{(\ell)}^\ell$ 
and $\delta_{(\ell)}=\delta_0^k/\ell$. 
Theorem \ref{thm:singleshot} applied to $X^{(\ell)}$ (note that for $\mu=  \mu_{X^{(\ell)}}$ we have  $\Sigma_1(X^{(\ell)})=1$) therefore yields  \begin{align}\label{eq:VnDn}
V_{n}(X^{( \ell)})\geq L_{n}\big(X^{(\ell)}\big)\quad\text{for all $n\in \naturals$}
\end{align} 
with 
\begin{align}\label{eq:Ln2}L_{n}\big(X^{(\ell)}\big)= 
\min\bigg\{d_{(\ell)}^{-\frac{pk}{ m}}n^{-\frac{pk}{\ell m}}, \frac{\delta_0^k}{\ell}\bigg\}\frac{\ell m}{\ell m+pk}. 
\end{align} 
Now, set $L_n=L_{n}\big(X^{(\ell)}\big)$ and 
\begin{align}\label{eq:Nl2}
K_{( \ell)}= \mleft\lceil\mleft(\frac{\ell}{\delta_0^k}\mright)^{\frac{\ell m}{pk}}d_{(\ell)}^{-\ell} \mright   \rceil   
\end{align}
and recall the definition of  $D_{( \ell)}$  in \eqref{eq:Dl} according to 
\begin{equation}\label{eq:Dl2}
D_{(\ell)} 
= 
\frac{\delta_0^{k}}{\ell} \frac{\ell m}{\ell m+pk}. 
\end{equation} 
Evaluation of the minimum in  \eqref{eq:Ln2} yields  
\begin{align}\label{eq:useNll}
L_n=
d_{(\ell)}^{-\frac{pk}{ m}}n^{-\frac{pk}{\ell m}}\frac{\ell m}{\ell m +pk}\quad \text{for all $n\geq \lceil K_{(\ell)}\rceil$}
\end{align}
and  $L_{n} = D_{( \ell)}$  for all $n \leq \lfloor K_{( \ell)}\rfloor$.    
Next, assume that  there exist   $n\in\naturals$ and  $R\in(0,\infty)$ such that  $V_{\lfloor e^{ \ell R}\rfloor}(X^{( \ell)})<L_n$. Then, \eqref{eq:VnDn} implies that 
$L_{\lfloor e^{ \ell R}\rfloor}< L_n$, which, as a consequence of $L_n$ being  monotonically decreasing in $n$, 
  is possible for  $e^{ \ell R}\geq n+1$ only. 
We conclude that, for  every $D\in [L_{n+1},L_n)$ and $n\geq  \lceil K_{( \ell)}\rceil-1$, it holds that 
\begin{align}
 R_{( \ell)}(D) 
&=\inf\{R>0: V_{\lfloor e^{\ell R}\rfloor}(X^{(\ell)})\leq D\}\label{eq:usefinalss1}\\
&\geq \inf\{R>0: V_{\lfloor e^{\ell R}\rfloor}(X^{(\ell)})< L_n\}\\
&\geq \frac{1}{ \ell} \log (n+1)\label{eq:usefinalss}\\
& = 
 \frac{m}{ p k} \log \bigg(   \frac{\ell m}{(\ell m +p k)L_{n+1} } \bigg)
-  \log (d_{(\ell)})\label{eq:Dnnapply}
\\
& \geq
 \frac{m}{pk} \log \bigg(   \frac{\ell m}{(\ell m + pk)D} \bigg)
-\log (d_{(\ell)}),
\label{eq:Dnnapply2}
\end{align}
where  \eqref{eq:Dnnapply} is by \eqref{eq:useNll}. 
Finally,  as  $L_{\lfloor K_{( \ell)}\rfloor}=  D_{( \ell)}$  and $\lim_{n\to\infty}L_n=0$, we can conclude that for every $D< D_{( \ell)}$, 
there must exist an  $n\geq  \lceil K_{( \ell)}\rceil-1$ such that $D\in [L_{n+1},L_n)$. 
Hence,  
\eqref{eq:usefinalss1}--\eqref{eq:Dnnapply2} establishes $ R_{( \ell)} (D ) \geq  \widetilde R_{(\ell)}(D) $ for all $D\in(0, D_{( \ell)})$.

Next, note that $\widetilde R_{(\ell)}(D) $ can be written as 
\begin{align}
\widetilde R_{(\ell)}(D)
&= -\log \mleft(\lVert\mathrm d \mu_{X_1}/\mathrm d \mu\rVert^{(\mu)}_{p/(p-1)}\mright) - \frac{m}{pk}\log (D)
-\log \mleft(c^\frac{1}{p}\, \Gamma\mleft(1+\frac{m}{pk}\mright)\mright)\\
&\ \ +\frac{m}{pk}\mleft(    \frac{pk}{\ell m} \log\mleft(\Gamma\mleft(1+\frac{\ell m}{pk}\mright)\mright) -\log \mleft(\ell+\frac{pk}{m}\mright) \mright)  \\
&= -\log \mleft(\lVert\mathrm d \mu_{X_1}/\mathrm d \mu\rVert^{(\mu)}_{p/(p-1)}\mright)  - \frac{m}{pk}\log(D)
-\log \mleft(c^\frac{1}{p}\, \Gamma\mleft(1+\frac{m}{pk}\mright)\mright)\\
&\ \ + \frac{m}{pk}\log \mleft(\frac{m}{pk}\mright) +\frac{m}{pk}\psi\mleft(\frac{\ell m}{pk}\mright), 
\end{align}
where we set    
\begin{align}
\psi(x)&=\frac{\log(\Gamma(1+x))}{x} -  \log(1+x). 
\end{align} 
The properties of $\widetilde R_{(\ell)}(D)$ we wish to establish now follow from the fact that 
 $\psi(x)$ is strictly monotonically decreasing  on $(-1,\infty)$ with $\lim_{x\to \infty}\psi(x)=-1$ \cite[Theorem 1]{vovo02}.  
\qed

\section{Proof of Lemma \ref{lem:sphere}}\label{lem:sphereproof}
 We first show that $G(\alpha)$ is continuous and strictly monotonically  increasing   on $(0,1]$. 
To this end, we write  $G(\alpha)$ in the form   
\begin{align}
G(\alpha)=\frac{a^{(d-1)}(1)}{2 B_{\frac{d-1}{2},\frac{1}{2}}\mleft(1\mright) } g\mleft(\alpha\mright)  
\end{align}
with 
\begin{align}
g(\alpha)&= \frac{B_{\frac{d-1}{2},\frac{1}{2}} \mleft(\alpha \mright)}{\alpha^{\frac{d-1}{2}}}. 
\end{align}
Since 
\begin{align}\label{eq:derivative}
\frac{\mathrm d g(\alpha)}{\mathrm d \alpha} &=\frac{1}{\alpha}\mleft(\frac{1}{\sqrt{1-\alpha}}-\frac{d-1}{2}\frac{B_{\frac{d-1}{2},\frac{1}{2}}\mleft(\alpha\mright)}{\alpha^\frac{d-1}{2}}\mright)\\
&=\frac{1}{\alpha}\mleft(\frac{1}{\sqrt{1-\alpha}}-\frac{d-1}{2}g(\alpha)\mright)\quad\text{for all $\alpha\in(0,1)$,}
\end{align}
which is strictly positive on $(0,1)$ owing to 
\begin{align}
 g(\alpha) 
  &< \frac{1}{\alpha^\frac{d-1}{2}\sqrt{1-\alpha}}\int_0^\alpha u^{\frac{d-1}{2}-1}\,\mathrm d u\label{eq:bound1g}\\
  &=\frac{2}{(d-1)\sqrt{1-\alpha}}\quad\text{for all $\alpha\in(0,1)$,}\label{eq:bound2g}
\end{align}
it follows that $G(\alpha)$ is  continuous and strictly monotonically increasing  on $(0,1)$. 
Since 
\begin{align}
\abs{B_{\frac{d-1}{2},\frac{1}{2}}\mleft(1\mright)- B_{\frac{d-1}{2},\frac{1}{2}}\mleft(1-\delta\mright)}
&=\int_{1-\delta} ^1 u^{\frac{d-1}{2}-1} (1-u)^{-1/2}\,\mathrm d u\\
&\leq (1-\delta)^{-1/2} \int_{1-\delta} ^1  (1-u)^{-1/2}\,\mathrm d u \\
&=  2 \sqrt{\frac{\delta}{1-\delta}}\quad\text{for all $\delta\in (0,1)$,}
\end{align}
the function $G(\alpha)$ is continuous at $\alpha=1$ as well so that  monotonicity   on $(0,1)$ implies monotonicity on $(0,1]$. 
The limit in 
\eqref{eq:limitGa} follows by application of L'H\^opital's rule together with  $B_{(d-1)/2,1/2}(1)=\Gamma((d-1)/2) \Gamma(1/2)/\Gamma(d/2)$ \cite[Theorem 1.1.4]{anasro99}.

We next establish the family of subregularity conditions in Item \ref{itemsubR}. 
To this end, fix $\delta_0\in(0, r]$ arbitrarily. 
In  light of  Item \ref{item:subregularity} 
in Lemma \ref{lem:extsub}, 
it is sufficient to  show that 
\begin{align}
\mu\mleft(\setB_{\rho^{1/2}}(y,\delta)\mright)
&\leq   c_{\delta_0}  \delta^{d-1}\quad\text{for all $y\in\setY$ and $\delta\in(0,\delta_0)$}  \label{eq:subsphere2AA}
\end{align}
with $c_{\delta_0}$ as 
specified in the statement. 
Now, fix   $y\in\setY$ and  $\delta\in(0, \delta_0)$ arbitrarily and note that 
\begin{align}
\lVert z-y\rVert_2^2
&= \lVert y\rVert_2^2+r^2-2 \tp{y}z\quad\text{for all $z\in\setS^{d-1}(r)$}\label{eq:11p}
\end{align} 
yields
\begin{align}\label{eq:Cap}
\setB_{\rho^{1/2}}(y,\delta)&=\mleft\{ z\in \setS^{d-1}(r):\tp{y}z > \frac{\lVert y\rVert_2^2 +r^2-\delta^2}{2} \mright\}.
\end{align}

Next, consider the one-dimensional subspace $\setA(y)=\{a y: a\in\reals\}$.   
We  have 
\begin{align}
\operatorname{dist}(z,\setA(y))
&:= \min \{\lVert z- ay \rVert_2:a\in\reals\}\label{eq:useproj0}\\
&= \sqrt{\lVert z\rVert^2_2 -\frac{(\tp{y}z)^2}{\lVert y\rVert_2^2}}\label{eq:useproj}\\
&\leq \sqrt{r^2-\frac{(r^2 -\delta^2+ \lVert y\rVert_2^2)^2}{4\lVert y\rVert_2^2}}\quad\text{for all $z\in \setB_{\rho^{1/2}}(y,\delta)$,}\label{eq:usexybound}
\end{align} 
where \eqref{eq:useproj} follows from the Hilbert projection theorem \cite[Theorem 4.11]{ru87} on $\reals^d$ and in \eqref{eq:usexybound} we used \eqref{eq:Cap} together with  $\delta<\delta_0\leq r$.
We can now write \eqref{eq:useproj0}--\eqref{eq:usexybound} in the form 
\begin{align}
\operatorname{dist}(z,\setA(y))
&\leq \sqrt{r^2-f^2(\lVert y\rVert_2)}\quad\text{for all $z\in \setB_{\rho^{1/2}}(y,\delta)$}
\end{align} 
with 
\begin{align}
f(s)=\frac{r^2 -\delta^2+ s^2}{2s}. 
\end{align}
As $\delta<\delta_0\leq r$, the function $f(s)$ is strictly convex on $(0,\infty)$ and, therefore, minimized for $s_0=\sqrt{r^2-\delta^2}$. 
We can  thus conclude that 
\begin{align}
\operatorname{dist}(z,\setA(y))
&\leq \sqrt{r^2-f^2\mleft(\sqrt{r^2-\delta^2}\mright)}\label{eq:bound1SC}\\
&=\delta \quad\text{for all $z\in \setB_{\rho^{1/2}}(y,\delta)$.}\label{eq:bound2SC}
\end{align}
Now, \eqref{eq:bound1SC}--\eqref{eq:bound2SC}  implies that the radius of the base of the hyperspherical cap $ \setB_{\rho^{1/2}}(y,\delta)$ is no larger than $\delta$. 
We conclude that 
\begin{align}
\mu\mleft(\setB_{\rho^{1/2}}(y,\delta)\mright) 
&\leq \frac{1}{2}  I_{\frac{d-1}{2},\frac{1}{2}}\mleft(\frac{\delta^2}{r^2}\mright) \label{eq:SCbound}\\
&=  \frac{G(\delta^2/r^2)}{a^{(d-1)}(r)} \delta^{d-1}\\
&\leq   \frac{G(\delta_0^2/r^2)}{a^{(d-1)}(r)}  \delta^{d-1},\label{eq:upperb2} 
\end{align}
where \eqref{eq:SCbound} is by     the hyperspherical cap formula \cite[Equation (1)]{li11} 
with  $\delta$ for the radius of the base of the cap
and in \eqref{eq:upperb2} we made use of the   monotonicity of $G(\alpha)$.  This yields  the upper bound in 
 \eqref{eq:subsphere2AA} as  $y\in\setY$ and $\delta\in(0,\delta_0)$ were assumed to be  arbitrary.

Next, we   establish the family of superregularity conditions in Item \ref{itemsubS}. 
To this end, fix $\delta_0\in(0, \sqrt{2}r]$ arbitrarily. 
Again, in  light of  Item \ref{item:superregularity} 
in Lemma \ref{lem:extsub}, 
it  suffices  to  show that 
\begin{align}
 \mu\mleft(\widetilde\setB_{ \rho^{1/2}}(x,\delta)\mright)  \geq b_{\delta_0}\delta^{d-1} \quad\text{for all $x\in\setS^{d-1}(r)$ and $\delta\in(0,\delta_0)$}
  \label{eq:subsphere22}
\end{align}
with $b_{\delta_0}$ as 
specified in the statement. 
To this end, fix   $x\in  \setS^{d-1}(r)$ and  $\delta\in(0, \delta_0)$ arbitrarily. 
As before, consider the one-dimensional subspace $\setA(x)=\{a x: a\in\reals\}$.   
We have  
\begin{align}
\operatorname{dist}(z,\setA(x))
&:= \min \{\lVert z- ax \rVert_2:a\in\reals\}\label{eq:useproj02}\\
&= \sqrt{\lVert z\rVert^2_2 -\frac{(\tp{x}z)^2}{\lVert x\rVert_2^2}}\label{eq:useproj2}\\
&= \sqrt{r^2 -\frac{(\tp{x}z)^2}{r^2}}\quad\text{for all $z\in \setS^{d-1}(r)$,}\label{eq:useproj3}
\end{align} 
where again \eqref{eq:useproj2} follows from the Hilbert projection theorem 
 on $\reals^d$. 
Now, \eqref{eq:useproj02}--\eqref{eq:useproj3} is equivalent to 
\begin{align}
\abs{\tp{x}z}=r\sqrt{r^2-\operatorname{dist}^2(z,\setA(x))}\quad\text{for all $z\in \setS^{d-1}(r)$.}\label{eq:useproj3A}
\end{align}
Using \eqref{eq:useproj3A} and $\delta<\delta_0\leq \sqrt{2}r$
in \eqref{eq:Cap}, we obtain  
\begin{align}\label{eq:Cap3}
 \setB_{ \rho^{1/2}}(x,\delta)
&=\mleft\{ z\in \setS^{d-1}(r): \operatorname{dist}(z,\setA(x)) < \delta \sqrt{1-\mleft(\frac{\delta}{2r}\mright)^2} \ \text{and}\ \tp{x}z >0\mright\}.
\end{align}
Now, \eqref{eq:Cap3} implies that the radius of the base of the hyperspherical cap $\setB_{ \rho^{1/2}}(x,\delta)$ equals $\delta \sqrt{1-(\delta/(2r))^2}$. 
We conclude that 
\begin{align}
\mu\mleft(\widetilde\setB_{ \rho^{1/2}}(x,\delta)\mright) &= \mu\mleft(\setB_{ \rho^{1/2}}(x,\delta)\mright)\label{eq:SCbound2a}\\
&=\frac{1}{2}  I_{\frac{d-1}{2},\frac{1}{2}}\mleft(h\mleft(\frac{\delta^2}{r^2}\mright)\mright)\label{eq:SCbound2A}\\
%
&=\frac{\Gamma\mleft(\frac{d}{2}\mright)}{2\sqrt{\pi}\,\Gamma\mleft(\frac{d-1}{2}\mright)} B_{\frac{d-1}{2},\frac{1}{2}}\mleft( h\mleft(\frac{\delta^2}{r^2}\mright) \mright)\\
&\geq \frac{\Gamma\mleft(\frac{d}{2}\mright)}{2\sqrt{\pi}\,\Gamma\mleft(\frac{d+1}{2}\mright)} \mleft( h\mleft(\frac{\delta^2}{r^2}\mright) \mright)^{\frac{d-1}{2}}  \label{eq:uselowerbeta}\\
&\geq \frac{\Gamma\mleft(\frac{d}{2}\mright)}{2\sqrt{\pi}\,\Gamma\mleft(\frac{d+1}{2}\mright)r^{d-1}}\mleft(1-\frac{\delta_0^2}{4r^2}\mright)^{\frac{d-1}{2}} \delta^{d-1}\\
&= \frac{\pi^{\frac{d-1}{2}}}{\Gamma\mleft(\frac{d+1}{2}\mright)a^{(d-1)}(r)}\mleft(1-\frac{\delta_0^2}{4r^2}\mright)^\frac{d-1}{2}\delta^{d-1}
\end{align} 
where \eqref{eq:SCbound2a} follows from 
\begin{align}
\widetilde\setB_{ \rho^{1/2}}(x,\delta)\cap \setS^{d-1}(r)
=\{z\in \setS^{d-1}(r) : \lVert x-z\rVert_2<\delta\} 
= \setB_{ \rho^{1/2}}(x,\delta),  
\end{align}  
\eqref{eq:SCbound2A} is by   the hyperspherical cap formula \cite[Equation (1)]{li11} 
with  $\delta \sqrt{1-(\delta/(2r))^2}$ for the radius of the base of the cap and we set  $h(t)=t(1-t/4)$, 
and 
  in \eqref{eq:uselowerbeta} we used 
\begin{align}\label{eq:lowerboundbeta}
 B_{\frac{d-1}{2},\frac{1}{2}}(\alpha) \geq \int_0^\alpha u^{\frac{d-1}{2}-1}\,\mathrm d u= \frac{2}{d-1} \alpha^{\frac{d-1}{2}}  
\quad\text{for all $\alpha\in(0,1]$.}
\end{align}
The argument is concluded by noting that $x\in  \setS^{d-1}(r)$ and  $\delta\in(0, \delta_0)$ were assumed to be arbitrary.

 Finally, the limit for $c_{\delta_0}$  follows from \eqref{eq:limitGa} and the limit for  $b_{\delta_0}$ is trivial. 
\qed



\section{A Layer Cake Argument}\label{sec:layercake}
An important tool in   lower/upper-bounding integrals using  sub/super-regularity of the underlying integration measure is based on the so-called layer cake argument, which is an immediate consequence of 
Tonelli's theorem for characteristic functions \cite[Theorem 10.9]{ba95}. The formal statement is as follows. 

\begin{lem}\cite[Exercise 1.7.2]{du10}\label{lem:fubI}
Consider a $\sigma$-finite measure space $(\setX,\colX,\mu)$. If $f\colon\setX\to\reals$ is  measurable and  $f(\setX)\subseteq [0,a]$ with $a\in (0,\infty]$, then  
\begin{align}
\int f(x)\,\mathrm d\mu(x)
&=\int_0^a \mu\mleft(\{x:f(x)\geq t\}\mright)\,\mathrm d t \label{eq:geq}\\
&=\int_0^a \mu\mleft(\{x:f(x)> t\}\mright)\,\mathrm d t.\label{eq:g}
\end{align}
\begin{proof}
To establish \eqref{eq:geq}, we note that 
\begin{align}
\int_0^a \mu\mleft(\{x:f(x)\geq t\}\mright)\,\mathrm d t\label{eq:fubapplemA}
&=\int_0^\infty \mu\mleft(\{x:f(x)\geq t\}\mright)\,\mathrm d t\\
&=\int \mleft(\int_0^{f(x)} \,\mathrm dt\mright)\,\mathrm d\mu(x)\label{eq:fubapplemB}\\
&=\int f(x)\,\mathrm d\mu(x),\label{eq:fubapplem}
\end{align}
where \eqref{eq:fubapplemA} holds because $f(\setX)\subseteq [0,a]$  
and  
\eqref{eq:fubapplemB} follows from  Tonelli's theorem for characteristic functions \cite[Theorem 10.9]{ba95}.  
Finally, \eqref{eq:g} is obtained  by repeating the steps in \eqref{eq:fubapplemA}--\eqref{eq:fubapplem} with $f(x)\geq t$ replaced by $f(x)> t$.
\end{proof}
\end{lem}


\section*{Acknowledgments}
The authors thank the anonymous reviewers for their valuable suggestions. 

\section*{Data Availability Statement}
No new data were generated or analysed in support of this review.

\section*{Statements and Declarations}
\begin{itemize}
\item G. Koliander was supported in part by  the WWTF project MA16-053. 
\item A weaker version of Theorem \ref{thm:new} was presented without proof at the 2018 IEEE International Symposium on Information Theory (ISIT) \cite{rikobo18}. 
\end{itemize}

\bibliographystyle{plain}
\bibliography{references}

\end{document}